\documentclass[10pt,reqno]{amsart}
\usepackage{color}
\usepackage{enumerate}
\usepackage{amsmath,amssymb}
\usepackage{mathrsfs}

\usepackage{hyperref}

\hypersetup{
    colorlinks=true,
    linkcolor=blue,
    filecolor=magenta,      
    urlcolor=cyan,
}

\newtheorem{theorem}{Theorem}[section]
\newtheorem{lemma}[theorem]{Lemma}
\newtheorem{corollary}[theorem]{Corollary}
\newtheorem{proposition}[theorem]{Proposition}
\newtheorem{open}[theorem]{Open problem}

\theoremstyle{remark}
\newtheorem{remark}[theorem]{Remark}

\theoremstyle{definition}
\newtheorem{assumption}[theorem]{Assumption}
\newtheorem{example}[theorem]{Example}
\newtheorem{definition}[theorem]{Definition}

\newcommand\cbrk{\text{$]$\kern-.15em$]$}}
\newcommand\opar{\text{\,\raise.2ex\hbox{${\scriptstyle|}$}\kern-.34em$($}}
\newcommand\cpar{\text{$)$\kern-.34em\raise.2ex\hbox{${\scriptstyle |}$}}\,}

\def\XXint#1#2#3{{\setbox0=\hbox{$#1{#2#3}{\int}$}
\vcenter{\hbox{$#2#3$}}\kern-.5\wd0}}

\newcommand\bE{\mathbb{E}}

\newcommand\fN{\mathbf{N}}
\newcommand\fR{\mathbf{R}}

\newcommand\cF{\mathcal{F}}

\newcommand\rF{\mathscr{F}}

\newcommand{\mysection}[1]{\section{#1}
\setcounter{equation}{0}}

\begin{document}

\title[Degenerate Cauchy problems with singular coefficients]
{Well-posedness of second-order evolution equations with non-integrable and degenerate coefficients in weighted $\mathrm{L}_p$-spaces}

\author{Ildoo Kim}
\address{Department of mathematics, Korea university, 1 anam-dong
sungbuk-gu, Seoul, south Korea 136-701}
\email{waldoo@korea.ac.kr}
\thanks{This work was supported by the National Research Foundation of Korea(NRF) grant
funded by the Korea government(MSIT) (RS-2025-16065358)}

\author{Kyeong-Hun Kim}
\address{Department of mathematics, Korea university, 1 anam-dong
sungbuk-gu, Seoul, south Korea 136-701}
\email{kyeonghun@korea.ac.kr}
\thanks{This work was supported by the National Research Foundation of Korea(NRF) grant
funded by the Korea government(MSIT) (RS-2025-00556160)}



\begin{abstract}
We study the Cauchy problem for inhomogeneous evolution equations with time-dependent, potentially degenerate, and unbounded coefficients. 
A key feature of our work is allowing the principal coefficients to undergo arbitrary blow-up at both the initial and terminal times.
\end{abstract}

\maketitle

\mysection{introduction}

Singular coefficients in partial differential equations are typically characterized by discontinuity, unboundedness, or degeneracy. 
The investigation of second-order PDEs featuring such coefficients is a classical and extensive area of mathematical research, tracing its origins to the seminal works of Picone, Tricomi, Keldyš, Sobolev, and Fichera. 
For a comprehensive treatment of these classical results, Ole\v{i}nik and Radkevi\v{c}'s literature (\cite{OR 1973}) remains a foundational reference.

The field continues to see vigorous activity, with numerous authors exploring various classes of singularities within a framework of weighted  spaces (see, e.g., \cite{DP 2021, KK 2023, KID 2024, DPT 2024, DJV 2024, AFV 2024, DP 2023, MNS 2023, DR 2024, DPT 2023} and the references therein).

Despite this rich literature, a prevailing limitation in contemporary studies is the standard assumption of local integrability. 
Even when coefficients exhibit blow-up behavior near the boundaries or the initial time, they are generally required to remain locally integrable. 
Intuitively, severe growth—such as exponential blow-up—is incompatible with non-zero boundary or initial data. 
However, a crucial insight is that well-posedness can still be recovered even if the coefficients undergo rapid, unrestricted growth near the initial time, provided that the initial condition is strictly zero. 
This observation provides the central motivation for the present paper: we aim to establish that the well-posedness of second-order evolution equations can be rigorously affirmed even in the presence of strongly non-integrable coefficients.

To formalize our theory, we consider the following Cauchy problem:\begin{align}
										\notag
&u_t(t,x) = a^{ij}(t) u_{x^i x^j}(t,x) + b^i(t) u_{x^i}(t,x) + c(t) u(t,x) + f(t,x), \\
										\label{main eqn 1}
&u(0,x) = 0, \quad (t,x) \in (0,T) \times \mathbf{R}^d,
\end{align}
where the spatial dimension $d \in \mathbf{N}$ and the terminal time $T > 0$ are fixed constants. 
Throughout this paper, we adopt the Einstein summation convention, where repeated indices imply summation over $i,j = 1, \dots, d$.
The coefficients $a^{ij}(t)$, $b^i(t)$, and $c(t)$ are real-valued Lebesgue measurable functions defined on the interval $(0,T)$. 
We impose a degenerate ellipticity condition on the principal part, \textit{i.e.}
$$
a^{ij}(t) \xi^i \xi^j \geq \lambda(t)|\xi|^2 \quad \text{for all } t \in (0,T) \text{ and } \xi \in \mathbf{R}^d,
$$
where $\lambda(t) \geq 0$ is a Lebesgue measurable function. 
No additional regularity conditions are imposed, nor is $\lambda(t)$ assumed to be strictly positive. 
Consequently, our framework naturally accommodates cases where the equation becomes irregular or degenerate on sets of positive measure.
Furthermore, we merely assume that the coefficients are locally integrable inside the open interval $(0,T)$.
Specifically, for any $0 < t_1 < t_2 < T$, we have
$$
\int_{t_1}^{t_2} \left( |a^{ij}(s)| + |b^i(s)| + |c(s)| \right) \mathrm{d}s < \infty.
$$
Crucially, this assumption permits the coefficients to exhibit extreme singularities—including exponential blow-up—at the boundaries $t=0$ and $t=T$.
Under these hypotheses, we demonstrate that for a given inhomogeneous source term $f$, the Cauchy problem \eqref{main eqn 1} admits a unique weak solution $u$. 
If we further assume that the zero-order coefficient $c(t)$ is integrable near the initial time, namely $\int_0^t |c(s)| \mathrm{d}s < \infty$ for all $t \in (0,T)$, we can establish the following maximal regularity estimate for any $p,q \in (1,\infty)$ and $T' \in (0,T)$:
\begin{align*}
&\int_0^{T'} \left( \int_{\mathbf{R}^d} |u_{xx}(t,x)|^p \mathrm{d}x \right)^{q/p} \mathrm{e}^{-q\int_0^t c(s)\mathrm{d}s} w(\alpha(t)) \lambda(t)\mathrm{d}t \\
&\quad \lesssim_{d,p,q,w} \int_0^{T'} \left( \int_{\mathbf{R}^d} |f(t,x)|^p \mathrm{d}x \right)^{q/p} \mathrm{e}^{-q\int_0^t c(s)\mathrm{d}s} w(\alpha(t)) (\lambda(t))^{1-q}\mathrm{d}t,\end{align*}
where $w$ represents a one-dimensional Muckenhoupt weight (\textit{cf.} Definition \ref{def weight}) and $\alpha(t) = \int_0^t \lambda(s) \mathrm{d}s$.

Physically, allowing for such strong singularities significantly broadens the modeling capabilities of the equation. 
Singularities in $a^{ij}(t)$ can describe diffusion processes with extreme temporal fluctuations, representing phases of instantaneous or highly rapid propagation. 
Similarly, a blow-up in the advection term $b^i(t)$ captures intervals where transport completely dominates the system's dynamics. 
Finally, an unbounded $c(t)$ can model explosive reaction kinetics, leading to sudden, unbounded growth or decay. 
Because our framework accommodates singularities precisely at the initial time $t=0$, it is exceptionally well-suited for capturing the initial bursts of rapidly evolving phenomena, such as explosive chemical reactions or heat conduction with extreme initial conductivity variations.

The presence of unbounded coefficients introduces substantial difficulties in establishing the existence, uniqueness, and stability of solutions, thereby demanding rigorous mathematical treatment. Analyzing parabolic PDEs with such singular coefficients is inherently challenging and typically relies on specialized methodologies, such as the introduction of tailored weight functions to control the singularities. 
Historically, however, these approaches impose a critical restriction: the coefficients are prohibited from blowing up faster than the reciprocal of the distance to the boundary (cf. \cite{CDK 2015}).
We emphasize that our framework fundamentally relaxes this constraint. Specifically, to guarantee the existence and uniqueness of a weak solution, our coefficients $a^{ij}(t)$ and $b^i(t)$ ($i,j = 1, \dots, d$) are entirely exempt from the standard bound:
$$
|a^{ij}(t)| + |b^i(t)| \lesssim \frac{1}{t} + \frac{1}{T-t}, \quad t \in (0,T).
$$
Likewise, the zero-order coefficient $c(t)$ is not required to satisfy the analogous condition:
$$
|c(t)| \lesssim \frac{1}{t} + \frac{1}{T-t}, \quad t \in (0,T).
$$
In fact, we impose absolutely no restrictions on the asymptotic behavior of these coefficients as they approach the initial and terminal times. 
As previously noted, they are permitted to exhibit singularities that grow faster than any exponential function. 
To the best of our knowledge, the present study is the first to establish well-posedness while allowing the coefficients' blow-up rates to strictly exceed the reciprocal of the distance to the temporal boundaries.

A natural extension of equation \eqref{main eqn 1} would involve coefficients that depend on both time and space, formulated as:
$$
u_t(t,x) = a^{ij}(t,x) u_{x^i x^j}(t,x) + b^i(t,x) u_{x^i}(t,x) + c(t,x) u(t,x) + f(t,x),
$$
$$u(0,x) = 0, \quad (t,x) \in (0,T) \times \mathbf{R}^d.
$$
However, establishing well-posedness for such spatially dependent equations in $L_p$-spaces is highly non-trivial. For instance, even for a simplified spatial operator of the form:
\begin{align*}
a^{ij}(t,x) u_{x^i x^j}(t,x) + b^i(t,x) u_{x^i}(t,x) + c(t,x) u(t,x) \\
= \Delta u(t,x) - d \frac{x^i}{|x|^2} u_{x^i}(t,x) - \lambda u(t,x),
\end{align*}
where $\lambda \in (0,\infty)$, well-posedness in $L_p$-spaces breaks down (as demonstrated in \cite[Example 2.1]{Krylov 2023}).
This implies that generalizing our framework to include spatially dependent coefficients would require imposing stringent structural conditions on their spatial behavior. 
Because our primary objective is to investigate the unique effects of strongly non-integrable temporal singularities, we leave this complex generalization outside the scope of the current work. 
Identifying the precise spatial assumptions necessary to accommodate such coefficients remains a promising avenue for future research.

While extending the present framework to spatially dependent coefficients $a^{ij}(t,x)$ is highly desirable, $L_p$-well-posedness generally breaks down without delicate structural control on $x$, as demonstrated by Krylov's counterexample. 
Nevertheless, our analysis strongly suggests that if the temporal singularity at $t=0$ is isolated via our weighted space approach, one may combine this with spatial conditions—such as Discontinuity of Mean Oscillation (DMO) or Vanishing Mean Oscillation (VMO) conditions on $x$—to recover well-posedness. 
Formulating the precise interplay between non-integrable temporal singularities and spatial regularity remains a crucial open direction.

The remainder of this paper is organized as follows. Section \ref{section main} introduces the principal framework and states our main results. 
In Section \ref{ito calculus}, we establish the existence of a solution to equation \eqref{main eqn 1} by employing It\^o's calculus. 
The uniqueness of this solution is then demonstrated through two independent methods: a Fourier transform approach in Section \ref{fourier method}, and an energy inequality technique in Section \ref{elementary section}. Section \ref{pf main thm} contains the complete rigorous proofs of our main theorems. 
We conclude in the final section by outlining several open problems prompted by our findings.

We conclude the introduction by establishing the primary mathematical notations and conventions utilized throughout this article.
\begin{itemize}
\item The sets of natural numbers and integers are denoted by $\mathbf{N}$ and $\mathbf{Z}$, respectively. We write $\mathbf{R}^d$ for the $d$-dimensional Euclidean space with points $x = (x^1, x^2, \ldots, x^d)$. For $i = 1, \dots, d$, multi-indices $\alpha = (\alpha_1, \dots, \alpha_d)$ with $\alpha_i \in \{0, 1, 2, \dots\}$, and a sufficiently smooth function $u(x)$, we define the partial derivatives as $u_{x^i} = \frac{\partial u}{\partial x^i} = D_i u$ and $D^\alpha u = D_1^{\alpha_1} \cdots D_d^{\alpha_d} u$. The gradient vector is denoted by $u_x = (u_{x^1}, \dots, u_{x^d})$.

\item The space of infinitely differentiable functions on $\mathbf{R}^d$ is denoted by $C^\infty(\mathbf{R}^d)$. Its subspace consisting of functions with compact support is denoted by $C_c^\infty(\mathbf{R}^d)$.

\item For a domain $\mathcal{O} \subset \mathbf{R}^d$ and a normed space $F$, $C(\mathcal{O}; F)$ represents the space of all $F$-valued continuous functions $u$ on $\mathcal{O}$ equipped with the norm $\|u\|_{C} := \sup_{x \in \mathcal{O}} \|u(x)\|_F < \infty$.

\item Let $(X, \mathcal{M}, \mu)$ be a measure space, $F$ be a normed space, and $p \in [1, \infty)$. We denote by $L_p(X, \mathcal{M}, \mu; F)$ the space of all $F$-valued, $\mathcal{M}^\mu$-measurable functions $u$ such that
$$
\|u\|_{L_p(X, \mathcal{M}, \mu; F)} := \left( \int_X \|u(x)\|_F^p \, \mu(\mathrm{d}x) \right)^{1/p} < \infty,
$$
where $\mathcal{M}^\mu$ is the completion of the $\sigma$-algebra $\mathcal{M}$ with respect to $\mu$. 
For $p = \infty$, $u \in L_\infty(X, \mathcal{M}, \mu; F)$ if its essential supremum is finite:
$$
\|u\|_{L_\infty(X, \mathcal{M}, \mu; F)} := \inf \left\{ \nu \geq 0 : \mu(\{ x : \|u(x)\|_F > \nu \}) = 0 \right\} < \infty.
$$
When the measure space and $\sigma$-algebra are clear from the context, they are omitted. 
Unless otherwise specified, we assume the completion of the Borel $\sigma$-algebra generated by the underlying topology.

\item We specifically utilize the completion of the Borel $\sigma$-algebra induced by the standard Euclidean norm, which is called the Lebesgue measurable sets. 
For a Lebesgue-measurable set $\mathcal{O} \subset \mathbf{R}^d$, $|\mathcal{O}|$ denotes its Lebesgue measure. In the context of Euclidean spaces, we use the term ``measurable" interchangeably with ``Lebesgue-measurable."

\item Let $\mathcal{O} \subset \mathbf{R}^d$ and $\mathcal{N} \subset \mathcal{O}$. A function $f$ defined on $\mathcal{O} \setminus \mathcal{N}$ is considered measurable and defined almost everywhere on $\mathcal{O}$ if $\mathcal{N}$ is a null set, and there exists a measurable function $g$ on $\mathcal{O}$ such that $f(x) = g(x)$ for all $x \in \mathcal{O} \setminus \mathcal{N}$. Here, $\mathcal{N}$ is a null set if it is contained in a Borel set of Lebesgue measure zero.

\item For a $d \times d$ matrix $A$, we use the component notation $A = (a^{ij})$ for $i,j = 1, \dots, d$, where $a^{ij}$ represents the entry in the $i$-th row and $j$-th column.

\item 
The notation $\alpha \lesssim \beta$ indicates that there exists a positive constant $N$ such that $\alpha \leq N \beta$. The precise value of $N$ may vary from line to line, or even within the same line. If we write $\alpha \lesssim_{a,b,c} \beta$, it signifies that the implicit constant $N$ depends exclusively on the parameters $a$, $b$, and $c$. The dependencies of such constants are consistently specified in the statements of our theorems, lemmas, and corollaries.
\end{itemize}

\mysection{Setting and main results}
									\label{section main}

To lay the mathematical groundwork, we direct readers seeking foundational context to Krylov's comprehensive text \cite{Krylov 2008}. 

The primary challenges—and consequently, the significance—of our results arise directly from the singularities within the equation's coefficients. To emphasize the novelty of our results, we begin by detailing our core assumptions regarding these terms. Throughout this work, we assume that all coefficients $a^{ij}(t)$, $b^i(t)$, and $c(t)$ (for $i,j = 1, \dots, d$) are real-valued, measurable functions defined on the open interval $(0,T)$.

\begin{assumption}[Degenerate ellipticity and local integrability]
						\label{main as 1}

The principal coefficients $a^{ij}(t)$ satisfy the following conditions:
\begin{enumerate}[(i)]
\item{\bf Degenerate ellipticity:} The coefficient matrix $(a^{ij}(t))$ is non-negative definite, meaning
\begin{align}
										\label{weak elliptic}
a^{ij}(t) \xi^i\xi^j  \geq 0 \quad \forall t \in (0,T)~\text{and}~ \forall \xi=(\xi^1,\ldots,\xi^d) \in \fR^d.
\end{align}
\item{\bf Inside Local integrability:} The components $a^{ij}(t)$ are locally integrable inside $(0,T)$, such that
\begin{align*}
 \int_s^t  \max_{i,j} |a^{ij}(r)| \mathrm{d}r < \infty \quad \text{for all $0<s<t<T$},
\end{align*}
where
\begin{align*}
\max_{i,j} |a^{ij}(t)|
=\max\{ |a^{ij}(t)| : i,j \in \{1,\ldots,d\} \}.
\end{align*}
\end{enumerate}
\end{assumption}

\begin{assumption}
						\label{main as 2}

The first-order coefficients $b^i(t)$ $(i=1,\ldots,d)$ are locally integrable inside $(0,T)$. Specifically,
\begin{align*}
 \int_s^t \max_i|b^i(r)|\mathrm{d}r < \infty \quad \text{for all $0<s\leq t<T$}.
\end{align*}
\end{assumption}

\begin{assumption}
						\label{main as 3}

The zero-order coefficient $c(t)$ is locally integrable inside $(0,T)$, satisfying
\begin{align*}
\int_s^t |c(r)| \mathrm{d}r < \infty \quad \text{for all $0<s\leq t<T$}.
\end{align*}
\end{assumption}

Next, we define a specific class of locally integrable functions on $[0,T) \times \mathbf{R}^d$, which provides the foundation for the rigorous formulation of our weak solutions. 
A real-valued measurable function $f$, defined almost everywhere on $[0,T) \times \mathbf{R}^d$, is said to be locally integrable if
$$
\int_0^t \int_{|x| < r} |f(s,x)| \mathrm{d}x \mathrm{d}s < \infty
$$
for all $t \in (0,T)$ and $r > 0$.
Note that $f(t,x)$ need not be evaluated precisely at $t=0$. 
However, we explicitly utilize the half-open interval $[0,T)$ rather than the open interval $(0,T)$ to emphasize that the time integration extends down to zero. 
This prevents any ambiguity with the preceding assumptions, where the coefficients were only required to be locally integrable strictly inside $(0,T)$.

\begin{definition}[Weak solution]
							\label{def weak sol}
Let $u$ and $f$ be locally integrable functions on $[0,T) \times \mathbf{R}^d$. 
We say that $u$ is a weak solution to the Cauchy problem \eqref{main eqn 1} if, for every test function $\varphi \in C_c^\infty(\mathbf{R}^d)$, the following integral identity holds for all $t \in (0,T)$:
\begin{align}
(u(t,\cdot), \varphi)_{L_2(\mathbf{R}^d)} 
										\notag
&= \int_0^t \left(u(s,\cdot), a^{ij}(s)\varphi_{x^i x^j} - b^i(s)\varphi_{x^i} + c(s)\varphi \right)_{L_2(\mathbf{R}^d)} \mathrm{d}s \\
										\label{20240909 70}
&\quad + \int_0^t (f(s,\cdot), \varphi)_{L_2(\mathbf{R}^d)} \mathrm{d}s, 
\end{align}
where $(\cdot, \cdot)_{L_2(\mathbf{R}^d)}$ denotes the standard spatial inner product in $L_2(\mathbf{R}^d)$, given by
$$
(u(t,\cdot), \varphi)_{L_2(\mathbf{R}^d)} := \int_{\mathbf{R}^d} u(t,x)\varphi(x) \mathrm{d}x.
$$
When necessary, we may refer to $u$ as a weak solution to \eqref{main eqn 1} with data $f$ to explicitly emphasize the associated inhomogeneous term.
\end{definition}

\begin{remark}

Since this paper avoids the use of general distribution-valued functions, our definition of a weak solution relies entirely on the standard $L_2(\mathbf{R}^d)$ inner product. 
As a result, the source term $f$ must be locally integrable on $[0,T) \times \mathbf{R}^d$ to ensure that the final integral in \eqref{20240909 70} is well-defined. 
Throughout the remainder of this work, it is implicitly assumed that any weak solution $u$ and its corresponding source data $f$ satisfy this local integrability requirement, without needing to restate it. 
Furthermore, this integral formulation remains robust even when the coefficients fail to be integrable near the initial time. 
This is because a solution $u$ with a zero initial condition decays rapidly enough to counteract the singularities of the coefficients.

\end{remark}

\begin{remark}
We now recall a standard PDE technique for simplifying the principal coefficients $a^{ij}(t)$. 
The degenerate ellipticity condition \eqref{weak elliptic} does not intrinsically require the leading coefficient matrix $A(t) := (a^{ij}(t))$ to be symmetric, meaning we cannot immediately apply spectral decomposition. 
To resolve this and facilitate element-wise matrix analysis, we replace $A(t)$ with its symmetric part. It is straightforward to verify that $u$ is a weak solution to \eqref{main eqn 1} if and only if it is a weak solution to the symmetrized equation:
$$
u_t(t,x) = \frac{a^{ij}(t) + a^{ji}(t)}{2} u_{x^i x^j}(t,x) + b^i(t)u_{x^i}(t,x) + c(t)u(t,x) + f(t,x),
$$
$$
u(0,x) = 0, \quad (t,x) \in (0,T) \times \mathbf{R}^d.
$$
Thus, we may assume without loss of generality that $A(t)$ is symmetric for all $t \in (0,T)$. 
We have rigorously verified that this standard convention remains valid even for the highly singular coefficients considered in our framework.
Assuming symmetry provides two fundamental analytical advantages:
\begin{itemize}
\item{\bf Square Root Matrix:} It guarantees the existence of a unique positive semi-definite square root matrix, denoted $\sqrt{A}(t)$, satisfying $\sqrt{A}(t)\sqrt{A}(t) = A(t)$.
\item{\bf Spectral Lower Bound:} It allows us to identify a non-negative function $\lambda(t)$ on $(0,T)$—termed a spectral lower-bound of $A(t)$—that satisfies the quadratic form inequality:
$$
\lambda(t)|\xi|^2 \leq a^{ij}(t) \xi^i \xi^j \quad \text{for all } t \in (0,T) \text{ and } \xi \in \mathbf{R}^d.
$$
\end{itemize}
Typically, $\lambda(t)$ is defined as the smallest eigenvalue of $A(t)$. 
However, a complication arises if this eigenvalue is not integrable near the initial time $t=0$ (i.e., if $\int_0^{T'} \lambda(t) \mathrm{d}t = \infty$ for some $T'$). 
We can circumvent this by introducing a truncated lower bound using a positive constant $M$. 
Setting $\lambda(t) \wedge M := \min\{\lambda(t), M\}$, the ellipticity bound still holds:
$$
(\lambda(t) \wedge M)|\xi|^2 \leq a^{ij}(t) \xi^i \xi^j \quad \text{for all } t \in (0,T) \text{ and } \xi \in \mathbf{R}^d.
$$
Crucially, this truncated function is locally integrable down to $t=0$, meaning:
$$
\int_0^{T'} (\lambda(t) \wedge M) \mathrm{d}t < \infty \quad \text{for all } T' \in (0,T).
$$
Consequently, under Assumption \ref{main as 1}, we may assume without loss of generality that $A(t)$ possesses a spectral lower-bound that is locally integrable on $[0,T)$. 
This integrability property plays a pivotal role in the proofs of Theorem \ref{main thm 5} and Theorem \ref{main thm 6}.
\end{remark}

To simplify our presentation throughout the remainder of the paper, we introduce the notation
$$
\mu_{a,b,c}(t) = \left(\max_{i,j}|a^{ij}(t)| + \max_i|b^i(t)| + |c(t)|\right).
$$
With this established, we now state our first main result, which establishes the existence of a weak solution to equation \eqref{main eqn 1}.
\begin{theorem}[Existence of a weak solution]
										\label{main thm 1}
Suppose $f$ is locally integrable on $[0,T) \times \fR^d$ and let $p \in [1,\infty]$. 
Under Assumptions \ref{main as 1}--\ref{main as 3}, if the following supplementary integrability conditions are satisfied:
\begin{enumerate}[(i)]
\item
$$
\int_s^t \mathrm{e}^{\int_s^r c(\rho) \mathrm{d}\rho}\max_{i,j}|a^{ij}(r)| \mathrm{d}r < \infty
$$
for all $0<s<t<T$,

\item 
$$
\int_0^t \mathrm{e}^{\int_s^t |c(r)|\mathrm{d}r} \|f(s,\cdot)\|_{\mathrm{L}_p(\fR^d)}\mathrm{d}s < \infty
$$for all $t \in (0,T)$,
\item
$$
\int_0^{T'}  \mu_{a,b,c}(t) \int_0^t \mathrm{e}^{\int_s^t c(r)\mathrm{d}r}\|f(s,\cdot)\|_{\mathrm{L}_p(\fR^d)}\mathrm{d}s \mathrm{d}t  <\infty
$$
for all $T' \in (0,T)$, 
\end{enumerate}
then there exists a weak solution $u$ to equation \eqref{main eqn 1}, which obeys
$$
\|u(t,\cdot)\|_{\mathrm{L}_p(\fR^d)} \leq \int_0^{t} \mathrm{e}^{\int_s^t c(r)\mathrm{d}r} \|f(s,\cdot)\|_{\mathrm{L}_p(\fR^d)}\mathrm{d}s
$$
as well as
$$
\int_0^{t} \mu_{a,b,c}(s) \|u(s,\cdot)\|_{\mathrm{L}_p(\fR^d)} \mathrm{d}s \leq \int_0^{t} \mu_{a,b,c}(s) \int_0^s \mathrm{e}^{\int_r^s c(\rho)\mathrm{d}\rho} \|f(r,\cdot)\|_{\mathrm{L}_p(\fR^d)}\mathrm{d}r \mathrm{d}s
$$
for any $t \in (0,T)$.
\end{theorem}
The proof of this theorem relies on techniques from stochastic calculus. To avoid overwhelming the reader with the extensive mathematical machinery of It\^o calculus upfront, we defer the rigorous proof to a later section (see Theorem \ref{stochastic weak existence}).

\begin{remark}
In Theorem \ref{main thm 1}, the implication $(ii) \implies (iii)$ fails because the coefficients are unbounded. Conversely, the reverse implication $(iii) \implies (ii)$ is also false, owing to the fact that the coefficients are permitted to be arbitrarily degenerate.
\end{remark}

We now turn our attention to the uniqueness of weak solutions to equation \eqref{main eqn 1}. 
Although the hypotheses on the source term $f$ could technically be relaxed if we were only concerned with uniqueness, discussing uniqueness is vacuous without first ensuring existence. 
Therefore, we restrict our focus to function classes where both the existence and uniqueness of a weak solution are simultaneously guaranteed.
For $p \in (1,\infty)$, the $L_p$-spaces are reflexive Banach spaces and thus benefit from robust functional analytic properties. 
In contrast, the boundary cases $p=1$ and $p=\infty$ lack reflexivity, introducing notorious technical challenges when handling $L_p$-valued functions. 
Because these extreme cases are both mathematically demanding and of profound theoretical interest—especially since existence has already been established by the preceding theorem—we choose to first tackle the highly intriguing case of $p=1$.

\begin{theorem}[Existence and uniqueness of a weak solution in $L_1$]
										\label{main thm 3-1}

Suppose $f$ is locally integrable on $[0,T) \times \fR^d$ and that Assumptions \ref{main as 1}--\ref{main as 3} hold. 
Assuming further that conditions (i)–(iii) of Theorem \ref{main thm 1} are met for $p=1$, equation \eqref{main eqn 1} possesses a unique weak solution $u$ belonging to the class of functions defined by the following four properties
\begin{enumerate}[(i)]
\item 
(Spatial Integrability) For any $t \in (0,T)$,$$\|u(t,\cdot)\|_{\mathrm{L}_{1}(\fR^d)} < \infty.$$
\item
(Weighted Fourier Integrability) For any $T' \in (0,T)$, the following integral is finite:
$$
\int_0^{T'}\mu_{a,b,c}(t)\|\cF[u(t,\cdot)]\|_{\mathrm{L}_{\infty}(\fR^d)} \mathrm{d}t < \infty,
$$
where $\cF[u(t,\cdot)]$ denotes the spatial Fourier transform defined by $\cF[u(t,\cdot)](\xi) = \frac{1}{(2\pi)^{d/2}} \int_{\fR^d} \mathrm{e}^{-ix \cdot \xi} u(t,x)\mathrm{d}x$.
\item
(Temporal Continuity in Frequency) The map $t \mapsto \cF[u(t,\cdot)](\xi)$ is continuous on $[0,T)$ for almost every $\xi \in \fR^d$.
\item
(Initial vanishing condition) For almost every $\xi \in \fR^d$ and all $t \in (0,T)$,
$$
\lim_{\varepsilon \downarrow 0}\exp\left(\int_\varepsilon^t\left(-a^{ij}(s)\xi^i\xi^j+c(s) \right)\mathrm{d}s  \right) \cF[u(\varepsilon,\cdot)](\xi)=0.
$$
\end{enumerate}
Additionally, this unique solution $u$ satisfies
$$
\int_0^{t} \mu_{a,b,c}(s) \|u(s,\cdot)\|_{\mathrm{L}_1(\fR^d)} \mathrm{d}s \leq \int_0^{t} \mu_{a,b,c}(s) \int_0^s \mathrm{e}^{\int_r^s c(\rho)\mathrm{d}\rho} \|f(r,\cdot)\|_{\mathrm{L}_1(\fR^d)}\mathrm{d}r \mathrm{d}s
$$
for any $t \in (0,T)$.
\end{theorem}
For the sake of exposition, we postpone the proof of this theorem. It will be restated as Theorem \ref{p1 thm} alongside its complete proof.
For $p \in (1,\infty)$, we can also establish both the existence and uniqueness of a weak solution within a slightly different function class. 
The corresponding theorem is stated below.

\begin{theorem}[Existence and uniqueness of a weak solution in $L_p$]
										\label{main thm 3}
Let $p \in (1,\infty)$ and suppose $f$ is a locally integrable function on $[0,T) \times \mathbf{R}^d$. Assume that Assumptions \ref{main as 1}--\ref{main as 3} hold, and further suppose that conditions (i)–(iii) of Theorem \ref{main thm 1} are satisfied for $p \in (1,\infty)$.
Then, there exists a unique weak solution $u$ to equation \eqref{main eqn 1} within the class of functions satisfying $\|u(t,\cdot)\|_{L_p(\mathbf{R}^d)} < \infty$ for all $t \in (0,T)$, the asymptotic initial condition
$$
\lim_{\delta \downarrow 0} \mathrm{e}^{\int_\delta^t c(r)\mathrm{d}r} \|u(\delta,\cdot)\|_{L_p(\mathbf{R}^d)} = 0,
$$
and the weighted integrability constraint
$$
\int_0^t \mu_{a,b,c}(s) \|u(s,\cdot)\|_{L_p(\mathbf{R}^d)} \mathrm{d}s < \infty.
$$
Moreover, for any $t \in (0,T)$, the solution $u$ obeys the bounds
\begin{align*}
 \|u(t,\cdot)\|_{\mathrm{L}_p(\fR^d)}
\leq  \int_0^{t} \mathrm{e}^{\int_s^t c(r)\mathrm{d}r} \|f(s,\cdot)\|_{\mathrm{L}_p(\fR^d)}\mathrm{d}s
\end{align*}
and
\begin{align*}
&\int_0^{t} \mu_{a,b,c}(s) \|u(s,\cdot)\|_{\mathrm{L}_p(\fR^d)} \mathrm{d}s \\
&\leq \int_0^{t} \mu_{a,b,c}(s) \int_0^s \mathrm{e}^{\int_r^s c(\rho)\mathrm{d}\rho} \|f(r,\cdot)\|_{\mathrm{L}_p(\fR^d)}\mathrm{d}r \mathrm{d}s.
\end{align*}
Assuming additionally that the coefficient $c(t)$ is locally integrable on $[0,T)$, meaning
\begin{align*}
\int_0^{T'} |c(t)| \mathrm{d}t < \infty \quad \forall T' \in (0,T),
\end{align*}
it follows that the map $t \mapsto \|u(t,\cdot)\|_{\mathrm{L}_p(\fR^d)}$ is absolutely continuous on $[0,T']$ for any $T' \in (0,T)$. 
Furthermore, we establish the following estimate:
\begin{align*}
\sup_{t \in [0,T']} \mathrm{e}^{-\int_0^t c(r)\mathrm{d}r}\|u(t,\cdot)\|_{\mathrm{L}p(\fR^d)}
\leq \int_0^{T'} \mathrm{e}^{-\int_0^s c(r)\mathrm{d}r} \|f(s,\cdot)\|_{\mathrm{L}_p(\fR^d)}\mathrm{d}s.
\end{align*}

\end{theorem}
We defer the proof of this theorem to Section \ref{elementary section}, where it will be restated and rigorously proven as Theorem \ref{other p thm}.

\begin{definition}
For $p \in (1,\infty)$, we let $\mathrm{AC}_{0,loc}([0,T); L_p(\mathbf{R}^d))$ denote the space of all locally integrable functions $u$ on $[0,T) \times \mathbf{R}^d$ that satisfy the following four conditions:
\begin{enumerate}[(i)]
\item For any $T' \in (0,T)$, the $L_p$-norm of $u$ is uniformly bounded over $[0,T']$, meaning that
$$
\sup_{t \in [0,T']} \|u(t,\cdot)\|_{L_p(\mathbf{R}^d)} := \sup_{t \in [0,T']} \left( \int_{\mathbf{R}^d} |u(t,x)|^p \mathrm{d}x \right)^{1/p} < \infty.
$$

\item For any $T' \in (0,T)$, the mapping $t \mapsto u(t,\cdot)$ is an absolutely continuous function from $[0,T']$ into $L_p(\mathbf{R}^d)$. 
That is, for every $\varepsilon > 0$, there exists a $\delta > 0$ such that for any finite collection of pairwise disjoint subintervals $(s_i, t_i) \subset [0,T']$, we have
$$
\sum_i (t_i - s_i) < \delta \implies \sum_i \|u(t_i,\cdot) - u(s_i,\cdot)\|_{L_p(\mathbf{R}^d)} < \varepsilon.
$$

\item  For almost every $x \in \mathbf{R}^d$ and any $T' \in (0,T)$, the temporal mapping $t \mapsto u(t,x)$ is a real-valued absolutely continuous function on $[0,T']$.

\item The function vanishes at the initial time in the $L_p$-sense, namely 
$$
\|u(0,\cdot)\|_{L_p(\mathbf{R}^d)} = 0.
$$
\end{enumerate}
\end{definition}

\begin{remark}
										\label{absolute remark}
Assume $p \ge 1$. Absolutely continuous functions taking values in $\mathrm{L}_p(\mathbf{R}^d)$ possess several significant characteristics. 
One primary property is that the absolute continuity of an $\mathrm{L}_p(\mathbf{R}^d)$-valued function on an interval is directly tied to its norm. 
Specifically, the mapping $t \mapsto u(t,\cdot)$ is absolutely continuous on $[0,T']$ if and only if its real-valued norm $t \mapsto \|u(t,\cdot)\|_{\mathrm{L}_p(\mathbf{R}^d)}$ is absolutely continuous on that same interval.
Furthermore, when $p > 1$, the space $\mathrm{L}_p(\mathbf{R}^d)$ possesses the Radon–Nikodým property. 
Because of this, any absolutely continuous $\mathrm{L}_p(\mathbf{R}^d)$-valued function $u$ defined on $[0,T']$ can be expressed via an integral representation:
$$
u(t) = u(0) + \int_0^t u_t(s) \, \mathrm{d}s \quad \forall t \in [0,T'].
$$
Here, the equality is understood in the $\mathrm{L}_p(\mathbf{R}^d)$-sense, and $u_t(s)$ represents the Fréchet derivative. 
If we consider $u$ as a function of both time and space on $[0,T'] \times \mathbf{R}^d$, this representation implies that for any given time $t \in [0,T']$, the equation $u(t,x) = u(0,x) + \int_0^t u_t(s,x) \, \mathrm{d}s$ holds for almost every $x \in \mathbf{R}^d$.
It is important to note that this pointwise equation does not automatically guarantee that the scalar function $t \mapsto u(t,x)$ is absolutely continuous on $[0,T']$ for almost every $x$. 
To achieve this, a straightforward and sufficient requirement is simply that $t \mapsto u(t,x)$ is continuous on $[0,T']$ for almost every $x \in \mathbf{R}^d$. 
When this continuity is combined with the integral equation, it guarantees that $t \mapsto u(t,x)$ is indeed an absolutely continuous real-valued function. 
Consequently, the formal definition of the space $\mathrm{AC}_{0,\mathrm{loc}}([0,T); \mathrm{L}_p(\mathbf{R}^d))$ for $p \in (1,\infty)$ can be simplified. 
The standard condition $(iii)$ can be replaced with the less demanding requirement that the mapping $t \mapsto u(t,x)$ is continuous on $[0,T']$ for almost every $x \in \mathbf{R}^d$ and for all $T' \in (0,T)$. 
For further extensive details on these properties, one can refer to literature such as \cite{HNVW 2016}, specifically Lemma 1.3.7 and Theorem 2.5.12.
\end{remark}

Because singularities in the coefficients complicate the standard analysis of solutions and data in traditional $\mathrm{L}_p$-spaces, we address this issue by utilizing weighted $\mathrm{L}_p$-spaces instead. 
Furthermore, since the coefficients depend entirely on time, the corresponding weights are also defined strictly in terms of the temporal variable.

\begin{definition}[Temporal weighted $\mathrm{L}_{p,q}$-spaces]
To formalize our functional framework, assume $p,q \in [1,\infty]$ and let $\mu$ represent a non-negative measurable weight function defined almost everywhere on the interval $(0,T)$.
We introduce the space $\mathrm{L}_{q,p,loc}((0,T) \times \fR^d, \mu(t)\mathrm{d}t)$ to denote the collection of all measurable functions $f$ defined $a.e.$ on $(0,T) \times \fR^d$ that satisfy a specific local integrability condition. 
Specifically, for any intermediate time $T' \in (0,T)$, the associated weighted mixed norm must be finite:
$$
\|f\|_{\mathrm{L}_{q,p}((0,T') \times \fR^d, \mu(t)\mathrm{d}t)} := \|f\|_{\mathrm{L}_{q}((0,T') , \mu(t)\mathrm{d}t ; \mathrm{L}_p(\fR^d))} < \infty.
$$
Explicitly, this norm evaluates to
$$
\begin{cases}
\left[\int_0^{T'} \|f(t,\cdot)\|^q_{\mathrm{L}_p(\fR^d)} \mu(t) \, \mathrm{d}t \right]^{1/q} < \infty & \text{if } q \in [1,\infty), \\
\inf\left\{ M : \mu(\{t : \|f(t,\cdot)\|_{\mathrm{L}_p(\fR^d)} > M\}) = 0 \right\} < \infty & \text{if } q=\infty,
\end{cases}
$$
where we utilize the measure notation $\mu(\mathrm{d}t) = \mu(t)\mathrm{d}t$.
For the standard unweighted scenario where $\mu(t)=1$ a.e. on $(0,T)$, we streamline the notation by dropping the measure.
 In this case, we simply write $\mathrm{L}_{q,p,loc}((0,T) \times \fR^d)$ for the space itself and $\|f\|_{\mathrm{L}_{q,p}((0,T') \times \fR^d)}$ for its corresponding norm.
\end{definition}

\begin{remark}
Let $\mu$ be a non-negative measurable function defined almost everywhere on $(0,T)$, and let $m(\mathrm{d}t)$ represent the standard Lebesgue measure on $\mathbf{R}$. 
When equivalence classes are formed based on equality $\mu$-almost everywhere, the space $\mathrm{L}_{q,p,\mathrm{loc}}((0,T) \times \mathbf{R}^d, \mu(t)\mathrm{d}t)$ constitutes a Banach space. 
However, a complication arises if the weight $\mu(t)$ is zero on a set of positive Lebesgue measure, as two functions being equal $\mu$-almost everywhere does not guarantee they are equal $m$-almost everywhere. 
Because of this discrepancy, if we define our equivalence classes using $m$-almost everywhere equality, the expression $\|f\|_{\mathrm{L}_{q,p}((0,T') \times \mathbf{R}^d, \mu(t)\mathrm{d}t)}$ functions strictly as a semi-norm rather than a true norm. 
To eliminate any potential confusion between these two interpretations, we explicitly reject the $\mu$-almost everywhere convention. 
Consequently, for the remainder of the paper, we treat the space $\mathrm{L}_{q,p,\mathrm{loc}}((0,T) \times \mathbf{R}^d, \mu(t)\mathrm{d}t)$ as a topological vector space induced by this semi-norm, and we emphasize that every reference to ``almost everywhere" or ``a.e." applies exclusively to the standard Lebesgue measure.
Nonetheless, to guarantee the uniqueness of limits, we implicitly pass to the associated quotient space whenever taking limits is required.
\end{remark}

\begin{assumption}
						\label{main as 4}

We assume the coefficient $c(t)$ is locally integrable on the interval $[0,T)$, meaning that for every $t \in (0,T)$, we have
$$
\int_0^t |c(s)| \, \mathrm{d}s < \infty.
$$
\end{assumption}

Assumption \ref{main as 4} imposes a tighter condition compared to Assumption \ref{main as 3}. 
By operating under this more stringent hypothesis, we can identify a highly precise functional class and derive straightforward criteria for the well-posedness of the problem. 
Furthermore, it naturally yields a sharper upper bound for the temporal supremum of $u$.

\begin{theorem}[Temporal estimate of a weak solution]
										\label{main thm 4}
Suppose $p \in 1,\infty)$ and the source term $f$ belongs to $\mathrm{L}_{1,p,loc}((0,T) \times \fR^d)$. 
Provided that Assumptions \ref{main as 1}, \ref{main as 2}, and \ref{main as 4} are satisfied, along with the integrability condition
$$
\int_0^{T'} \mu_{a,b,c}(t) \int_0^t \|f(s,\cdot)\|_{\mathrm{L}_p(\fR^d)} \, \mathrm{d}s \, \mathrm{d}t < \infty
$$
for every $T' \in (0,T)$, it follows that the unique weak solution $u$ to equation \eqref{main eqn 1}—as established in Theorems \ref{main thm 3-1} and \ref{main thm 3}—belongs to the spaces
$$
u \in \mathrm{AC}_{0,loc}([0,T) ; \mathrm{L}_p(\fR^d)) \cap \mathrm{L}_{1,p,loc}((0,T) \times \fR^d, \mu_{a,b,c}(t)\mathrm{d}t)
$$
and
$$
u \in  \mathrm{L}_{\infty,p,loc}\left((0,T) \times \fR^d, \mathrm{e}^{-\int_0^t c(s)\mathrm{d}s}\mathrm{d}t\right).
$$
Additionally, for any $T' \in (0,T)$, this solution $u$ fulfills the following estimates:
$$
\sup_{t \in [0,T']} \mathrm{e}^{-\int_0^t c(r)\mathrm{d}r} \|u(t,\cdot)\|_{\mathrm{L}_p(\fR^d)} \leq \sup_{t \in [0,T']} \int_0^{t} \mathrm{e}^{-\int_0^s c(r)\mathrm{d}r} \|f(s,\cdot)\|_{\mathrm{L}_p(\fR^d)} \, \mathrm{d}s,
$$
as well as
$$
\int_0^{T'} \mu_{a,b,c}(t) \|u(t,\cdot)\|_{\mathrm{L}_p(\fR^d)} \, \mathrm{d}t \leq \mathrm{e}^{\int_0^{T'} |c(t)|\mathrm{d}t} \int_0^{T'} \mu_{a,b,c}(t) \int_0^t \|f(s,\cdot)\|_{\mathrm{L}_p(\fR^d)} \, \mathrm{d}s \, \mathrm{d}t.
$$

\end{theorem}

Although this theorem follows straightforwardly from Theorem \ref{main thm 3-1} and Theorem \ref{main thm 3}, its complete proof is provided in Section \ref{pf main thm}.

\begin{remark}
We emphasize the critical role that the space $\mathrm{AC}_{0,loc}([0,T) ; \mathrm{L}_p(\mathbf{R}^d))$ plays in establishing a  desirable regularity property for our weak solution. 
Because our solution to equation \eqref{main eqn 1} resides in this class, the temporal mapping $t \mapsto u(t,x)$ is differentiable almost everywhere in $t \in (0,T)$ for almost every fixed $x \in \mathbf{R}^d$. 
Consequently, the pointwise time derivative $u_t(t,x)$ exists almost everywhere on $(0,T) \times \mathbf{R}^d$, despite $u$ being defined only in a weak sense. 
This is a notable result; it is generally challenging to infer such pointwise differentiability directly from standard weak formulations, primarily because limiting processes do not typically preserve absolute continuity.
\end{remark}

\begin{remark}
Assuming the coefficients $a^{ij}(t)$, $b^i(t)$, and $c(t)$ are locally integrable on $[0,T)$—meaning they satisfy the condition
$$
\int_0^t \left(\sum_{i=1}^d\sum_{j=1}^d|a^{ij}(s)|+\sum_{i=1}^d |b^i(s)| + |c(s)|\right) \mathrm{d}s < \infty \quad \text{for all } t \in (0,T),
$$
equation \eqref{main eqn 1} admits a unique weak solution within a broader functional class, specifically referred to as a Fourier-space weak solution. 
Remarkably, this result holds entirely independently of any ellipticity assumptions. For an in-depth discussion, we refer the reader to \cite{CK 2024, CK 2024-2}.
\end{remark}

Muckenhoupt weights are fundamental tools in both harmonic analysis and the theory of partial differential equations, making the study of weighted estimates for solutions and data a standard practice. 
However, as noted earlier, traditional Muckenhoupt estimates are incompatible with our current framework. Instead, we must consider adapted estimates that can handle singular coefficients. 
Furthermore, because incorporating spatial weights introduces significant technical complexity, we defer their analysis for now and focus exclusively on one-dimensional Muckenhoupt weights applied to the time variable.
 While comprehensive details on the properties of these weights can be found in literature such as \cite{Gra 2014}, we will limit our current discussion to recalling their formal definition and providing a fundamental example.

\begin{definition}[Muckenhoupt's weight]
							\label{def weight}
For any $p \in (1,\infty)$, the notation $\mathrm{A}_p(\fR)$ designates the set of all non-negative, locally integrable functions $w$ defined on $\fR$ that satisfy the Muckenhoupt condition
$$
[w]_{\mathrm{A}_p(\fR)} := \sup_{-\infty < a < b < \infty} \left( \frac{1}{b-a} \int_a^b w(t) \, \mathrm{d}t \right) \left( \frac{1}{b-a} \int_a^b w(t)^{-1/(p-1)} \, \mathrm{d}t \right)^{p-1} < \infty.
$$
A direct consequence of this finiteness condition is that $w(t) > 0$ almost everywhere on $\fR$.
\end{definition}
\begin{remark}
									\label{mucken remark}
A straightforward and notable example of an $\mathrm{A}_p(\fR)$-weight is the absolute value function raised to a restricted power. 
Specifically, it is a well-established result that the function $|t|^\beta$ belongs to the Muckenhoupt class $\mathrm{A}_p(\fR)$ provided the exponent satisfies the condition $-1 < \beta < p-1$ (\textit{cf.} \cite[Example 7.1.7]{Gra 2014}).
\end{remark}

Next, we introduce Sobolev derivatives via a weak formulation. While multiple approaches exist for defining these derivatives, we deliberately adopt this specific method. 
We acknowledge that a comprehensive review of these elementary concepts may seem overly exhaustive to some readers. 
However, we have chosen to present them in full due to the non-standard nature of our weighted Sobolev spaces.
 Furthermore, within this generalized framework, the Sobolev derivatives emerge as locally integrable functions, demanding a level of rigorous treatment that is not typically found in conventional literature.

\begin{definition}[Sobolev derivative]
Let $\alpha$ be a multi-index and suppose $v$ is a locally integrable function on $\fR^d$, meaning that for any finite radius $r > 0$, the following holds:
$$
\int_{|x| < r} |v(x)| \, \mathrm{d}x < \infty.
$$
We define a locally integrable function $v^\alpha$ on $\fR^d$ to be the Sobolev derivative (or weak derivative) of $v$ with respect to the multi-index $\alpha$ if it satisfies the following integration-by-parts relation for every test function $\varphi \in \mathrm{C}_c^\infty(\fR^d)$:
$$
\int_{\fR^d} v(x) D^\alpha \varphi(x) \, \mathrm{d}x = (-1)^{|\alpha|} \int_{\fR^d} v^\alpha(x) \varphi(x) \, \mathrm{d}x.
$$

It is straightforward to verify that this Sobolev derivative is unique up to a set of measure zero. 
That is, if $v_1^\alpha$ and $v_2^\alpha$ are both Sobolev derivatives of $v$ for the same multi-index $\alpha$, then $v_1^\alpha(x) = v_2^\alpha(x)$ almost everywhere on $\fR^d$. 
We standardly denote this unique equivalence class by $D^\alpha v$.
For lower-order derivatives where $|\alpha|=1$ or $2$, we adopt conventional subscript notation, utilizing $v_{x^i}$ and $v_{x^ix^j}$ instead of $D^\alpha v$ (for $i,j \in \{1, \ldots, d\}$). 
Correspondingly, $v_x$ represents the gradient vector, and $v_{xx}$ denotes the Hessian matrix composed of these second-order Sobolev derivatives. 
Because weak mixed partial derivatives commute ($v_{x^ix^j} = v_{x^jx^i}$), it immediately follows that the Hessian matrix $v_{xx}$ is symmetric.
\end{definition}

\begin{definition}[Temporal weighted Sobolev spaces]
Let $p, q \in [1, \infty)$ and assume $\mu$ is a non-negative, almost-everywhere defined measurable function on the interval $(0,T)$. 
We define the space $\mathrm{H}^2_{q,p,loc}((0,T) \times \fR^d, \mu(t)\mathrm{d}t)$ to be the set of all measurable functions $u$ on $(0,T) \times \fR^d$ that exhibit finite weighted integrability up to their second-order spatial weak derivatives.
Explicitly, a function $u$ belongs to this space if, for every truncation time $T' \in (0,T)$, its corresponding norm is finite:
\begin{align*}
&\|u\|_{\mathrm{H}^2_{q,p}((0,T') \times \fR^d, \mu(t)\mathrm{d}t)} \\
&:= \left[\int_0^{T'} \left(\|u(t,\cdot)\|^q_{\mathrm{L}_p(\fR^d)} + \|u_x(t,\cdot)\|^q_{\mathrm{L}_p(\fR^d)} + \|u_{xx}(t,\cdot)\|^q_{\mathrm{L}_p(\fR^d)}\right) \mu(t) \, \mathrm{d}t\right]^{1/q} < \infty.
\end{align*}
In this notation, $u_x(t,\cdot)$ and $u_{xx}(t,\cdot)$ denote the spatial gradient vector and the spatial Hessian matrix, respectively, which are assembled from their Sobolev partial derivatives.
\end{definition}

We now establish a maximal $\mathrm{L}_p$-regularity estimate for equation \eqref{main eqn 1}.

\begin{theorem}[Regularity estimates]
									\label{main thm 5}
Let $p,q \in (1,\infty)$, $f \in \mathrm{L}_{1,p,loc}\left( (0,T) \times \fR^d\right)$, and $w \in \mathrm{A}_q(\fR)$.
Suppose that Assumption \ref{main as 1},  Assumption \ref{main as 2}, and Assumption \ref{main as 4} hold.
Additionally, assume that for all $T' \in (0,T)$,
\begin{align*}
\int_0^{T'} \mu_{a,b,c}(t) \int_0^t  \|f(s,\cdot)\|_{\mathrm{L}_p(\fR^d)}\mathrm{d}s \mathrm{d}t < \infty
\end{align*}
and
\begin{align}
										\notag
&\int_0^{T'}  \|f(t,\cdot)\|_{\mathrm{L}_p(\fR^d)}^q\mathrm{e}^{-q\int_0^tc(s)\mathrm{d}s}  w(\alpha(t)) (\lambda(t))^{1-q}\mathrm{d}t  \\
										\label{20240912 50}
&:= \lim_{\delta \downarrow 0}\int_0^{T'}\|f(t,\cdot)\|_{\mathrm{L}_p(\fR^d)}^q\mathrm{e}^{-q\int_0^tc(s)\mathrm{d}s}  w(\alpha(t)+\delta t) (\lambda(t) +\delta)^{1-q}\mathrm{d}t  < \infty,
\end{align}
where $\lambda(t)$ is a non-negative locally integrable function on $[0,T)$ such that
\begin{align*}
\lambda(t)|\xi|^2 \leq a^{ij}(t)\xi^i\xi^j \quad \text{a.e.}~ t \in (0,T)~\text{and}~ \xi \in \fR^d
\end{align*}
and
$$
\alpha(t):=\int_0^{t} \lambda(s) \mathrm{d}s < \infty \quad \forall t \in (0,T).
$$
Then the unique weak solution $u$ to equation \eqref{main eqn 1}—as established in Theorem \ref{main thm 3}— belongs to the intersection of the following four classes:
$\mathrm{AC}_{0,loc}\left([0,T) ; \mathrm{L}_p(\fR^d)  \right)$, 
$$
\mathrm{L}_{1,p,loc}\left( (0,T) \times \fR^d, \mu_{a,b,c}(t)\mathrm{d}t\right),  
$$
$$
\mathrm{L}_{\infty,p,loc}\left( (0,T) \times \fR^d,  \mathrm{e}^{-\int_0^t c(s)\mathrm{d}s}\mathrm{d}t\right),
$$
and
$$
\mathrm{H}^2_{q,p,loc}\left( (0,T) \times \fR^d, \mu_{c,\lambda}(t)\mathrm{d}t\right),
$$
where 
\begin{align*}
\mu_{c,\lambda}(t)
=\mathrm{e}^{-q\int_0^tc(s)\mathrm{d}s}  w(\alpha(t)) \lambda(t).
\end{align*}
Moreover, the solution $u$ satisfies the following estimates that for any $T' \in (0,T)$,
\begin{align}
										\label{main est 1-1}
\sup_{t \in [0,T']} \left(\mathrm{e}^{-\int_0^tc(s)\mathrm{d}s}\|u(t,\cdot)\|_{\mathrm{L}_p(\fR^d)} \right)
\leq \|f\|_{\mathrm{L}_{1,p}\left((0,T') \times \fR^d, \mathrm{e}^{-\int_0^tc(s)\mathrm{d}s} \mathrm{d}t\right)},
\end{align}
\begin{align}
										\label{main est 1-2}
\|u\|^q_{\mathrm{L}_{q,p}\left((0,T') \times \fR^d, \mu_{a,b,c}(t)\mathrm{d}t \right)} 
\leq \mathrm{e}^{\int_0^{T'} |c(t)|\mathrm{d}t} \|f\|^q_{\mathrm{L}_{q,p}\left((0,T') \times \fR^d, \mu_{a,b,c}(t)\mathrm{d}t \right)},
\end{align}
\begin{align}
										\notag
&\|u\|^q_{\mathrm{L}_{q,p}\left((0,T') \times \fR^d, \mu_{c,\lambda}(t)\mathrm{d}t \right)} 
\\
										\label{main est 2-1}
&\leq [w]_{\mathrm{A}_p(\fR)}[\alpha(T')]^q
\int_0^{T'}  \|f(t,\cdot)\|_{\mathrm{L}_p(\fR^d)}^q\mathrm{e}^{-q\int_0^tc(s)\mathrm{d}s}  w(\alpha(t)) (\lambda(t))^{1-q}\mathrm{d}t,
\end{align}
\begin{align}
										\notag
&\|u_x\|^q_{\mathrm{L}_{q,p}\left((0,T') \times \fR^d, \mu_{c,\lambda}(t)\mathrm{d}t \right)}  \\
										\label{main est 2-2}
&\lesssim_{T,d,p,q,w}
\int_0^{T'}  \|f(t,\cdot)\|_{\mathrm{L}_p(\fR^d)}^q\mathrm{e}^{-q\int_0^tc(s)\mathrm{d}s}  w(\alpha(t)) (\lambda(t))^{1-q}\mathrm{d}t,
\end{align}
and
\begin{align}
										\notag
&\|u_{xx}\|^q_{\mathrm{L}_{q,p}\left((0,T') \times \fR^d, \mu_{c,\lambda}(t)\mathrm{d}t \right)} \\
										\label{main est 2}
&\lesssim_{d,p,q,w}
\int_0^{T'}  \|f(t,\cdot)\|_{\mathrm{L}_p(\fR^d)}^q\mathrm{e}^{-q\int_0^tc(s)\mathrm{d}s}  w(\alpha(t)) (\lambda(t))^{1-q}\mathrm{d}t.
\end{align}
\end{theorem}

The proof of Theorem \ref{main thm 5} is also provided in Section \ref{pf main thm}.
\vspace{2mm}

\begin{remark}
Generally, the integral in Theorem \ref{main thm 5} should not be viewed as a standard Lebesgue integral. Because $\lambda(t)$ may vanish on a set of positive measure, the expression must be interpreted as an improper integral, as formally defined in \eqref{20240912 50}. 
This distinction explains why we provide the integral in its expanded form within the estimates rather than using the shorthand norm notation associated with weighted $\mathrm{L}_{q,p}$-spaces; the latter typically presumes a standard Lebesgue measure that might not be applicable here.
\end{remark}

To demonstrate the scope and generality of our results, we present two primary examples. 
These cases focus on the leading coefficients $a^{ij}(t)$ and serve as straightforward yet powerful illustrations of the diverse types of singularities accommodated by our framework.
\begin{example}
									\label{exam 1}
A straightforward illustration is provided by the leading coefficients $a^{ij}(t) = \left(t^{\beta_1} + (T-t)^{\beta_2} \right)\delta^{ij}$ for $\beta_1, \beta_2 \in \fR$. In this expression, $\delta^{ij}$ represents the Kronecker delta, which is equal to $1$ when $i = j$ and $0$ otherwise.
\end{example}

\begin{example}
									\label{exam 2}
It is straightforward to verify that our coefficients $a^{ij}(t)$ may exhibit exponential growth, as we impose no integrability constraints at the initial or terminal times. 
For example, the coefficients defined by
$$
a^{ij}(t) = \left(\exp\left(\mathrm{e}^{1/t}\right) + \exp\left(\mathrm{e}^{1/(T-t)}\right)\right)\delta^{ij}
$$
are consistent with the requirements of Assumption \ref{main as 1}.
\end{example}

Our analysis will further explore the existence and properties of strong solutions for \eqref{main eqn 1}. 
The specific definition of a strong solution used in this study is provided below.

\begin{definition}[Strong solution]
										\label{def strong sol}
Let $f$ be a locally integrable function on $[0,T) \times \fR^d$ and let $u$ be a weak solution to \eqref{main eqn 1}. 
Suppose that for almost every $t \in (0,T)$ and all $i,j \in \{1,\ldots,d\}$, $u(t,\cdot)$ possesses the spatial Sobolev derivatives $u_{x^ix^j}(t,\cdot)$ and $u_{x^i}(t,\cdot)$. 
We then say that $u$ is a strong solution to \eqref{main eqn 1} if, for every $t \in (0,T)$, the following integral identity holds for almost every $x \in \fR^d$:
\begin{align}
										\notag
&u(t,x)  \\
										\label{20240912 01}
&= \int_0^t \left(a^{ij}(s)u_{x^ix^j}(s,x)  +b^i(s) u_{x^i}(s,x) + c(s)u(s,x) + f(s,x) \right) \mathrm{d}s.
\end{align}
\end{definition}

\begin{remark}
									\label{strong weak sol}
According to Definition \ref{def strong sol}, a strong solution $u$ to \eqref{main eqn 1} must also be a weak solution. 
While one might assume that any strong solution naturally satisfies the weak formulation without requiring explicit enforcement, proving this directly from the integral identity \eqref{20240912 01} relies on an additional integrability assumption:
\begin{align}
										\label{20250816 01}
\int_0^t \int_{|x|<r} \left(|a^{ij}(s)| |u_{x^ix^j}(s,x)|  +|b^i(s)| |u_{x^i}(s,x)| + |c(s)| |u(s,x)|  \right) \mathrm{d}x \mathrm{d}s < \infty,
\end{align}
When this holds for every $t \in (0,T)$ and $r > 0$, Fubini's theorem justifies that $u$ is indeed a weak solution.
 In our framework, however, the presence of significant singularities means this local integrability cannot be guaranteed. 
Consequently, to ensure validity, we must explicitly include the weak solution property as a foundational requirement in the definition of a strong solution.

\end{remark}

\begin{theorem}[Well-posedness of a strong solution]
									\label{main thm 6}
Let $p,q \in (1,\infty)$, $f \in \mathrm{L}_{1,p,loc}\left( (0,T) \times \fR^d\right)$, and $w \in \mathrm{A}_q(\fR)$.
Suppose that Assumption \ref{main as 1},  Assumption \ref{main as 2}, and Assumption \ref{main as 4} hold.
Additionally, assume that for all $T' \in (0,T)$,
\begin{align}
									\notag
&\int_0^{T'} \mu_{a,b,c}(t)^{\frac{q}{q-1}} \left(w(\alpha(t)) \lambda(t) \right)^{-1/(q-1)} \mathrm{d}t \\
									\label{20240912 60}
&:=\lim_{\delta \downarrow 0}\int_0^{T'} \mu_{a,b,c}(t)^{\frac{q}{q-1}}
\left(w(\alpha(t)+\delta t) (\lambda(t) +\delta) \right)^{-1/(q-1)} \mathrm{d}t 
< \infty,
\end{align}
\begin{align*}
\int_0^{T'} \mu_{a,b,c}(t) \int_0^t  \|f(s,\cdot)\|_{\mathrm{L}_p(\fR^d)}\mathrm{d}s \mathrm{d}t < \infty,
\end{align*}
and
\begin{align*}
&\int_0^{T'}  \|f(t,\cdot)\|_{\mathrm{L}_p(\fR^d)}^q\mathrm{e}^{-q\int_0^tc(s)\mathrm{d}s}  w(\alpha(t)) (\lambda(t))^{1-q}\mathrm{d}t  \\
&:= \lim_{\delta \downarrow 0}\int_0^{T'}  \|f(t,\cdot)\|_{\mathrm{L}_p(\fR^d)}^q\mathrm{e}^{-q\int_0^tc(s)\mathrm{d}s}  w(\alpha(t)+\delta t) (\lambda(t) +\delta)^{1-q}\mathrm{d}t  < \infty,
\end{align*}
where $\lambda(t)$ is a non-negative locally integrable function on $[0,T)$ such that
\begin{align*}
\lambda(t)|\xi|^2 \leq a^{ij}(t)\xi^i\xi^j \quad \text{a.e.}~ t \in (0,T)~\text{and}~ \xi \in \fR^d
\end{align*}
and
$$
\alpha(t):=\int_0^{t} \lambda(t) \mathrm{d}t < \infty  \quad \forall t \in (0,T).
$$
Then there exists a strong solution $u$ to \eqref{main eqn 1} in the intersection of the following four classes:
$\mathrm{AC}_{0,loc}\left([0,T) ; \mathrm{L}_p(\fR^d)  \right)$, 
$$
\mathrm{L}_{1,p,loc}\left( (0,T) \times \fR^d, \mu_{a,b,c}(t)\mathrm{d}t\right),  
$$
$$
\mathrm{L}_{\infty,p,loc}\left( (0,T) \times \fR^d,  \mathrm{e}^{-\int_0^t c(s)\mathrm{d}s}\mathrm{d}t\right),
$$
and
$$
\mathrm{H}^2_{q,p,loc}\left( (0,T) \times \fR^d, \mu_{c,\lambda}(t)\mathrm{d}t\right)
$$
where 
\begin{align*}
\mu_{c,\lambda}(t)
=\mathrm{e}^{-q\int_0^tc(s)\mathrm{d}s}  w(\alpha(t)) \lambda(t).
\end{align*}
Moreover, the solution $u$ satisfies the following estimates that for any $T' \in (0,T)$,
\begin{align}
									\label{20240914 31}
\sup_{t \in [0,T']} \left(\mathrm{e}^{-\int_0^tc(s)\mathrm{d}s}\|u(t,\cdot)\|_{\mathrm{L}_p(\fR^d)} \right)
\leq \|f\|_{\mathrm{L}_{1,p}\left((0,T') \times \fR^d, \mathrm{e}^{-\int_0^tc(s)\mathrm{d}s} \mathrm{d}t\right)},
\end{align}
\begin{align*}
\|u\|_{\mathrm{L}_{q,p}\left((0,T') \times \fR^d, \mu_{a,b,c}(t)\mathrm{d}t \right)} 
\leq \mathrm{e}^{\int_0^{T'} |c(t)|\mathrm{d}t} \|f\|_{\mathrm{L}_{q,p}\left((0,T') \times \fR^d, \mu_{a,b,c}(t)\mathrm{d}t \right)},
\end{align*}
\
\begin{align*}
&\|u_t\|^q_{\mathrm{L}_{q,p}\left((0,T') \times \fR^d, \mu_{c,\lambda}(t)\mathrm{d}t \right)} +\|u\|^q_{\mathrm{H}^2_{q,p}\left((0,T') \times \fR^d, \mu_{c,\lambda}(t)\mathrm{d}t \right)}\\
&\lesssim_{T,d,p,q,w}
\int_0^{T'}  \|f(t,\cdot)\|_{\mathrm{L}_p(\fR^d)}^q\mathrm{e}^{-q\int_0^tc(s)\mathrm{d}s}  w(\alpha(t)) (\lambda(t))^{1-q}\mathrm{d}t,
\end{align*}
and
\begin{align}
										\notag
&\|u\|_{\mathrm{L}_{1,p}\left((0,T') \times \fR^d \right)}\\
										\notag
&\lesssim_{d,p,q,w}
\int_0^{T'} \mu_{a,b,c}(t)^{\frac{q}{q-1}} \left(w(\alpha(t)) \lambda(t) \right)^{-1/(q-1)} \mathrm{d}t \\
										\label{20240914 20}
&\quad \times \int_0^{T'}  \|f(t,\cdot)\|_{\mathrm{L}_p(\fR^d)}^q\mathrm{e}^{-q\int_0^tc(s)\mathrm{d}s}  w(\alpha(t)) (\lambda(t))^{1-q}\mathrm{d}t.
\end{align}
\end{theorem}
\vspace{2mm}

Similarly, the proof of Theorem \ref{main thm 6} is provided in Section \ref{pf main thm}.

\begin{remark}

Under the hypotheses of Theorem \ref{main thm 6}, it is straightforward to verify that any strong solution $u$ is also a weak solution in the sense of Theorem \ref{main thm 3}. 
Consequently, the uniqueness of the strong solution within the class defined by Theorem \ref{main thm 6} follows directly from Theorem \ref{main thm 3}.

\end{remark}

\begin{remark}
Since the solution $u$ belongs to the class $\mathrm{AC}_{0,loc}([0,T) ; \mathrm{L}_p(\fR^d))$, the term $u_t$ in \eqref{20240914 20} is well-defined as an $\mathrm{L}_p(\fR^d)$-valued Fréchet derivative. 
Furthermore, one can readily verify that $u_t$ coincides with the Sobolev derivative by applying the fundamental theorem of calculus in conjunction with integration by parts.
\end{remark}

\begin{remark}
Condition \eqref{20240912 60} is imposed to ensure the well-posedness of the strong solution in Theorem \ref{main thm 6}. Clearly, \eqref{20240912 60} is satisfied if all coefficients are bounded and $a^{ij}(t)$ is uniformly elliptic; that is, there exists a constant $\kappa \in (0,1)$ such that
$$
\kappa |\xi|^2 \leq a^{ij}(t)\xi^i\xi^j \leq \frac{1}{\kappa}|\xi|^2 \quad \forall \xi \in \fR^d \text{ and } t \in (0,T).
$$
Furthermore, condition \eqref{20240912 60} guarantees the local integrability of the coefficients on $[0,T)$. To see this, let $\delta > 0$ and $T' \in (0,T)$. By applying H\"older's inequality and a straightforward change of variables, we obtain
\begin{align*}
&\int_0^{T'} \mu_{a,b,c}(t)\mathrm{d}t \\
&= \int_0^{T'} \mu_{a,b,c}(t) \left(w(\alpha(t)+\delta t) (\lambda(t)+\delta) \right)^{-1/q} \left(w(\alpha(t)+\delta t) (\lambda(t)+\delta) \right)^{1/q} \mathrm{d}t \\
&\leq \left(\int_0^{T'} \mu_{a,b,c}(t)^{\frac{q}{q-1}} \left(w(\alpha(t)+\delta t) (\lambda(t)+\delta) \right)^{-1/(q-1)} \mathrm{d}t \right)^{\frac{q-1}{q}} \\
&\quad \times \left(\int_0^{\alpha(T')+\delta T'} w(t) \mathrm{d}t \right)^{1/q}.
\end{align*} 
Letting $\delta \downarrow 0$, it follows that for any $T' \in (0,T)$,
$$
\int_0^{T'} \mu_{a,b,c}(t)\mathrm{d}t 
\lesssim \left(\int_0^{T'} \mu_{a,b,c}(t)^{\frac{q}{q-1}} \left(w(\alpha(t)) \lambda(t) \right)^{-1/(q-1)} \mathrm{d}t \right)^{\frac{q-1}{q}}.
$$
Consequently, the local integrability of the coefficients on $[0,T)$, namely that
$$
\int_0^{T'} \mu_{a,b,c}(t)\mathrm{d}t  < \infty,
$$
is a necessary requirement to establish the existence and uniqueness of the strong solution to \eqref{main eqn 1} in Theorem \ref{main thm 6}.
\end{remark}

Although the local integrability of the coefficients on $[0,T)$ is a prerequisite for applying our theorem, our condition is robust enough to capture compelling, non-trivial scenarios well beyond standard uniform ellipticity. We provide an illustrative example below to demonstrate this.

\begin{example}
									\label{exam 233}
Let $-\infty < a_1 \leq a_2 <1$.
We investigate existence and uniqueness of a strong solution to the following simple equation:
\begin{align}
									\notag
&u_t(t,x)=  t^{-a_1}u_{x^1x^1}(t,x)+t^{-a_2}u_{x^2x^2}(t,x)+f(t,x),  \\
									\label{20240913 10}
&u(0,x)=0, \qquad (t,x) \in (0,T) \times \fR^2 
\end{align}
We establish the existence of a unique strong solution to \eqref{20240913 10} in a weighted $\mathrm{L}_{p,q}$-space, provided $f$ satisfies an appropriate condition. 
For the special case where $a_1 = a_2$ and $f$ is smooth, this result is already known (see, \textit{e.g.}, \cite[Theorem 3.1]{KL 2017}).
Recalling Remark \ref{mucken remark} and the notation from Theorem \ref{main thm 6}, we set
$$
A(t) = (a_{ij}(t)) = \begin{bmatrix} t^{-a_1} & 0 \\ 0 & t^{-a_2} \end{bmatrix}.
$$
For all $t \in (0,T)$, we define $\lambda(t) = t^{-a_1}$, $\alpha(t) = \frac{1}{1-a_1} t^{1-a_1}$, $b^i(t) = c(t) = 0$, and the weight $w(t) = t^\beta$ with $-1 < \beta < q-1$. 
Then, for any $T' \in (0,T)$, the integrability condition
$$
\int_0^{T'} \mu_{a,b,c}(t)^{\frac{q}{q-1}} \left(w(\alpha(t)) \lambda(t) \right)^{-1/(q-1)} \mathrm{d}t < \infty
$$
is equivalent to
$$
\int_0^{T'} t^{\frac{-a_2q - (1-a_1)\beta + a_1}{q-1}} \mathrm{d}t < \infty.
$$
Clearly, for any $q \in (1,\infty)$, we can ensure that the exponent satisfies
$$
\frac{-a_2q - (1-a_1)\beta + a_1}{q-1} > -1
$$
by choosing $\beta$ sufficiently close to $-1$. 
Therefore, applying Theorem \ref{main thm 6} with the corresponding condition on $f$, we conclude that there exists a unique strong solution to \eqref{20240913 10} that satisfies all the estimates detailed in the theorem.
\end{example}

\begin{remark}
For the anisotropic degenerate model in Example \ref{exam 233}, Theorem \ref{main thm 6} explicitly dictates the permissible range of the source term $f$ and weight exponent $\beta$. Specifically, for $a_1 \le a_2 < 1$, taking $w(t) = t^\beta$ requires selecting $\beta \in (-1, q-1)$ such that $\frac{-a_2 q - (1-a_1)\beta + a_1}{q-1} > -1$. 
This explicit quantitative trade-off illustrates how the severe temporal degeneracy is precisely balanced by the temporal weight $w(t)$ and the weighted norm of $f$
\end{remark}

\begin{remark}
The non-zero initial value problem associated with \eqref{main eqn 1} is generally ill-posed because the coefficients $a^{ij}(t)$ blow up near the initial time $t=0$. To briefly illustrate this, consider the non-zero initial value problem:
$$
\begin{aligned} &u_t(t,x) = a^{ij}(t) u_{x^ix^i}(t,x), \quad (t,x) \in (0,T) \times \fR^d \\ &u(0,x) = u_0. \end{aligned}
$$
Formally, the solution $u$ can be represented either via convolution with a fundamental solution,$$u(t,x) = \int_{\fR^d} p(t,x-y) u_0(y) \mathrm{d}y,
$$
or through a stochastic representation,
$$
u(t,x) = \bE\left[ u_0\left( x + \int_0^t \sqrt{A(s)} \mathrm{dB}_s \right) \right].
$$
Here, $\mathrm{B}_t$ denotes a $d$-dimensional Brownian motion, $\sqrt{A(s)}$ is the positive semi-definite matrix satisfying $\sqrt{A(s)}\sqrt{A(s)} = A(s) = (a^{ij}(s))$, and the heat kernel is given by
$$
p(t,x) = \frac{1}{(2\pi)^{d/2}} \cF^{-1}\left[ \exp\left( - \xi^i\xi^j \int_0^t a^{ij}(s)\mathrm{d}s \right) \right](x).
$$
Since the coefficients $a^{ij}(t)$ fail to be integrable near the initial time $t=0$, both the deterministic integral $\int_0^t a^{ij}(s)\mathrm{d}s$ and the stochastic integral $\int_0^t \sqrt{A(s)} \mathrm{dB}_s$ generally diverge. 
This phenomenon is illustrated concretely in Examples \ref{exam 1} and \ref{exam 2}.
To provide better intuition for why equation \eqref{main eqn 1} is ill-posed when the initial condition $u_0$ is non-zero, we can look at a simplified scenario. 
For convenience, let's restrict the problem to one spatial dimension ($d=1$) and assume that the coefficient $a(t)$ has a divergent integral over the interval:
$$
\int_0^T a(s) \mathrm{d}s = \infty.
$$
Suppose $u$ is a weak solution to the following initial value problem:
$$
u_t(t,x) = a(t) u_{xx}(t,x)
$$
$$
u(0,x) = u_0(x), \quad \text{where } (t,x) \in (0,T) \times \mathbf{R}.
$$
Assume that the mapping $x \mapsto u(t,x)$ belongs to $L_p(\mathbf{R})$ for some $p \in [1,2]$.
By utilizing mollifiers, we can assume without loss of generality that $u$ is spatially smooth and qualifies as a strong solution. 
This allows us to express the solution in its integral form:
$$
u(t,x) = u_0(x) + \int_0^t a(s)u_{xx}(s,x) \mathrm{d}s.
$$
Applying the spatial Fourier transform to both sides yields a first-order ordinary differential equation with respect to time for each frequency $\xi$:
$$
\partial_t \mathcal{F}[u(t,\cdot)](\xi) = -a(t) \xi^2 \mathcal{F}[u(t,\cdot)](\xi) \quad \text{for all } (t,\xi) \in (0,T) \times \mathbf{R}.
$$
If we solve this ODE starting from an arbitrary base time $t_0 \in (0,T)$, we get
$$
\mathcal{F}[u(t,\cdot)](\xi) = \mathcal{F}[u(t_0,\cdot)](\xi) \exp \left( -|\xi|^2 \int_{t_0}^t a(s) \mathrm{d}s \right) \quad \text{for all } t \in (0,T).
$$
Now, taking the limit as $t$ approaches $0$ from above, we find
$$
\mathcal{F}[u_0](\xi) = \mathcal{F}[u(t_0,\cdot)](\xi) \exp \left( |\xi|^2 \int_0^{t_0} a(s) \mathrm{d}s \right) = \infty.
$$
Because the integral of $a(s)$ diverges, the exponential term blows up to infinity. In order to prevent the initial data $\mathcal{F}[u_0](\xi)$ from evaluating to infinity, it forces the condition $\mathcal{F}[u(t_0,\cdot)](\xi) = 0$.
Since the choice of $t_0$ was arbitrary, the solution must be identically zero ($u = 0$), which consequently means $u_0$ must also be zero. 
Therefore, the problem cannot be well-posed for any non-zero initial condition if the coefficients blow up near the initial time.

However, if the coefficients are integrable, a weighted estimate in a Besov space can still be established (\textit{cf}. \cite{CKL 2025}).

\end{remark}

\mysection{It\^o calculus for existence of a solution}
												\label{ito calculus}

Let $(\Omega, \rF, \mathrm{P})$ be a probability space equipped with a filtration $\{\rF_t\}_{t \geq 0}$ satisfying the usual conditions; that is, each $\rF_t$ is a sub-$\sigma$-algebra of $\rF$, the filtration is right-continuous ($\rF_t = \bigcap_{s>t} \rF_s$), and $\rF_0$ contains all $\mathrm{P}$-null sets. 
We assume the existence of a $d$-dimensional Brownian motion (or Wiener process) $\mathrm{B}_t = (\mathrm{B}^1_t, \ldots, \mathrm{B}^d_t)$ that is adapted to $\{\rF_t\}_{t \geq 0}$, meaning $\mathrm{B}_t$ is $\rF_t$-measurable for every $t \geq 0$.
For a comprehensive background on It\^o calculus, we refer readers to standard texts such as \cite{KS 2014, Krylov 1995, Krylov 2002, RY 1999}. 
However, we emphasize that our current framework does not require the full generality of It\^o integrals with predictable or progressively measurable random integrands.
 Instead, our analysis is strictly confined to It\^o integrals with deterministic integrands.
For any deterministic $d \times d$-matrix-valued function $\sigma(t) = (\sigma^{ij}(t))$ satisfying
$$
\int_0^\infty |\sigma(t)|^2 \mathrm{d}t < \infty,
$$
the It\^o stochastic integral $\int_0^\infty \sigma(t) \mathrm{dB}_t$ is well-defined as an element of $\mathrm{L}_2(\Omega; \fR^d)$ via the It\^o isometry. 
This integral is an $\fR^d$-valued random variable whose $i$-th component is given by the scalar It\^o integral $\sum_{j=1}^d \int_0^\infty \sigma^{ij}(t) \mathrm{dB}^j_t$.
Furthermore, if $\sigma$ is locally square-integrable such that
$$
\int_s^t |\sigma(r)|^2 \mathrm{d}r < \infty \quad \text{for all } 0 < s < t < T,
$$
we define the truncated It\^o integral for $0 < s \leq t < T$ as
$$
\int_s^t \sigma(r) \mathrm{dB}_r := \int_0^\infty \mathbf{1}_{(s,t)}(r) \sigma(r) \mathrm{dB}_r.
$$
By Doob's martingale inequality, this stochastic integral admits a continuous modification. Specifically, for any fixed $s \in (0,T)$ and almost every $\omega \in \Omega$, the sample path mapping $t \in [s, \infty) \mapsto \int_s^t \sigma(r) \mathrm{dB}_r(\omega)$ is a continuous trajectory.
The primary methodological tool in our analysis is It\^o's formula, which requires functions to be twice continuously differentiable. 
Therefore, our strategy is to first establish the definition of a classical solution and apply It\^o's formula directly to it. 
Following standard techniques in the literature, we will then construct our general weak solutions by approximating them with sequences of these classical solutions.

\begin{definition}[Classical solution]
										\label{classic sol}
Let $u$ be a locally integrable function on $[0,T) \times \fR^d$.
We say that $u$ is a classical solution to \eqref{main eqn 1} if for all $(t,x) \in [0,T) \times \fR^d$,
\begin{align}
										\label{classical equal}
u(t,x) = \int_0^t \left(a^{ij}(s) u_{x^ix^j}(s,x) + b^i(s)u_{x^i}(s,x) + c(s)u(s,x)+f(t,x) \right) \mathrm{d}s.
\end{align}
Here $u_{x^ix^j}$ and $u_{x^i}$ are classical derivatives, \textit{i.e.} 
\begin{align*}
u_{x^i}(t,x)= \lim_{h \downarrow 0} \frac{ u(t,x+h e_i) - u(t,x)}{h},
\end{align*}
where $e_i$ is the standard basis vector in $\fR^d$ whose $i$-th coordinate is 1 and the other coordinates are zero.
\end{definition}

\begin{remark}
It is necessary to clarify a subtle distinction between our concept of a classical solution and the conventional definition, particularly regarding differentiability in the temporal variable.
Suppose $u$ is a classical solution to \eqref{main eqn 1} according to our definition. Because \eqref{classical equal} guarantees that $u$ is absolutely continuous with respect to $t$, the Fundamental Theorem of Calculus applies in the temporal variable. 
Consequently, the Lebesgue differentiation theorem ensures that $u$ is differentiable with respect to $t$ almost everywhere on $(0,T)$, yielding
\begin{align}
										\label{20240923 01}
u_t(t,x) = a^{ij}(t) u_{x^ix^j}(t,x) + b^i(t)u_{x^i}(t,x) + c(t)u(t,x)
\end{align}
for all $x \in \fR^d$.
However, without assuming that the coefficients are continuous, we generally cannot expect this equation to hold for every $t \in (0,T)$. 
Because the coefficients are highly general, the right-hand side of \eqref{20240923 01} may not satisfy Darboux's property (the intermediate value property for derivatives). 
By Darboux's theorem, this implies that no function $u$ can possess an everywhere-defined time derivative that satisfies \eqref{20240923 01} pointwise on all of $(0,T) \times \fR^d$.
Therefore, our formulation in Definition \ref{classic sol} represents the highest level of regularity achievable given the inherent irregularity of the coefficients. 
For this reason, it is thoroughly justified to refer to it as a classical solution to \eqref{main eqn 1}.
\end{remark}
\begin{remark}
										\label{classical sol derivative}
Even when considering a classical solution $u$ to \eqref{main eqn 1}, it is impossible to guarantee that the classical derivatives $u_{x^ix^j}(t,x)$ and $u_{x^i}(t,x)$ exist uniformly for every $t \in (0,T)$ and $x \in \fR^d$. 
This limitation arises mainly because the coefficients $a^{ij}(t)$ and $b^i(t)$ are permitted to vanish on certain subsets of $(0,T)$, which potentially disrupts the regularity of the solution.
\end{remark}

\begin{remark}
									\label{classic weak sol}
Suppose that for all $t \in (0,T)$ and $r \in (0,\infty)$,
\begin{align*}
\int_0^t \int_{|x|<r} \left(|a^{ij}(s)| |u_{x^ix^j}(s,x)|  +|b^i(s)| |u_{x^i}(s,x)| + |c(s)| |u(s,x)|  \right) \mathrm{d}x \mathrm{d}s < \infty.
\end{align*}
Due to the  local integrability, any existing classical derivatives coincide with their corresponding Sobolev derivatives. Consequently, every classical solution $u$ to \eqref{main eqn 1} is automatically a strong solution. Furthermore, as discussed in Remark \ref{strong weak sol}, any such classical solution simultaneously qualifies as a weak solution.
\end{remark}

Recall that the function defined by
\begin{align}
									\label{20240923 10}
u(t,x) := \int_0^t  \mathrm{e}^{\int_s^t c(r)\mathrm{d}r}  \bE\left[f\left(s,x + X_{s,t}\right)  \right] \mathrm{d}s
\end{align}
typically constitutes a classical solution to \eqref{main eqn 1}, where $X_{s,t}=\int_s^tb(r)\mathrm{d}r+\int_s^t \sigma(r) \mathrm{dB}_r$,  $\sigma(r)= \sqrt{2} \sqrt{A(r)}$, $b(t)=\left(b^1(t),\ldots, b^d(t)\right)$, and $A(r)= (a^{ij}(r))$  (\textit{cf.} \cite{Krylov 1995,KK 2018,BKRS 2022}).  

However, our framework allows all coefficients to be unbounded or even non-integrable, a rigorous justification is required to show that  expression \eqref{20240923 10} still provides a classical solution to equation \eqref{main eqn 1}.

\begin{theorem}[Existence of a classical solution]
										\label{classic existence solution}
Assume that Assumptions \ref{main as 1} - \ref{main as 3} are satisfied. 
Furthermore, suppose the leading coefficients satisfy the following local integrability condition for any $0 < s \leq t < T$,
\begin{align}
									\label{20240909 01}
\int_s^t \mathrm{e}^{\int_s^r c(\rho)\mathrm{d}\rho}\max_{i,j}|a^{ij}(r)| \mathrm{d}r < \infty.
\end{align}
Let $f$ be a locally integrable source function on $[0,T) \times \fR^d$ that fulfills the following regularity and weighted integrability criteria:
\begin{enumerate}[(i)]
\item{(Spatial Smoothness)} For every $t \in (0,T)$, the spatial mapping $x \mapsto f(t,x)$ is twice continuously differentiable.
\item{(Zeroth-Order Integrability)} For any truncation time $T' \in (0,T)$,
$$
  \int_0^{T'} \|f(t,\cdot)\|_{\mathrm{L}_\infty(\fR^d)}\mathrm{d}t + \int_0^{T'} |c(t)| \int_0^t \mathrm{e}^{\int_s^t c(r)\mathrm{d}r}\|f(s,\cdot)\|_{\mathrm{L}_\infty(\fR^d)}\mathrm{d}s \, \mathrm{d}t < \infty.
$$
\item{(First-Order Integrability)} 
For any $T' \in (0,T)$, the double integral involving the drift coefficients and the gradient is finite:
$$
  \int_0^{T'} |b^i(t)| \int_0^t \mathrm{e}^{\int_s^t c(r)\mathrm{d}r}\|f_{x^i}(s,\cdot)\|_{\mathrm{L}_\infty(\fR^d)}\mathrm{d}s \, \mathrm{d}t < \infty.
$$
\item{(Pointwise Gradient Bound)}
 The spatial gradient is essentially bounded in time, meaning $\|f_x(t,\cdot)\|_{\mathrm{L}_\infty(\fR^d)} < \infty$ for almost every $t \in (0,T)$.
\item{(Second-Order Integrability)} For any $T' \in (0,T)$, the double integral involving the diffusion coefficients and the Hessian is finite:
$$
  \int_0^{T'} |a^{ij}(t)| \int_0^t \mathrm{e}^{\int_s^t c(r)\mathrm{d}r} \|f_{x^ix^j}(s,\cdot)\|_{\mathrm{L}_\infty(\fR^d)}\mathrm{d}s \, \mathrm{d}t < \infty.
$$
\end{enumerate}
Under these conditions, the function $u$ explicitly defined in \eqref{20240923 10} constitutes a classical solution to equation \eqref{main eqn 1}.

\end{theorem}
\begin{proof}
We divide proof into several parts.

\vspace{2mm}
{\bf I. Applying Itô's Formula and the Martingale Property}
\vspace{2mm}

To begin, let $x \in \fR^d$ and fix a time $s \in (0,T)$ such that the spatial gradient is essentially bounded
$$
\|f_x(s,\cdot)\|_{\mathrm{L}_\infty(\fR^d)} < \infty.
$$
By applying Itô's formula to the exponentially weighted process, we obtain the following expansion for any $t \in [s, T)$:
\begin{align}
										\notag
&\mathrm{e}^{\int_s^t c(r)\mathrm{d}r}f\left(s,x + X_{s,t} \right)  \\
										\notag
&= \int_s^t c(r) \mathrm{e}^{\int_s^r c(\rho)\mathrm{d}\rho} f\left(s,x + X_{s,r} \right) \mathrm{d}r \\
										\notag
&\quad + f(s,x) + \int_s^t \mathrm{e}^{\int_s^r c(\rho)\mathrm{d}\rho} b^i(r)f_{x^i}\left(s,x+X_{s,r}\right)  \mathrm{d}r \\
										\notag
&\quad + \int_s^t \mathrm{e}^{\int_s^r c(\rho)\mathrm{d}\rho} f_{x^i}\left(s,x+X_{s,r}\right) \sigma^{ij}(r) \mathrm{dB}^j_r \\
										\label{20240903 01}
&\quad + \int_s^t \mathrm{e}^{\int_s^r c(\rho)\mathrm{d}\rho} a^{ij}(r) f_{x^ix^j}\left(s,x+X_{s,r} \right) \mathrm{d}r.
\end{align}
Based on condition (iv) and estimate \eqref{20240909 01}, we can naturally bound the quadratic variation of the stochastic integral
\begin{align*}
&\int_s^t \left| \mathrm{e}^{\int_s^r c(\rho)\mathrm{d}\rho}  f_{x^i}(s,x+X_{s,r}) \sigma^{ij}(r)\right|^2 \mathrm{d}r \\
&\lesssim  \|f_{x}(s,\cdot)\|^2_{\mathrm{L}_\infty(\fR^d)}
\int_s^t \mathrm{e}^{2\int_s^r c(\rho)\mathrm{d}\rho} \max_{i,j}|a^{ij}(r)|\mathrm{d}r
< \infty.
\end{align*}
Because this integral is finite, the stochastic integral qualifies as a true martingale, which guarantees that its expected value vanishes, \textit{i.e.}
$$
\bE\left[\int_s^t  \mathrm{e}^{\int_s^r c(\rho)\mathrm{d}\rho} f_{x^i}\left(s,x+X_{s,r}\right) \sigma^{ij}(r) \mathrm{dB}^j_r\right] = 0.
$$

\vspace{2mm}
{\bf II. Deriving the Expected Value Equation}
\vspace{2mm}

Taking the expectation of both sides of \eqref{20240903 01} and utilizing Fubini's theorem to commute the expectation with the time integrals, we arrive at the following relation. 
By assumptions $(i)\sim (v)$ in the theorem, for almost $s\in (0,T)$ and for all $t\in [s,T)$,
\begin{align}
										\notag
\mathrm{e}^{\int_s^t c(r)\mathrm{d}r}\bE\left[ f\left(s,x + X_{s,t} \right) \right]
&=f(s,x) 
+\int_s^t c(r) \mathrm{e}^{\int_s^r c(\rho)\mathrm{d}\rho} \bE\left[ f\left(s,x + X_{s,r} \right) \right]  \mathrm{d}r  \\
										\notag
&+ \int_s^t b^i(r) \mathrm{e}^{\int_s^r c(\rho)\mathrm{d}\rho} \bE\left[f_{x^i}\left(s,x+X_{s,r}\right)\right] \mathrm{d}r \\
										\label{20240903 02}
&+ \int_s^t a^{ij}(r) \mathrm{e}^{\int_s^r c(\rho)\mathrm{d}\rho} \bE\left[f_{x^ix^j}\left(s,x+X_{s,r} \right) \right] \mathrm{d}r.
\end{align}

\vspace{2mm}
{\bf III. Integration and Recovering the Classical Solution}
\vspace{2mm}

For any fixed $t \in (0,T)$, the identity in \eqref{20240903 02} remains valid for almost every $s \in (0,t]$. By integrating both sides from $0$ to $t$ with respect to $s$, and invoking Fubini's theorem a second time to swap the order of integration, we find
\begin{align}
										\notag
u(t,x) 
&= \int_0^t \mathrm{e}^{\int_s^t c(r)\mathrm{d}r}\bE\left[ f\left(s,x + X_{s,t} \right) \right]
\mathrm{d}s \\
										\notag
&=\int_0^t f(s,x) \mathrm{d}s 
+\int_0^t\int_s^t c(r) \mathrm{e}^{\int_s^r c(\rho)\mathrm{d}\rho} \bE\left[ f\left(s,x + X_{s,r} \right) \right]  \mathrm{d}r \mathrm{d}s \\
										\notag
&\quad +\int_0^t \int_s^t b^i(r) \mathrm{e}^{\int_s^r c(\rho)\mathrm{d}\rho} \bE\left[f_{x^i}\left(s,x+X_{s,r}\right)\right] \mathrm{d}r \mathrm{d}s \\
										\notag
&\quad +\int_0^t \int_s^t a^{ij}(r) \mathrm{e}^{\int_s^r c(\rho)\mathrm{d}\rho} \bE\left[f_{x^ix^j}\left(s,x+X_{s,r} \right) \right] \mathrm{d}r \mathrm{d}s \\
										\notag
&=\int_0^t f(s,x) \mathrm{d}s 
+\int_0^tc(r) \int_0^r  \mathrm{e}^{\int_s^r c(\rho)\mathrm{d}\rho} \bE\left[ f\left(s,x + X_{s,r} \right) \right] \mathrm{d}s \mathrm{d}r\\
										\notag
&\quad +\int_0^t  b^i(r)  \int_0^r \mathrm{e}^{\int_s^r c(\rho)\mathrm{d}\rho} \bE\left[f_{x^i}\left(s,x+X_{s,r}\right)\right] \mathrm{d}s \mathrm{d}r  \\
										\notag
&\quad +\int_0^t a^{ij}(r) \int_0^r  \mathrm{e}^{\int_s^r c(\rho)\mathrm{d}\rho}  \bE\left[f_{x^ix^j}\left(s,x+X_{s,r} \right) \right] \mathrm{d}s \mathrm{d}r  \\
										\label{20240907 10}
&=\int_0^t \left( a^{ij}(r) u_{x^ix^j}(r,x) +b^i(r)u_{x^i}(r,x)+ c(r)u(r,x)+ f(r,x) \right) \mathrm{d}r.
\end{align}
This sequence of equalities demonstrates that $u$ is a classical solution to equation \eqref{main eqn 1}.

\vspace{2mm}
{\bf IV. Rigorous Justification of Fubini's Theorem}
\vspace{2mm}

To make this derivation completely rigorous, we must formally justify our use of Fubini's theorem by verifying that the corresponding absolute integrals are finite. Applying Tonelli's theorem alongside our initial assumptions on $f$, we can bound the terms as follows
\begin{align*}
&\int_0^t\int_s^t |c(r) \mathrm{e}^{\int_s^r c(\rho)\mathrm{d}\rho} \bE\left[ f\left(s,x + X_{s,r} \right) \right]|  \mathrm{d}r \mathrm{d}s  \\
&\leq \int_0^t\int_s^t \left|c(r)\right| \mathrm{e}^{\int_s^r c(\rho)\mathrm{d}\rho} \|f(s,\cdot)\|_{\mathrm{L}_\infty(\fR^d)} \mathrm{d}r \mathrm{d}s  \\
&= \int_0^t  \left|c(r)\right| \int_0^r \mathrm{e}^{\int_s^r c(\rho)\mathrm{d}\rho}\|f(s,\cdot)\|_{\mathrm{L}_\infty(\fR^d)}\mathrm{d}s \mathrm{d}r <\infty,
\end{align*}$$$$\begin{align*}
&\int_0^t\int_s^t \left|b^i(r) \mathrm{e}^{\int_s^r c(\rho)\mathrm{d}\rho} \bE\left[f_{x^i}(s,x+X_{s,r})\right] \right|\mathrm{d}r \mathrm{d}s   \\
&\leq \int_0^t  \left|b^i(r)\right| \int_0^r \mathrm{e}^{\int_s^r c(\rho)\mathrm{d}\rho}\|f_{x^i}(s,\cdot)\|_{\mathrm{L}_\infty(\fR^d)}\mathrm{d}s \mathrm{d}r 
<\infty,
\end{align*} and
\begin{align*}
&\int_0^t\int_s^t \left|a^{ij}(r)\bE\left[f_{x^ix^j}(r,x+X_{s,r})\right] \right|\mathrm{d}r \mathrm{d}s   \\
&\lesssim \int_0^t  \left|a^{ij}(r)\right| \int_0^r \mathrm{e}^{\int_s^r c(\rho)\mathrm{d}\rho} \|f_{x^ix^j}(s,\cdot)\|_{\mathrm{L}_\infty(\fR^d)}\mathrm{d}s \mathrm{d}r 
<\infty.  
\end{align*}
Finally, the initial source integral $\int_0^t f(s,x)\mathrm{d}s$ is guaranteed to be well-defined because
\begin{align*}
\int_0^t |f(s,x)|\mathrm{d}s 
&\lesssim \int_0^t\|f(s,\cdot)\|_{\mathrm{L}_\infty(\fR^d)}\mathrm{d}s < \infty.
\end{align*}
Since all of these bounding integrals are finite, $u(t,x)$ is rigorously well-defined for all $(t,x) \in (0,T)$ as established by the relationships in \eqref{20240907 10}.

\end{proof}

\begin{remark}
Recall
$X_{s,t}=\int_s^tb(r)\mathrm{d}r+\int_s^t \sigma(r) \mathrm{dB}_r$, $\sigma(r)= \sqrt{2} \sqrt{A(r)}$, and $A(r)= (a^{ij}(r))$,
and
\begin{align*}
u(t,x) :=  \int_0^t  \mathrm{e}^{\int_s^t c(r)\mathrm{d}r} \bE\left[f\left(s,x+ X_{s,t}\right)  \right] \mathrm{d}s.
\end{align*}
Put
\begin{align*}
\tilde f (t,x) = \mathrm{e}^{-\int_0^t c(s)\mathrm{d}s} f\left(t,x-\int_0^tb(r)\mathrm{d}r\right)
\end{align*}
and
\begin{align*}
v(t,x) :=  \int_0^t   \bE\left[\tilde f\left(s,x+ \int_s^t \sigma(r) \mathrm{dB}_r\right)  \right] \mathrm{d}s.
\end{align*}
Then $v$ is a classical solution to 
\begin{align*}
&u_t(t,x)=a^{ij}(t) u_{x^ix^i}(t,x)+\tilde f(t,x),  \\
&u(0,x)=0, \qquad (t,x) \in (0,T) \times \fR^d
\end{align*}
and
\begin{align*}
u(t,x)= \mathrm{e}^{\int_0^t c(s)\mathrm{d}s} v\left(t,x + \int_0^tb(r)\mathrm{d}r \right).
\end{align*}
It might appear that this observation facilitates the demonstration of Theorem \ref{classic existence solution} by initially posing that $b^i(t)=c(t)=0$ for all $t \in (0,t)$ and $i=1,\ldots,d$. 
Nevertheless, this approach necessitates the imposition of an additional condition, specifically, that $b^i(t)$ is integrable near $0$ in order to ensure the well-definedness of $\int_0^tb(r)\mathrm{d}r$.
\end{remark}

\begin{remark}
We now clarify why conditions $(iii)$ and $(iv)$ are stated separately in the preceding theorem.
Assuming $(iii)$ holds, we have
$$
|b^i(t)| \int_0^t \mathrm{e}^{\int_s^t c(r)\mathrm{d}r}\|f_{x^i}(s,\cdot)\|_{\mathrm{L}_\infty(\fR^d)}\mathrm{d}s < \infty \quad \text{for a.e. } t \in (0,T).
$$
Provided that $|b^i(t)| \neq 0$ a.e. on $(0,T)$, this naturally yields
$$
\|f_{x^i}(t,\cdot)\|_{\mathrm{L}_\infty(\fR^d)} < \infty \quad \text{for a.e. } t \in (0,T).
$$
However, if $|b^i(t)| = 0$ on a set of positive measure, this deduction is no longer valid. 
Because measure theory frequently employs the convention $0 \cdot \infty = 0$, the $L_\infty$-norm could theoretically be infinite while the product remains zero. 
To prevent this ambiguity and ensure finiteness, condition $(iv)$ is introduced as an independent requirement.
\end{remark}

\begin{remark}
As demonstrated in Theorem \ref{classic existence solution}, we cannot guarantee that the solution $u$ to equation \eqref{main eqn 1} is universally twice continuously differentiable with respect to the spatial variable on $(0,T) \times \fR^d$. 
As Remark \ref{classical sol derivative} points out, this lack of spatial $C^2$-regularity fundamentally stems from the possible degeneracy of the system's coefficients. 
Consequently, a smooth source term $f$ is insufficient on its own to ensure this level of regularity.
Mathematically, this restriction occurs because applying Fubini's theorem to the integral representation of $u$ in \eqref{20240923 10} requires a stringent integrability condition on the spatial derivatives of $f$, specifically,
\begin{align}
									\label{20240923 11}
\int_0^t \mathrm{e}^{\int_s^t c(r)\mathrm{d}r} \|f_{x^ix^j}(s, \cdot)\|_{\mathrm{L}_\infty(\fR^d)} \mathrm{d}s < \infty.
\end{align}
While condition \eqref{20240923 11} naturally holds for times $t$ where $a^{ij}(t) \neq 0$ (per assumption $(v)$ of Theorem \ref{classic existence solution}), it may fail elsewhere. 
Nonetheless, even when this bound is not globally satisfied for all $t$, the solution $u$ still preserves several highly beneficial properties. 
The subsequent corollary outlines these specific characteristics, relying strictly on the foundational assumptions from Theorem \ref{classic existence solution} that ensure the solution's existence.
\end{remark}

\begin{corollary}[Properties of a solution]
											\label{property classical solution}
Assuming the conditions of Theorem \ref{classic existence solution} are satisfied, let $u$ denote the classical solution to \eqref{main eqn 1}, explicitly defined by the probabilistic representation:
$$
u(t,x) := \int_0^t \mathrm{e}^{\int_s^t c(r)\mathrm{d}r} \bE\left[f\left(s,x+ X_{s,t}\right)\right] \mathrm{d}s.
$$This solution $u$ exhibits the following characteristics
\begin{enumerate}[(i)]
\item {(Supremum Bound)} For any $T' \in (0,T)$, the temporal supremum of the $\mathrm{L}_\infty$-norm is controlled by the source term and its derivatives:
\begin{align*}
\sup_{t \in [0,T']} \|u(t,\cdot)\|_{\mathrm{L}_\infty(\fR^d)}
&\leq \int_0^{T'} |a^{ij}(t)| \int_0^t \mathrm{e}^{\int_s^t c(r)\mathrm{d}r} \|f_{x^ix^j}(s,\cdot)\|_{\mathrm{L}_\infty(\fR^d)}\mathrm{d}s \mathrm{d}t  \\
&\quad +\int_0^{T'} |b^i(t)| \int_0^t \mathrm{e}^{\int_s^t c(r)\mathrm{d}r} \|f_{x^i}(s,\cdot)\|_{\mathrm{L}_\infty(\fR^d)}\mathrm{d}s \mathrm{d}t \\
&\quad +\int_0^{T'} |c(t)| \int_0^t \mathrm{e}^{\int_s^t c(r)\mathrm{d}r} \|f(s,\cdot)\|_{\mathrm{L}_\infty(\fR^d)}\mathrm{d}s , \mathrm{d}t.
\end{align*}

\item {(Absolute Continuity in Time)} For any fixed spatial point $x \in \fR^d$ and any $T' \in (0,T)$, the temporal mapping $t \mapsto u(t,x)$ is absolutely continuous on $[0,T']$.

\item {(Weighted Spatial Derivative Bounds)}  For every $T' \in (0,T)$, the temporal integrals of the weighted $\mathrm{L}_\infty$-norms of $u_{x^ix^j}$, $u_{x^i}$, and $u$ satisfy the respective estimates
$$
\int_0^{T'} |a^{ij}(t)| \|u_{x^ix^j}(t,\cdot)\|_{\mathrm{L}_\infty(\fR^d)} \mathrm{d}t \leq \int_0^{T'} |a^{ij}(t)| \int_0^t \mathrm{e}^{\int_s^t c(r)\mathrm{d}r} \|f_{x^ix^j}(s,\cdot)\|_{\mathrm{L}_\infty(\fR^d)}\mathrm{d}s \, \mathrm{d}t,
$$
$$
\int_0^{T'} |b^i(t)| \|u_{x^i}(t,\cdot)\|_{\mathrm{L}_\infty(\fR^d)} \mathrm{d}t \leq \int_0^{T'} |b^i(t)| \int_0^t \mathrm{e}^{\int_s^t c(r)\mathrm{d}r} \|f_{x^i}(s,\cdot)\|_{\mathrm{L}_\infty(\fR^d)}\mathrm{d}s \, \mathrm{d}t,
$$
$$\int_0^{T'} |c(t)| \|u(t,\cdot)\|_{\mathrm{L}_\infty(\fR^d)} \mathrm{d}t \leq \int_0^{T'} |c(t)| \int_0^t \mathrm{e}^{\int_s^t c(r)\mathrm{d}r} \|f(s,\cdot)\|_{\mathrm{L}_\infty(\fR^d)}\mathrm{d}s \, \mathrm{d}t.
$$
\end{enumerate}
Furthermore, if we assume an additional integrability condition for a given $p \in [1,\infty]$, namely that $\int_0^t \mathrm{e}^{\int_s^t c(r)\mathrm{d}r} \|f(s,\cdot)\|_{\mathrm{L}_p(\fR^d)}\mathrm{d}s < \infty$ for all $t \in (0,T)$, then the $\mathrm{L}_p$-norm of the solution satisfies the bound:
$$
\|u(t,\cdot)\|_{\mathrm{L}_p(\fR^d)} \leq \int_0^{t} \mathrm{e}^{\int_s^t c(r)\mathrm{d}r} \|f(s,\cdot)\|_{\mathrm{L}_p(\fR^d)}\mathrm{d}s < \infty \quad \text{for all } t \in (0,T).
$$

\end{corollary}
\begin{proof}
While the validity of these properties is mostly intuitive, a concise justification is necessary to ensure rigor. 
We can verify each condition as follows:
\begin{itemize}
\item{Estimates in $(iv)$:} These inequalities, along with the supplementary $\mathrm{L}_p$-estimate, follow directly from the integral representation of $u$ and our initial hypotheses on the source term $f$. 
We rigorously confirm the spatial differentiability of $u$ by applying the Mean Value Theorem and the Lebesgue Dominated Convergence Theorem with our established bounds—provided the relevant coefficients are non-zero, as noted in the preceding remark.

\item{Supremum Bound $(i)$:} Because $u$ satisfies \eqref{main eqn 1} as a classical solution, the supremum estimate in $(i)$ emerges naturally as a direct consequence of the results established in $(iv)$.

\item{Absolute Continuity $(iii)$:} This property holds because the classical nature of \eqref{main eqn 1}—arising from our definition of classical solutions—combined with the estimates derived in $(iv)$, guarantees that the solution is absolutely continuous with respect to time.

\item{Additional comment:} This final assertion is a straightforward consequence of the generalized Minkowski inequality.
\end{itemize}
\end{proof}

We now proceed to derive weak solutions from their classical counterparts.

\begin{theorem}[Existence of a weak solution]
									\label{stochastic weak existence}
Let $p \in [1,\infty]$ and suppose $f$ is a locally integrable function on  $[0,T) \times \fR^d$. 
We assume that Assumptions \ref{main as 1} -  \ref{main as 3} are satisfied, along with the following three specific integrability constraints:
\begin{enumerate}[(i)]
\item For any $0<s<t<T$,
\begin{align*}
\int_s^t \mathrm{e}^{\int_s^r c(\rho)\mathrm{d}\rho}\max_{i,j}|a^{ij}(r)| \mathrm{d}r < \infty.
\end{align*}
\item For any $t \in (0,T)$,
\begin{align*}
\int_0^t \mathrm{e}^{\int_s^t |c(r)|\mathrm{d}r} \|f(s,\cdot)\|_{\mathrm{L}_p(\fR^d)}\mathrm{d}s < \infty.
\end{align*}
\item For any $T' \in (0,T)$,
\begin{align*}
\int_0^{T'}  \mu_{a,b,c}(t) \int_0^t \mathrm{e}^{\int_s^t c(r)\mathrm{d}r}\|f(s,\cdot)\|_{\mathrm{L}_p(\fR^d)}\mathrm{d}s \mathrm{d}t  <\infty.
\end{align*}
\end{enumerate}
Under these prerequisites, the probabilistically defined function
\begin{align}
									\label{20240909 50}
u(t,x) := \int_0^t  \mathrm{e}^{\int_s^t c(r)\mathrm{d}r}  \bE\left[f\left(s,x + X_{s,t}\right)  \right] \mathrm{d}s
\end{align}
constitutes a valid weak solution to equation \eqref{main eqn 1}. 
Furthermore, for all $t \in (0,T)$, this solution $u$ obeys the following pointwise and weighted integral bounds
\begin{align}
										\label{20240909 82}
 \|u(t,\cdot)\|_{\mathrm{L}_p(\fR^d)}
\leq \int_0^{t} \mathrm{e}^{\int_s^t c(r)\mathrm{d}r} \|f(s,\cdot)\|_{\mathrm{L}_p(\fR^d)}\mathrm{d}s
\end{align}
and
\begin{align}
										\notag
&\int_0^{t} \mu_{a,b,c}(s) \|u(s,\cdot)\|_{\mathrm{L}_p(\fR^d)} \mathrm{d}s \\
										\label{20240909 51}
&\leq \int_0^{t} \mu_{a,b,c}(s) \int_0^s \mathrm{e}^{\int_r^s c(\rho)\mathrm{d}\rho} \|f(r,\cdot)\|_{\mathrm{L}_p(\fR^d)}\mathrm{d}r \mathrm{d}s.
\end{align}
Finally, if the zeroth-order coefficient $c(t)$ is locally integrable on $[0,T)$ meaning that
\begin{align}
										\label{20250826 01}
\int_0^{T'} |c(t)| \mathrm{d}t < \infty \quad \forall T' \in (0,T),
\end{align}
then for any $T' \in (0,T)$, the map $t \mapsto \|u(t,\cdot)\|_{\mathrm{L}_p(\fR^d)}$ is absolutely continuous on $[0,T']$ and the following supremum bound holds:
\begin{align}
										\label{20250826 02}
\sup_{t \in [0,T']} \mathrm{e}^{-\int_0^t c(r)\mathrm{d}r} \|u(t,\cdot)\|_{\mathrm{L}_p(\fR^d)}
\leq \int_0^{T'} \mathrm{e}^{-\int_0^s c(r)\mathrm{d}r} \|f(s,\cdot)\|_{\mathrm{L}_p(\fR^d)}\mathrm{d}s.
\end{align}

\end{theorem}
\begin{proof}
Applying the generalized Minkowski inequality to the definition of $u$ in \eqref{20240909 50} directly yields the estimates \eqref{20240909 82} and \eqref{20240909 51}. 
Therefore, it suffices to show that $u$ acts as a weak solution to \eqref{main eqn 1}. 
We proceed by a standard regularization argument using a Sobolev mollifier. 
Let $\varphi \in \mathrm{C}_c^\infty(\fR^d)$ be a non-negative bump function with unit integral, and for $\varepsilon > 0$, define the scaled mollifier $\varphi^{(\varepsilon)}(x) = \varepsilon^{-d} \varphi(x/\varepsilon)$. 
We construct the smooth approximations $u^{(\varepsilon)}(t,x)$ and $f^{(\varepsilon)}(t,x)$ by taking the spatial convolutions of $u$ and $f$ with $\varphi^{(\varepsilon)}$, respectively.
For the sake of completeness, we briefly recall several fundamental properties of these mollified functions without providing their standard proofs:
\begin{itemize}
\item{(Pointwise Convergence)} 
As $\varepsilon \downarrow 0$, the regularized functions $u^{(\varepsilon)}(t,x)$ and $f^{(\varepsilon)}(t,x)$ converge to $u(t,x)$ and $f(t,x)$, respectively, for almost every $(t,x) \in (0,T) \times \fR^d$.
\item{(Smoothness)} For any fixed $\varepsilon > 0$ and $t \in (0,T)$, the spatial mapping $x \mapsto f^{(\varepsilon)}(t,x)$ is infinitely differentiable.
\item {(Derivative Bounds)} For any $d$-dimensional multi-index $\alpha$, the spatial derivatives of the mollified source term satisfy the following bound, which depends on the $\mathrm{L}_p$-norm of the original function:$$\|D^\alpha_x f^{(\varepsilon)}(t,\cdot)\|_{\mathrm{L}_\infty(\fR^d)} \lesssim_{\varepsilon,\alpha} \|f(t,\cdot)\|_{\mathrm{L}_{p}(\fR^d)}.$$
\item{(Strong Convergence)} For each fixed time $t \in (0,\infty)$, the approximations converge strongly to their original counterparts in the $\mathrm{L}_p$-sense as the smoothing parameter vanishes:$$\lim_{\varepsilon \downarrow 0}f^{(\varepsilon)}(t,\cdot) = f(t,\cdot) \quad \text{and} \quad \lim_{\varepsilon \downarrow 0}u^{(\varepsilon)}(t,\cdot) = u(t,\cdot) \quad \text{in } \mathrm{L}_p(\fR^d).$$
\end{itemize}

We assert that for each fixed $\varepsilon \in (0,\infty)$, the regularized source function $(t,x) \mapsto f^{(\varepsilon)}(t,x)$ fulfills all the prerequisites demanded of $f$ in Theorem \ref{classic existence solution}. 
It is readily verified by combining our initial assumptions on $f$ and $a^{ij}$ with the standard derivative bounds for mollifiers. 
Furthermore, an application of Fubini's theorem confirms that the convolution $u^{(\varepsilon)}(t,x)$ corresponds to the probabilistic representation associated with the smoothed source $f^{(\varepsilon)}$, \textit{i.e.}
$$
u^{(\varepsilon)}(t,x) = \int_0^t \mathrm{e}^{\int_s^t c(r)\mathrm{d}r} \bE\left[f^{(\varepsilon)}\left(s,x + X_{s,t}\right)\right] \mathrm{d}s.
$$
Consequently, by invoking Theorem \ref{classic existence solution}, $u^{(\varepsilon)}$ serves as the classical solution to the modified equation. 
That is,
$$
u^{(\varepsilon)}(t,x) = \int_0^t \left(a^{ij}(s) u^{(\varepsilon)}_{x^ix^j}(s,x) + b^i(s)u^{(\varepsilon)}_{x^i}(s,x) + c(s)u^{(\varepsilon)}(s,x) + f^{(\varepsilon)}(s,x) \right) \mathrm{d}s
$$
and it inherits the properties detailed in Corollary \ref{property classical solution}.

Given these classical properties and applying Fubini's theorem, it is straightforward to demonstrate that $u^{(\varepsilon)}$ also satisfies the weak formulation of \eqref{main eqn 1} with the smoothed source term $f^{(\varepsilon)}$ because of integration by parts. 
Specifically, for any test function $\varphi \in \mathrm{C}_c^\infty(\fR^d)$ and almost every $t \in (0,T)$, the following identity holds:
\begin{align*}
(u^{(\varepsilon)}(t,\cdot), \varphi)_{\mathrm{L}_2(\fR^d)} 
&= \int_0^t \left(u^{(\varepsilon)}(s,\cdot), a^{ij}(s)\varphi_{x^ix^j} -b^i(s)\varphi_{x^i}+c(s) \varphi\right)_{\mathrm{L}_2(\fR^d)} ds \\ 
&\quad +\int_0^t\left(f^{(\varepsilon)}(s,\cdot), \varphi \right)_{\mathrm{L}_2(\fR^d)} ds \quad \text{a.e.}~ t \in (0,T).
\end{align*}
Finally, we pass to the limit as $\varepsilon \downarrow 0$. Utilizing the strong convergence properties of the mollifiers alongside the integrability conditions on $f$, the Lebesgue Dominated Convergence Theorem yields the desired relation for any test function $\varphi \in \mathrm{C}_c^\infty(\fR^d)$:
\begin{align*}
(u(t,\cdot), \varphi)_{\mathrm{L}_2(\fR^d)} 
&= \int_0^t \left(u(s,\cdot), a^{ij}(s)\varphi_{x^ix^j} -b^i(s)\varphi_{x^i}+c(s) \varphi\right)_{\mathrm{L}_2(\fR^d)} ds \\ 
&\quad +\int_0^t\left(f(s,\cdot), \varphi \right)_{\mathrm{L}_2(\fR^d)} ds \quad \text{a.e.}~ t \in (0,T).
\end{align*}
This confirms that $u$ is indeed a weak solution to \eqref{main eqn 1}, thereby completing the proof of the primary statement.

It remains only to verify the additional statements regarding the solution's bounds and temporal regularity. 
First, establishing the supremum bound in \eqref{20250826 02} is straightforward; this inequality follows as a direct and immediate consequence of combining \eqref{20240909 82} with the local integrability condition \eqref{20250826 01}.
Next, we establish the absolute continuity of the map $t \mapsto \|u(t,\cdot)\|_{\mathrm{L}_p(\fR^d)}$. 
By utilizing the generalized Minkowski inequality in conjunction with the translation invariance of the $\mathrm{L}_p(\fR^d)$-norm, we can bound the variation of the norm between any two times $0 < s < t < T' < T$:
\begin{align*}
&\left|\|u(t,\cdot)\|_{\mathrm{L}_p(\fR^d)} - \|u(s,\cdot)\|_{\mathrm{L}_p(\fR^d)}\right| \\
&=\left|\left\|\int_0^t  \mathrm{e}^{\int_r^t c(\rho)\mathrm{d}\rho}  \bE\left[f\left(r, \cdot \right)  \right] \mathrm{d}r\right\|_{\mathrm{L}_p(\fR^d)}
-\left\|\int_0^s  \mathrm{e}^{\int_r^s c(\rho)\mathrm{d}\rho}  \bE\left[f\left(r, \cdot \right)  \right] \mathrm{d}r \right\|_{\mathrm{L}_p(\fR^d)} \right|\\
&\leq \left|\left(\mathrm{e}^{\int_0^t c(\rho)\mathrm{d}\rho}  -\mathrm{e}^{\int_0^s c(\rho)\mathrm{d}\rho}\right) \right| \int_0^t  \mathrm{e}^{-\int_0^r c(\rho)\mathrm{d}\rho}  \left\| f\left(r,\cdot\right) \right\|_{\mathrm{L}_p(\fR^d)} \mathrm{d}r \\
&\quad+\mathrm{e}^{\int_0^s |c(\rho)|\mathrm{d}\rho} \int_s^t  \mathrm{e}^{-\int_0^r c(\rho)\mathrm{d}\rho}  \left\| f\left(r,\cdot\right) \right\|_{\mathrm{L}_p(\fR^d)}\mathrm{d}r .
\end{align*}
The upper bound is easily controlled because the integrals of locally integrable functions are absolutely continuous. 
From this, it immediately follows that the temporal mapping $t \mapsto \|u(t,\cdot)\|_{\mathrm{L}_p(\fR^d)}$ is also absolutely continuous on $[0,T']$.
\end{proof}

\begin{remark}
									\label{mollifier remark}
Suppose $u$ is the weak solution established in Theorem \ref{stochastic weak existence}. 
As demonstrated in the theorem's proof, its Sobolev mollification $u^{(\varepsilon)}(t,x)$ acts as a classical solution to equation \eqref{main eqn 1} corresponding to the regularized source term $f^{(\varepsilon)}$.
\end{remark}

\begin{remark}
									\label{Ito uniqueness}
Having established existence, it is natural to address the uniqueness of the solution. 
However, the classical solution $u$ derived in Theorem \ref{classic existence solution} does not possess sufficient regularity for a direct application of It\^o's formula. 
A standard uniqueness proof using It\^o's formula demands that the solution belong to $C^{1,2}([0,T) \times \fR^d)$—requiring continuous differentiability once in time and twice in space. 
Because our solution lacks a guaranteed continuous time derivative $u_t$, a rigorous application of It\^o's formula is precluded. 
To overcome this limitation, the subsequent sections will introduce alternative strategies to prove uniqueness.
\end{remark}

\mysection{Fourier transform method for uniqueness of a weak solution}
												\label{fourier method}

In this section, we establish the uniqueness of the solution by utilizing a classical approach based on the Fourier transform. 
This approach is highly beneficial when the inhomogeneous term is an $\mathrm{L}_p(\mathbf{R}^d)$-valued function for $p \in [1,2]$.
We begin by briefly recalling the foundational definitions and properties of the Fourier transform.
For any function $f \in \mathrm{L}_1(\fR^d)$, its Fourier transform is given by
$$
\cF[f](\xi) = \frac{1}{(2\pi)^{d/2}} \int_{\fR^d} \mathrm{e}^{-i x \cdot \xi} f(x) \mathrm{d}x, \quad \xi \in \fR^d.
$$
Directly from this definition, we obtain the uniform bound
$$
\|\cF[f]\|_{\mathrm{L}_\infty(\fR^d)} \leq \frac{1}{(2\pi)^{d/2}}\|f\|_{\mathrm{L}_1(\fR^d)}.$$
Furthermore, Plancherel's theorem guarantees that
$$
\|\cF[f]\|_{\mathrm{L}_2(\fR^d)} = \|f\|_{\mathrm{L}_2(\fR^d)} \quad \forall f \in \mathrm{L}_1(\fR^d) \cap \mathrm{L}_2(\fR^d).
$$
Due to the completeness of $\mathrm{L}_2(\fR^d)$, this equality allows the Fourier transform to be uniquely extended as an isometry on $\mathrm{L}_2(\fR^d)$.
By interpolating between these $\mathrm{L}_1$ and $\mathrm{L}_2$ estimates using the Riesz–Thorin theorem, the Fourier transform can be extended to a bounded operator from $\mathrm{L}_1(\fR^d) + \mathrm{L}_2(\fR^d)$ into $\mathrm{L}_\infty(\fR^d) + \mathrm{L}_2(\fR^d)$. 
Most notably, for any $p \in [1,2]$, we obtain the Hausdorff–Young inequality
\begin{align}
									\label{Fourier est}
\|\cF[f]\|_{\mathrm{L}_{p/(p-1)}(\fR^d)} \leq \left(\frac{1}{2\pi}\right)^{\frac{d(2-p)}{2p}} \|f\|_{\mathrm{L}_{p}(\fR^d)} 
\quad \forall f \in \mathrm{L}_{p}(\fR^d),
\end{align}
where we adopt the convention $1/0 := \infty$. 
Consequently, for any $p \in [1,2]$ and $f \in \mathrm{L}_p(\fR^d)$, the Fourier transform $\cF[f]$ is well-defined as a function in the conjugate space $\mathrm{L}_{p'}(\fR^d)$ with $p' = \frac{p}{p-1}$. 
In particular, this guarantees that $\cF[f]$ is locally integrable. 
We refer the reader to \cite{Gra 2014} for a comprehensive treatment of these details. 
Relying on these properties in conjunction with a suitable approximation scheme, we arrive at the following identity.

\begin{lemma}
										\label{Plancherel lemma}
Let $p \in [1,2]$, $f \in \mathrm{L}_{p}(\fR^d)$, and $\varphi \in \mathrm{C}_c^\infty(\fR^d)$.
Then
\begin{align*}
\int_{\fR^d} f(x) \varphi(x)\mathrm{d}x
=\int_{\fR^d} \cF[f](\xi) \overline{\cF[\varphi](\xi)}\mathrm{d}x.
\end{align*}
\end{lemma}

We are now in a position to establish a fundamental lemma that guarantees the uniqueness of a solution $u$ to \eqref{main eqn 1} within a specific function class, provided the source term $f$ belongs to $\mathrm{L}_{1,p}((0,T) \times \fR^d)$ for $p \in [1,2]$.

Throughout this section, the coefficients $a^{ij}(t)$, $b^i(t)$, and $c(t)$ are considered merely measurable functions unless otherwise specified.
\begin{lemma}
									\label{pre representation}
Assume $p \in [1,2]$ and let $f \in \mathrm{L}_{1,p}((0,T) \times \fR^d)$. 
Suppose $u$ is a locally integrable function on $[0,T)\times \fR^d$ that serves as a weak solution to \eqref{main eqn 1}. 
We further impose the following three regularity and integrability conditions on $u$:
\begin{enumerate}[(i)]
\item{(Spatial $\mathrm{L}_p$-Bound)} For any $t \in (0,T)$,
\begin{align*}
\|u(t,\cdot)\|_{\mathrm{L}_{p}(\fR^d)} < \infty.
\end{align*}
\item{(Weighted Fourier Integrability)} For any $T' \in (0,T)$, 
\begin{align}
										\label{20240908 20-2}
\int_0^{T'}\mu_{a,b,c}(t)\|\cF[u(t,\cdot)]\|_{\mathrm{L}_{p'}(\fR^d)} \mathrm{d}t 
< \infty,
\end{align}
where $p' = \frac{p}{p-1}$ denotes the H\"older conjugate of $p$ (adopting the convention $1/0 := \infty$).
\item{(Temporal Continuity in Frequency)} For almost every frequency $\xi \in \fR^d$, the temporal mapping $t \mapsto \cF[u(t,\cdot)](\xi)$ is continuous over the interval $[0,T)$.
\end{enumerate}
Under these assumptions, the Fourier transform of the solution satisfies the following integral identity for almost every $\xi \in \fR^d$ and all $t \in [0,T)$,
\begin{align}
										\notag
\cF[u(t,\cdot)](\xi) 
&= \int_0^t \left(-a^{ij}(s)\xi^i\xi^j+ \mathrm{i}b^i(s)\xi^i+c(s) \right)\cF[u(s,\cdot)](\xi) \mathrm{d}s  \\
										\label{20240908 40-2}
&\quad +\int_0^t \cF[f(s,\cdot)](\xi) \mathrm{d}s.
\end{align}

\end{lemma}
\begin{proof}
Leveraging the Fourier estimate \eqref{Fourier est} alongside our assumptions regarding $u$ and $f$, one can readily verify that both sides of \eqref{20240908 40-2} are continuous with respect to $t \in (0,T)$ for almost every $\xi \in \fR^d$. 
Consequently, it suffices to prove that for any fixed $t \in (0,T)$, the following equality holds for almost every $\xi \in \fR^d$,
\begin{align}
										\notag
\cF[u(t,\cdot)](\xi) 
&= \int_0^t \left(-a^{ij}(s)\xi^i\xi^j+ \mathrm{i}b^i(s)\xi^i+c(s) \right)\cF[u(s,\cdot)](\xi) \mathrm{d}s  \\
										\label{20240908 40}
&\quad +\int_0^t \cF[f(s,\cdot)](\xi) \mathrm{d}s.
\end{align}

We organize the remainder of the proof into several distinct parts.

\vspace{2mm}
{\bf I. Transitioning to the Frequency Domain }
\vspace{2mm}

By the standard definition of a weak solution, for any $t \in (0,T)$ and test function $\varphi \in \mathrm{C}_c^\infty(\fR^d)$, the following identity holds:
\begin{align*}
\int_{\fR^d} u(t,x)\varphi(x) \mathrm{d}x
&= \int_0^t \int_{\fR^d} u(s,x) \left(a^{ij}(s)\varphi_{x^ix^j}(x) -b^i(s)\varphi_{x^i}(x)+c(s) \varphi(x)\right) \mathrm{d}x ds \\ 
&\quad +\int_0^t \int_{\fR^d} f(s,x) \varphi(x) \mathrm{d}x ds.
\end{align*}
Applying Lemma \ref{Plancherel lemma} and utilizing the standard properties of the Fourier transform (which convert spatial derivatives into algebraic multipliers), we rewrite this equation in the frequency domain. This yields the integral identity:
\begin{align}
										\notag
&\int_{\fR^d} \cF[u(t,\cdot)](\xi) \overline{\cF[\varphi](\xi)} \mathrm{d}\xi \\
										\notag
&= \int_0^t \int_{\fR^d} \cF[u(s,\cdot)](\xi) \overline{\left(a^{ij}(s)\cF[\varphi_{x^ix^j}](\xi) -b^i(s)\cF[\varphi_{x^i}](\xi)+c(s) \cF[\varphi](\xi)\right)} \mathrm{d}\xi ds \\ 
										\notag
&\quad +\int_0^t \int_{\fR^d} \cF[f(s,\cdot)] \overline{\cF[\varphi](\xi)} \mathrm{d}\xi ds  \\ 
										\notag
&= \int_0^t \int_{\fR^d} \cF[u(s,\cdot)](\xi) \left(\overline{-\left(a^{ij}(s)\xi^i\xi^j\cF[\varphi](\xi) -\mathrm{i} b^i(s) \xi^i \cF[\varphi](\xi)+c(s) \cF[\varphi](\xi)\right)}\right) \mathrm{d}\xi ds \\ 
										\notag
&\quad +\int_0^t \int_{\fR^d} \cF[f(s,\cdot)] \overline{\cF[\varphi](\xi)} \mathrm{d}x ds  \\ 
										\notag
&= \int_0^t \int_{\fR^d}  \left(-a^{ij}(s)\xi^i\xi^j+\mathrm{i}b^i(s)\xi^i+c(s) \right)\cF[u(s,\cdot)](\xi) \overline{\cF[\varphi](\xi)} \mathrm{d}\xi ds \\ 
											\label{20240908 01}
&\quad +\int_0^t \int_{\fR^d} \cF[f(s,\cdot)] \overline{\cF[\varphi](\xi)} \mathrm{d}\xi ds  
\end{align}
for all $t \in (0,T)$ and  $\varphi \in \mathrm{C}_c^\infty(\fR^d)$.

\vspace{2mm}
{\bf II. Constructing the Approximation }
\vspace{2mm}

To extract the pointwise equality, we fix a symmetric test function $\psi \in \mathrm{C}_c^\infty(\fR^d)$ such that $\psi(y) = \psi(-y)$ and normalized to satisfy $\int_{\fR^d} \psi(y) \mathrm{d}y = (2\pi)^{d/2}$.
For any scaling parameter $\varepsilon > 0$, we define the rescaled function
$$
\psi^{(\varepsilon)}(y) = \frac{1}{\varepsilon^d}\psi(y/\varepsilon).
$$

\vspace{2mm}
{\bf III. Applying Fubini and Fourier Inversion Theorems }
\vspace{2mm}

We substitute the specific shifted test function $\phi = \psi^{(\varepsilon)}(x - \cdot)$ into \eqref{20240908 01}. Recalling that $\cF[\psi^{(\varepsilon)}(x-\cdot)](\xi) = \frac{1}{(2\pi)^{d/2}}\int_{\fR^d} \mathrm{e}^{-\mathrm{i} \xi \cdot y} \psi^{(\varepsilon)}(x-y) \mathrm{d}y$, and justifying the swap of integration order via Fubini's theorem (permitted by condition \eqref{20240908 20-2}), we obtain
\begin{align*}
&\int_{\fR^d} \mathrm{e}^{\mathrm{i} x \cdot \xi} \cF[u(t,\cdot)](\xi) \overline{\cF[\psi^{(\varepsilon)}(\xi)} \mathrm{d}\xi  \\
&=\int_{\fR^d} \cF[u(t,\cdot)](\xi) \overline{\cF[\psi^{(\varepsilon)}(x-\cdot)](\xi)} \mathrm{d}\xi  \\
&= \int_0^t \int_{\fR^d}  \left(-a^{ij}(s)\xi^i\xi^j+\mathrm{i}b^i(s)\xi^i+c(s) \right)\cF[u(s,\cdot)](\xi) \overline{\cF[\psi^{(\varepsilon)}(x- \cdot)](\xi)} \mathrm{d}\xi ds \\ 
&\quad +\int_0^t \int_{\fR^d} \cF[f(s,\cdot)] \overline{\cF[\psi^{(\varepsilon)}(x- \cdot)](\xi)} \mathrm{d}\xi ds  \\
&=  \int_{\fR^d}  \mathrm{e}^{\mathrm{i} x \cdot \xi} \int_0^t \left(-a^{ij}(s)\xi^i\xi^j+\mathrm{i}b^i(s)\xi^i+c(s) \right)\cF[u(s,\cdot)](\xi) \overline{\cF[\psi^{(\varepsilon)}](\xi)} ds  \mathrm{d}\xi\\ 
&\quad + \int_{\fR^d} \mathrm{e}^{\mathrm{i} x \cdot \xi} \int_0^t  \cF[f(s,\cdot)] \overline{\cF[\psi^{(\varepsilon)}](\xi)} ds \mathrm{d}\xi.  
\end{align*}
By applying the Fourier inversion theorem, we can strip away the outer integral against $\mathrm{e}^{\mathrm{i} x \cdot \xi}$ to establish a pointwise equality valid for almost every $\xi \in \fR^d$,
\begin{align*}
&\cF[u(t,\cdot)](\xi) \overline{\cF[\psi^{(\varepsilon)}(x-\cdot)](\xi)} \\
&= \int_0^t \left(-a^{ij}(s)\xi^i\xi^j+\mathrm{i}b^i(s)\xi^i+c(s) \right)\cF[u(s,\cdot)](\xi) ds  \overline{\cF[\psi^{(\varepsilon)}](\xi)}\\
&\quad +\int_0^t  \cF[f(s,\cdot)]  ds \overline{\cF[\psi^{(\varepsilon)}](\xi)}.
\end{align*}

\vspace{2mm}
{\bf IV. Passing to the Limit}
\vspace{2mm}

Finally, we analyze the behavior as the scaling parameter vanishes. Observe that $\overline{\cF[\psi^{(\varepsilon)}](\xi)} = \overline{\cF[\psi](\varepsilon\xi)}$. 
Because of our initial normalization, the value at the origin is strictly $\overline{\cF[\psi](0)} = 1$. 
By evaluating along a sequence $\varepsilon = 1/n$ and taking the limit as $n \to \infty$, the multiplier $\overline{\cF[\psi](\varepsilon\xi)}$ converges to 1. 
This limiting process directly yields the desired identity \eqref{20240908 40}.
\end{proof}

\begin{remark}
The pointwise requirement in condition $(i)$ of Lemma \ref{pre representation} could technically be relaxed to hold merely for almost every $t$, such that
$$
\|u(t,\cdot)\|_{\mathrm{L}_{p}(\fR^d)} < \infty.
$$
Nevertheless, because the solutions we examine are inherently continuous in time, this subtle theoretical distinction is of little practical consequence in our context.
\end{remark}

\begin{proposition}[Representation of a weak solution]
										\label{prop repre}

Let the hypotheses of Lemma \ref{pre representation} be satisfied. 
Furthermore, assume that the following initial asymptotic condition holds for almost every frequency $\xi \in \fR^d$ and all $t \in (0,T)$,
\begin{align}
									\label{20240909 99}
\lim_{\varepsilon \downarrow 0}\exp\left(\int_\varepsilon^t\left(-a^{ij}(s)\xi^i\xi^j+c(s) \right)\mathrm{d}s  \right) \cF[u(\varepsilon,\cdot)](\xi)=0.
\end{align}
Under these conditions, the Fourier transform of the solution admits the following explicit integral representation for almost every $\xi \in \fR^d$ and all $t \in [0,T)$,
\begin{align}
										\notag
&  \cF[u(t, \cdot)](\xi) \\
										\label{20240909 19}
&=\int_0^t\exp\left( \int_s^t\left(-a^{ij}(r)\xi^i\xi^j+ib^i(r)\xi^i+c(r) \right)\mathrm{d}r\right) \cF[f(s,\cdot)](\xi) \mathrm{d}s.
\end{align}
\end{proposition}

\begin{proof}

To facilitate clarity, we organize the proof into three parts.

\vspace{2mm}
{\bf I. Establishing the Frequency-Domain ODE }
\vspace{2mm}

By invoking Lemma \ref{pre representation}, we know that for almost every frequency $\xi \in \fR^d$, the following integral equation holds for all $t \in (0,T)$:
\begin{align}
									\notag
\cF[u(t,\cdot)](\xi) 
&= \int_0^t \left(-a^{ij}(s)\xi^i\xi^j+ib^i(s)\xi^i+c(s) \right)\cF[u(s,\cdot)](\xi) \mathrm{d}s  \\
									\label{20240909 90}
&\quad +\int_0^t \cF[f(s,\cdot)](\xi) \mathrm{d}s.
\end{align}
By isolating the integration over the interval $[\varepsilon, t]$, we can rewrite this for any $0 < \varepsilon \leq t < T$ as
$$
\begin{aligned}
\cF[u(t,\cdot)](\xi) 
&= \cF[u(\varepsilon,\cdot)](\xi) +\int_\varepsilon^t \left(-a^{ij}(s)\xi^i\xi^j+\mathrm{i}b^i(s)\xi^i+c(s) \right)\cF[u(s,\cdot)](\xi) \mathrm{d}s \\
&\quad +\int_\varepsilon^t \cF[f(s,\cdot)](\xi) \mathrm{d}s.
\end{aligned}
$$
Now, fix a specific $\xi \in \fR^d$ such that both the limit condition \eqref{20240909 99} and the integral representation above are satisfied. Because $\cF[u(t,\cdot)](\xi)$ is expressed as the integral of locally integrable functions, it is absolutely continuous on $[0,T']$ for any $T' \in (0,T)$. Differentiating with respect to $t$ yields a first-order linear ordinary differential equation valid for almost every $t \in (0,T)$:
$$
\cF[u(t,\cdot)]_t(\xi) 
= \left(-a^{ij}(t)\xi^i\xi^j+\mathrm{i}b^i(t)\xi^i+c(t) \right)\cF[u(t,\cdot)](\xi) + \cF[f(t,\cdot)](\xi).
$$

\vspace{2mm}
{\bf II. Applying the Integrating Factor }
\vspace{2mm}

To solve this differential equation on the interval $[\varepsilon, T']$, we introduce the corresponding integrating factor. Because the exponent is absolutely continuous, the product
$$
\exp\left(- \int_\varepsilon^t\left(-a^{ij}(s)\xi^i\xi^j+\mathrm{i}b^i(s)\xi^i+c(s) \right)\mathrm{d}s \right) \cF[u(t, \cdot)](\xi)
$$
is also absolutely continuous on $[\varepsilon, T']$ and differentiable for almost every $t \in [\varepsilon,T)$. Applying the product rule of differentiation directly yields
$$
\begin{aligned}
&\frac{\mathrm{d}}{\mathrm{d}t} \left(\exp\left(- \int_\varepsilon^t\left(-a^{ij}(s)\xi^i\xi^j+\mathrm{i}b^i(s)\xi^i+c(s) \right)\mathrm{d}s \right) \cF[u(t, \cdot)](\xi) \right) \\
&= \exp\left(- \int_\varepsilon^t\left(-a^{ij}(s)\xi^i\xi^j+\mathrm{i}b^i(s)\xi^i+c(s) \right)\mathrm{d}s\right) \cF[f(t,\cdot)](\xi) \quad \text{a.e.}~ t \in [\varepsilon,T).
\end{aligned}
$$

\vspace{2mm}
{\bf III. Integration and Taking the Limit }
\vspace{2mm}

By the Fundamental Theorem of Calculus, we integrate this derivative from $\varepsilon$ to $t$ and obtain
$$
\begin{aligned}
&\exp\left(- \int_\varepsilon^t\left(-a^{ij}(s)\xi^i\xi^j+\mathrm{i}b^i(s)\xi^i+c(s) \right)\mathrm{d}s \right) \cF[u(t, \cdot)](\xi) \\
&= \cF[u(\varepsilon,\cdot)](\xi) + \int_\varepsilon^t\exp\left(- \int_\varepsilon^s\left(-a^{ij}(r)\xi^i\xi^j+\mathrm{i}b^i(r)\xi^i+c(r) \right)\mathrm{d}r\right) \cF[f(s,\cdot)](\xi) \mathrm{d}s.
\end{aligned}
$$
Multiplying both sides by the inverse integrating factor isolates $\cF[u(t, \cdot)](\xi)$, \textit{i.e.}
$$
\begin{aligned}
\cF[u(t, \cdot)](\xi) 
&= \exp\left(\int_\varepsilon^t\left(-a^{ij}(s)\xi^i\xi^j+\mathrm{i}b^i(s)\xi^i+c(s) \right)\mathrm{d}s \right) \cF[u(\varepsilon,\cdot)](\xi) \\
&\quad +\int_\varepsilon^t\exp\left( \int_s^t\left(-a^{ij}(r)\xi^i\xi^j+\mathrm{i}b^i(r)\xi^i+c(r) \right)\mathrm{d}r\right) \cF[f(s,\cdot)](\xi) \mathrm{d}s.
\end{aligned}
$$
Finally, we take the limit as $\varepsilon \downarrow 0$. Invoking the vanishing initial condition \eqref{20240909 99}, the first term on the right-hand side drops out entirely. This directly yields the target representation \eqref{20240909 19}, completing the proof.
\end{proof}

\begin{remark}
Formal manipulation of \eqref{20240909 19} suggests that $u(t,x)$ can be represented through the inverse Fourier transform:
$$
u(t,x) := \cF_{\xi}^{-1}\left[\int_0^t \exp\left( \int_s^t \left(-a^{ij}(r)\xi^i\xi^j + \mathrm{i}b^i(r)\xi^i + c(r) \right) \mathrm{d}r \right) \cF f(s,\cdot) \, \mathrm{d}s \right](x).
$$
Despite its validity as a distribution, the local integrability of $u$ remains uncertain—especially in the presence of degenerate coefficients $a^{ij}$. 
To ensure we obtain a solution with the required functional properties, we depart from the transform method and adopt a stochastic approach based on It\^o's stochastic calculus to achieve a rigorous construction.
\end{remark}

As a corollary of this integral representation, the uniqueness of the weak solution to \eqref{main eqn 1} is readily established.

\begin{corollary}[Uniqueness of a weak solution]
											\label{Fourier unique}
Let $p \in [1,2]$ and assume $f$ is locally integrable on $[0,T) \times \fR^d$. 
A weak solution $u$ to equation \eqref{main eqn 1} is unique within the class of functions satisfying the following four criteria:
\begin{enumerate}[(i)]
\item{(Spatial Integrability)} For any $t \in (0,T)$,
\begin{align*}
\|u(t,\cdot)\|_{\mathrm{L}_{p}(\fR^d)} < \infty.
\end{align*}
\item{(Weighted Fourier Integrability)} For any $T' \in (0,T)$, 
\begin{align}
										\label{20240908 20}
\int_0^{T'}\mu_{a,b,c}(t)\|\cF[u(t,\cdot)]\|_{\mathrm{L}_{p'}(\fR^d)} \mathrm{d}t 
< \infty,
\end{align}
where $p'$ is the H\"older conjugate of $p$, \textit{i.e.} $p'=\frac{p}{p-1}$ and $\frac{1}{0}:=\infty$.
\item{(Temporal Continuity in Frequency)} For almost every $\xi \in \fR^d$, the mapping $t \mapsto \cF[u(t,\cdot)](\xi)$ is continuous on $[0,T)$.
\item{(Initial Vanishing Condition)} For almost every $\xi \in \fR^d$ and all $t \in (0,T)$,
\begin{align*}
\lim_{\varepsilon \downarrow 0}\exp\left(\int_\varepsilon^t\left(-a^{ij}(s)\xi^i\xi^j+c(s) \right)\mathrm{d}s  \right) \cF[u(\varepsilon,\cdot)](\xi)=0.
\end{align*}
\end{enumerate}
\end{corollary}

\begin{proof}
Suppose $u_1$ and $u_2$ are two weak solutions to \eqref{main eqn 1} that both reside in the previously defined function class. 
By the linearity of the operator, their difference $u := u_1 - u_2$ is a weak solution to the following homogeneous problem
$$
\begin{cases} 
u_t(t,x) = a^{ij}(t)u_{x^ix^j}(t,x) + b^i(t)u_{x^i}(t,x) + c(t)u(t,x), & (t,x) \in (0,T) \times \fR^d \\
u(0,x) = 0.
\end{cases}
$$
Since $u$ also satisfies the class conditions, Proposition \ref{prop repre} implies that for a vanishing source term $f \equiv 0$, the Fourier transform must vanish, \textit{i.e.}
$$
\cF[u(t,\cdot)](\xi) = 0 \quad \text{for almost every } (t,\xi) \in (0,T) \times \fR^d.
$$
Applying the Fourier inversion theorem, we conclude that
$$
u(t,x) = 0 \quad \text{for almost every } (t,x) \in (0,T) \times \fR^d.
$$
Thus, $u_1(t,x) = u_2(t,x)$ in the sense of locally integrable functions on $[0,T) \times \fR^d$, establishing the uniqueness.
\end{proof}

Next, for the case $p=1$, we examine a set of sufficient conditions that satisfy the hypotheses of Corollary \ref{Fourier unique}. 
This allows us to derive new criteria that do not explicitly involve Fourier transforms.

\begin{corollary}
											\label{Fourier unique 2}
Let $f$ be a locally integrable function on $[0,T) \times \fR^d$. 
A weak solution $u$ to equation \eqref{main eqn 1} is unique within the class of functions that satisfy the following four criteria:
\begin{enumerate}[(i)]
\item{(Uniform Local Bound)} For any $T' \in (0,T)$,
\begin{align}
									\label{20240914 60}
\sup_{t \in [0,T']}\|u(t,\cdot)\|_{\mathrm{L}_{1}(\fR^d)} < \infty.
\end{align}
\item{(Weighted Integrability)} For any $T' \in (0,T)$, 
\begin{align*}
\int_0^{T'}\mu_{a,b,c}(t)\|u(t,\cdot)\|_{\mathrm{L}_{1}(\fR^d)} \mathrm{d}t 
< \infty.
\end{align*}
\item{(Pointwise Temporal Continuity)} For almost every $x \in \fR^d$, the mapping $t \in (0,T) \mapsto u(t,x)$ is continuous on $[0,T)$.
\item{(Initial Vanishing Condition)} For all $t \in (0,T)$,
\begin{align*}
\lim_{\varepsilon \downarrow 0}\exp\left(\int_\varepsilon^t\mu_{a,b,c}(s)\mathrm{d}s  \right) \|u(\varepsilon,\cdot)\|_{\mathrm{L}_1(\fR^d)}=0.
\end{align*}
\end{enumerate}

\end{corollary}
\begin{proof}
Let $u$ be a weak solution to \eqref{main eqn 1} satisfying the aforementioned conditions. 
To establish uniqueness, it is sufficient to verify that $u$ fulfills the criteria of Corollary \ref{Fourier unique}, \textit{i.e.}
\begin{enumerate}[(i)]
\item For any $t \in (0,T)$,
\begin{align*}
\|u(t,\cdot)\|_{\mathrm{L}_{1}(\fR^d)} < \infty.
\end{align*}
\item For any $T' \in (0,T)$, 
\begin{align*}
\int_0^{T'}\mu_{a,b,c}(t)\|\cF[u(t,\cdot)]\|_{\mathrm{L}_{\infty}(\fR^d)} \mathrm{d}t 
< \infty,
\end{align*}
\item For almost every $\xi \in \fR^d$, the mapping $t \mapsto \cF[u(t,\cdot)](\xi)$ is continuous on $[0,T)$.
\item For almost every $\xi \in \fR^d$,
\begin{align*}
\lim_{\varepsilon \downarrow 0}\exp\left(\int_\varepsilon^t\left(-a^{ij}(s)\xi^i\xi^j+c(s) \right)\mathrm{d}s  \right) \cF[u(\varepsilon,\cdot)](\xi)=0 \quad \forall t \in (0,T).
\end{align*}
\end{enumerate}
We omit the specific details of this verification, as they follow directly from the Fourier estimate \eqref{Fourier est} and the Lebesgue Dominated Convergence Theorem.
\end{proof}
\begin{remark}
										\label{Fourier conti remark}
We observe that in the case $p=1$, the mapping $\xi \mapsto \cF[u(t,\cdot)](\xi)$ is continuous for all $\xi \in \fR^d$. 
However, this continuity is generally lost when $p \in (1,2]$. 
As a result, for this higher $p$ range, the relaxed conditions derived in Corollary \ref{Fourier unique 2} do not appear sufficient to replace the original assumptions of Corollary \ref{Fourier unique}.
\end{remark}

\begin{remark}
The function class defined in Corollary \ref{Fourier unique 2} for the uniqueness of weak solutions to \eqref{main eqn 1} is not necessarily optimal. 
While Corollary \ref{Fourier unique} demonstrates that this class can be broadened by utilizing the Fourier transform, Corollary \ref{Fourier unique 2} remains valuable because it provides criteria independent of spectral analysis. 
Nevertheless, even without relying on the Fourier transform, further expansion of this class is possible; specifically, the uniform bound in \eqref{20240914 60} is not strictly required. 
Although \eqref{20240914 60} is typically used to ensure the temporal continuity of $t \mapsto \cF[u(t,\cdot)](\xi)$, this requirement can be bypassed by leveraging the inherent continuity of the $\mathrm{L}_p(\fR^d)$-Fourier transform alongside the structural link between $u$ and $f$. 
Following the introduction of a preliminary lemma, we provide a rigorous formulation of this result in Theorem \ref{p1 thm}.
\end{remark}

\begin{lemma}
										\label{Fourier taking lemma}
Let $p \in [1,2]$ and $f \in \mathrm{L}_{1,p,loc}((0,T) \times \fR^d)$, and suppose that Assumptions \ref{main as 1} - \ref{main as 3} hold. 
Furthermore, assume that the source term $f$ satisfies the integrability condition 
\begin{align}
									\label{20240914 99}
\int_0^t\mathrm{e}^{\int_s^t c(r)\mathrm{d}r}\|f(t,\cdot)\|_{\mathrm{L}_p(\fR^d)}\mathrm{d}t < \infty.
\end{align} 
Define the function $u$ via the probabilistic representation:
$$
u(t,x) := \int_0^t \mathrm{e}^{\int_s^t c(r)\mathrm{d}r}\bE\left[f\left(s,x + X_{s,t}\right)  \right] \mathrm{d}s,
$$
where the stochastic process $X_{s,t}$ is given by $X_{s,t}=\int_s^tb(r)\mathrm{d}r+\int_s^t \sigma(r) \mathrm{dB}_r$ with $\sigma(r)= \sqrt{2A(r)}$, $b(t)=(b^1(t),\ldots, b^d(t))$, and $A(r)= (a^{ij}(r))$. Then, for each $t \in (0,T)$, the Fourier transform of $u$ satisfies
\begin{align}
										\notag
&\cF[u(t,\cdot)](\xi)  \\
										\label{20240910 21}
&= \int_0^t \exp\left( \int_s^t\left(-a^{ij}(r)\xi^i\xi^j+ib^i(r)\xi^i+c(r) \right)\mathrm{d}r\right) \cF\left[f\left(s, \cdot \right)\right](\xi)\mathrm{d}s \quad \text{a.e.}~ \xi \in \fR^d.
\end{align}
In the specific case where $p=1$, this identity \eqref{20240910 21} is valid for all $(t,\xi) \in (0,T) \times \fR^d$.

\end{lemma}
\begin{proof}
Although this proof is highly technical, it can be decomposed into a clear, logical sequence. The primary objective is to introduce an approximate solution using spatial truncation, apply the Fourier transform by exploiting the properties of Gaussian processes, and ultimately take the limit as the truncation parameter approaches infinity.

\vspace{2mm}
{\bf I. Defining the Truncated Approximation}
\vspace{2mm}

First, we localize the source term by cutting it off outside a ball of radius $n$. 
For any natural number $n \in \fN$ and variables $(t,x) \in (0,T) \times \fR^d$, we define the truncated function
$$
f^n(t,x) := f(t,x) 1_{|x| < n}
$$
Using this, we define an approximate solution $u^n(t,x)$ via the stochastic representation
$$
u^n(t,x) := \int_0^t \mathrm{e}^{\int_s^t c(r)\mathrm{d}r}\bE\left[f^n\left(s,x + X_{s,t}\right) \right] \mathrm{d}s.
$$

\vspace{2mm}
{\bf II. Justifying the Fourier Transform}
\vspace{2mm}

To apply the Fourier transform to $u^n$, we must ensure it is sufficiently integrable. By applying H\"older's inequality, we can bound the $\mathrm{L}_1$-norm of the truncated function by the $\mathrm{L}_p$-norm of the original function $f$, \textit{i.e.}
$$
\int_0^t\mathrm{e}^{\int_s^t c(r)\mathrm{d}r}\|f^n(t,\cdot)\|_{\mathrm{L}_1(\fR^d)}\mathrm{d}t \lesssim_n \int_0^t\mathrm{e}^{\int_s^t c(r)\mathrm{d}r}\|f(t,\cdot)\|_{\mathrm{L}_p(\fR^d)}\mathrm{d}t < \infty.
$$
Because this integral is finite, Fubini's theorem permits us to apply the Fourier transform $\cF$ directly to both sides of our definition for $u^n$, passing the transform inside the expected value $\bE$, that is,
$$
\cF[u^n(t,\cdot)](\xi) = \int_0^t \mathrm{e}^{\int_s^t c(r)\mathrm{d}r}\bE\left[\cF[f^n\left(s,\cdot + X_{s,t}\right)](\xi) \right] \mathrm{d}s.
$$

\vspace{2mm}
{\bf III. Applying Gaussian Process Properties}
\vspace{2mm}

The shift in the spatial variable by the stochastic process $X_{s,t}$ translates to a phase shift in the Fourier domain. Taking the expectation of this phase shift yields the characteristic function of the Gaussian process:
$$
\bE\left[\cF[f^n\left(s,\cdot + X_{s,t}\right)](\xi) \right] = \cF[f^n(s,\cdot)](\xi) \exp\left( -\int_s^t a^{ij}(r) \mathrm{d}r \xi^i\xi^j +\mathrm{i}\int_s^t b^i(r) \mathrm{d}r\xi^i\right)
$$
Substituting this back into our equation for $\cF[u^n(t,\cdot)](\xi)$, we obtain an exact analytical expression for the Fourier transform of the approximate solution:
\begin{align}
										\label{20260312 01}
\cF[u^n(t,\cdot)](\xi) = \int_0^t \exp\left( \int_s^t\left(-a^{ij}(r)\xi^i\xi^j+\mathrm{i}b^i(r)\xi^i+c(r) \right)\mathrm{d}r\right) \cF\left[f^n\left(s, \cdot \right)\right](\xi)\mathrm{d}s.
\end{align}
This equality holds almost everywhere for $\xi \in \fR^d$.

\vspace{2mm}
{\bf IV. The Limit Procedure ($n \to \infty$)}
\vspace{2mm}

To find the representation for true solution, we must take the limit to both sides of \eqref{20260312 01} as $n \to \infty$. We measure the convergence in the $\mathrm{L}_{p'}$-space using the generalized Minkowski inequality and standard Fourier estimates.
The distance between the approximate and true solutions in the Fourier domain is bounded by the distance between $f^n$ and $f$ in the spatial domain:
$$
\|\cF[u^n(t,\cdot)] -\cF[u(t,\cdot)]\|_{\mathrm{L}_{p'}(\fR^d)} \lesssim \int_0^t \mathrm{e}^{\int_s^t c(r)\mathrm{d}r}\|f^n\left(s,\cdot\right)-f\left(s,\cdot\right)\|_{\mathrm{L}_p(\fR^d)}\mathrm{d}s
$$
When setting up the integral bound for the difference, we rely on the degenerate ellipticity condition ($a^{ij}(r)\xi^i\xi^j \geq 0$). 
This ensures that the exponential term related to the diffusion matrix is bounded by 1, allowing us to drop it from the inequality:
$$
\|\cF[u^n] - \cF[u]\|_{\mathrm{L}_{p'}(\fR^d)} \leq \int_0^t \mathrm{e}^{\int_s^t c(r)\mathrm{d}r} \|\cF[f^n(s,\cdot)] -\cF[f(s,\cdot)]\|_{\mathrm{L}_{p'}(\fR^d)} \mathrm{d}s
$$
Because $f^n \to f$ in $\mathrm{L}_p$, the right-hand side goes to zero as $n \to \infty$. 
Thus, the limiting case satisfies \eqref{20240910 21}.

\vspace{2mm}
{\bf V. Continuity of the Result for the case $p=1$}
\vspace{2mm}

Finally, we establish the continuity of the resulting limiting function with respect to the frequency variable $\xi$. 
This is achieved by demonstrating that $u$ has a bounded $\mathrm{L}_1$-norm. 
By applying the generalized Minkowski inequality alongside the probabilistic representation and \eqref{20240914 99}, we obtain
$$
\|u(t,\cdot)\|_{\mathrm{L}_1(\mathbf{R}^d)} \lesssim \int_0^t \exp\left( \int_s^t c(r) \mathrm{d}r\right) \|f(s, \cdot)\|_{\mathrm{L}_1(\mathbf{R}^d)}\mathrm{d}s < \infty.
$$
The finiteness of this $\mathrm{L}_1$-norm ensures that the Fourier transform is continuous in $\xi$ for any $t \in (0,T)$. 
Consequently, the relation \eqref{20240910 21} is valid for all $(t,\xi) \in (0,T) \times \mathbf{R}^d$.
\end{proof}

For the reader's convenience, we restate Theorem \ref{main thm 3-1} here.
\begin{theorem}[Existence and uniqueness of a weak solution]
														\label{p1 thm}
Let $f$ be a locally integrable function $[0,T) \times \fR^d$.
Suppose that Assumptions \ref{main as 1} - \ref{main as 3} hold.
Additionally, assume the following conditions:
\begin{enumerate}[(i)]
\item For any $0<s<t<T$,
\begin{align*}
\int_s^t \mathrm{e}^{\int_s^r c(\rho)\mathrm{d}\rho}\max_{i,j}|a^{ij}(r)| \mathrm{d}r < \infty.
\end{align*}
\item For any $t \in (0,T)$,
\begin{align*}
\int_0^t \mathrm{e}^{\int_s^t c(r)\mathrm{d}r} \|f(s,\cdot)\|_{\mathrm{L}_1(\fR^d)}\mathrm{d}s < \infty.
\end{align*}
\item For any $T' \in (0,T)$,
\begin{align*}
\int_0^{T'}  \mu_{a,b,c}(t) \int_0^t \mathrm{e}^{\int_s^t c(r)\mathrm{d}r}\|f(s,\cdot)\|_{\mathrm{L}_1(\fR^d)}\mathrm{d}s \mathrm{d}t  <\infty.
\end{align*}
\end{enumerate}
Then there exists a unique weak solution $u$ to \eqref{main eqn 1} in the class satisfying the following conditions:
\begin{enumerate}[(i)]
\item For almost every $t \in (0,T)$,
\begin{align*}
\|u(t,\cdot)\|_{\mathrm{L}_{1}(\fR^d)} < \infty.
\end{align*}
\item For any $T' \in (0,T)$, 
\begin{align}
										\label{20240908 20}
\int_0^{T'}\mu_{a,b,c}(t)\|\cF[u(t,\cdot)]\|_{\mathrm{L}_{\infty}(\fR^d)} \mathrm{d}t 
< \infty.
\end{align}
\item For almost every $\xi \in \fR^d$, the mapping $t \mapsto \cF[u(t,\cdot)](\xi)$ is continuous on $[0,T)$.
\item For almost every $\xi \in \fR^d$,
\begin{align*}
\lim_{\varepsilon \downarrow 0}\exp\left(\int_\varepsilon^t\left(-a^{ij}(s)\xi^i\xi^j+c(s) \right)\mathrm{d}s  \right) \cF[u(\varepsilon,\cdot)](\xi)=0 \quad \forall t \in (0,T).
\end{align*}
\end{enumerate}
Moreover, the solution $u$ satisfies the estimates that for all $t \in (0,T)$,
\begin{align*}
 \|u(t,\cdot)\|_{\mathrm{L}_p(\fR^d)}
\leq \int_0^{t} \mathrm{e}^{\int_s^t c(r)\mathrm{d}r} \|f(s,\cdot)\|_{\mathrm{L}_p(\fR^d)}\mathrm{d}s
\end{align*}
and for all $T' \in (0,T)$,
\begin{align*}
&\int_0^{T'} \mu_{a,b,c}(t) \|u(t,\cdot)\|_{\mathrm{L}_p(\fR^d)} \mathrm{d}t \\
&\leq \int_0^{T'} \mu_{a,b,c}(t) \int_0^t \mathrm{e}^{\int_s^t c(r)\mathrm{d}r} \|f(s,\cdot)\|_{\mathrm{L}_p(\fR^d)}\mathrm{d}s \mathrm{d}t.
\end{align*}
If additionally the coefficient $c(t)$ is locally integrable on $[0,T)$, \textit{i.e.}
\begin{align*}
\int_0^{T'} |c(t)| \mathrm{d}t < \infty \quad \forall T' \in (0,T),
\end{align*}
then the mapping $t \mapsto \|u(t,\cdot)\|_{\mathrm{L}_1(\fR^d)}$ is absolutely continuous and for all $T' \in (0,T)$, 
\begin{align*}
\sup_{t \in [0,T']} \mathrm{e}^{-\int_0^t c(r)\mathrm{d}r} \|u(t,\cdot)\|_{\mathrm{L}_1(\fR^d)}
\leq \int_0^{T'} \mathrm{e}^{-\int_0^s c(r)\mathrm{d}r} \|f(s,\cdot)\|_{\mathrm{L}_1(\fR^d)}\mathrm{d}s.
\end{align*}
\end{theorem}
\begin{proof}

Given the complexity of the proof, we break it down into the following sequential parts.

\vspace{2mm}
{\bf I. Existence and Established Bounds}
\vspace{2mm}

Theorem \ref{stochastic weak existence} provides the existence of a weak solution $u$ to \eqref{main eqn 1}, which can be explicitly represented via the stochastic formula:$$u(t,x) = \int_0^t  \mathrm{e}^{\int_s^t c(r)\mathrm{d}r}  \bE\left[f\left(s,x + X_{s,t}\right)  \right] \mathrm{d}s.$$This solution $u$ obeys the following integral bounds:\begin{align}
										\label{20240914 71}
 \|u(t,\cdot)\|_{\mathrm{L}_p(\fR^d)}
\leq \int_0^{t} \mathrm{e}^{\int_s^t c(r)\mathrm{d}r} \|f(s,\cdot)\|_{\mathrm{L}_p(\fR^d)}\mathrm{d}s
\end{align}
and
\begin{align}
										\notag
&\int_0^{t} \mu_{a,b,c}(s) \|u(s,\cdot)\|_{\mathrm{L}_1(\fR^d)} \mathrm{d}s \\
										\label{20240914 72}
&\leq \int_0^{t} \mu_{a,b,c}(s) \int_0^s \mathrm{e}^{\int_r^s c(\rho)\mathrm{d}\rho} \|f(r,\cdot)\|_{\mathrm{L}_1(\fR^d)}\mathrm{d}r \mathrm{d}s.
\end{align}
Theorem \ref{stochastic weak existence} additionally states that if the coefficient $c(t)$ is locally integrable on $[0,T)$, the map $t \mapsto \|u(t,\cdot)\|_{\mathrm{L}_1(\fR^d)}$ becomes absolutely continuous. This yields the following supremum bound for any $T' \in (0,T)$:
$$
\sup_{t \in [0,T']} \mathrm{e}^{-\int_0^t c(r)\mathrm{d}r} \|u(t,\cdot)\|_{\mathrm{L}_1(\fR^d)}
\leq \int_0^{T'} \mathrm{e}^{-\int_0^s c(r)\mathrm{d}r} \|f(s,\cdot)\|_{\mathrm{L}_1(\fR^d)}\mathrm{d}s.
$$

\vspace{2mm}
{\bf II. Verifying the Uniqueness Class}
\vspace{2mm}

As established in Corollary \ref{Fourier unique}, the uniqueness of the weak solution is guaranteed within the class of functions that satisfy conditions $(i)$ through $(iv)$.
Thus, to conclude the proof, it is sufficient to demonstrate that this specific solution $u$ falls into the required uniqueness class. 
This means verifying that $u$ satisfies all four of the following  properties:
\begin{enumerate}[(i)]
\item For any $t \in (0,T)$,
\begin{align*}
\|u(t,\cdot)\|_{\mathrm{L}_{p}(\fR^d)} < \infty.
\end{align*}
\item For any $T' \in (0,T)$, 
\begin{align*}
\int_0^{T'}\mu_{a,b,c}(t)\|\cF[u(t,\cdot)]\|_{\mathrm{L}_{\infty}(\fR^d)} \mathrm{d}t 
< \infty.
\end{align*}
\item For almost every $\xi \in \fR^d$, the mapping $t \mapsto \cF[u(t,\cdot)](\xi)$ is continuous on $[0,T)$.
\item For almost every $\xi \in \fR^d$,
\begin{align*}
\lim_{\varepsilon \downarrow 0}\exp\left(\int_\varepsilon^t\left(-a^{ij}(s)\xi^i\xi^j+c(s) \right)\mathrm{d}s  \right) \cF[u(\varepsilon,\cdot)](\xi)=0 \quad \forall t \in (0,T).
\end{align*}
\end{enumerate}

\vspace{2mm}
{\bf III.  Proof of the condtions}
\vspace{2mm}

Condition $(i)$ has already been established in \eqref{20240909 82}.
To verify the remaining conditions, we apply Lemma \ref{Fourier taking lemma} to find the explicit Fourier transform of $u$ for all $(t,\xi) \in (0,T) \times \fR^d$:
\begin{align}
										\notag
&\cF[u(t,\cdot)](\xi)  \\
										\label{20240914 70}
&= \int_0^t \exp\left( \int_s^t\left(-a^{ij}(r)\xi^i\xi^j+ib^i(r)\xi^i+c(r) \right)\mathrm{d}r\right) \cF\left[f\left(s, \cdot \right)\right](\xi)\mathrm{d}s.
\end{align}
Because of the degenerate ellipticity condition, the real part of the second-order term is non-positive, meaning that for all $0 < s \leq t < T$, the exponential factor is bounded by 1, that is,
$$
\left|\exp\left( \int_s^t\left(-a^{ij}(r)\xi^i\xi^j+\mathrm{i}b^i(r)\xi^i \right)\mathrm{d}r\right) \right| \leq 1.
$$
This uniform upper bound, when combined with the estimates from \eqref{Fourier est} and \eqref{20240914 72}, immediately confirms conditions $(ii)$ and $(iv)$.
Finally, we address the continuity condition $(iii)$. 
For any $0 < t < T' < T$ and a sequence $t_n \in (t/2, T')$ where $n=1,2,\ldots$, we can bound the integral as follows:
\begin{align*}
&\int_0^{t_n} \exp\left( \int_s^{t_n} c(r) \mathrm{d}r\right) \left|\cF\left[f\left(s, \cdot \right)\right](\xi)\right|\mathrm{d}s \\
&\leq \int_0^{t_n} \exp\left( \int_s^{T'} c(r) \mathrm{d}r  +\int_{t/2}^{T'} |c(r)| \mathrm{d}r\right) \left|\cF\left[f\left(s, \cdot \right)\right](\xi)\right|\mathrm{d}s \\
&\leq \exp\left( \int_{t/2}^{T'} |c(r)| \mathrm{d}r\right) \int_0^{T'} \exp\left( \int_s^{T'} c(r) \mathrm{d}r \right) \left|\cF\left[f\left(s, \cdot \right)\right](\xi)\right|\mathrm{d}s.
\end{align*}
By utilizing Assumption \ref{main as 3}, the bound from \eqref{Fourier est}, the estimate in \eqref{20240914 71}, and applying the Lebesgue dominated convergence theorem, we find that the right-hand side of \eqref{20240914 70} is continuous with respect to $t$ on $(0,T)$. 
This guarantees that the mapping $t \mapsto \cF[u(t,\cdot)](\xi)$ is continuous for any $\xi \in \fR^d$.
This verifies all required conditions, thus completing the proof of the theorem.

\end{proof}

\mysection{Elementary calculus approach for uniqueness of a weak solution}
										\label{elementary section}

In this section, we present an energy estimate tailored to our framework. 
While this estimate is derived using fundamental calculus tools—namely, integration by parts and the fundamental theorem of calculus—the singular nature of the coefficients requires every step to be handled with meticulous care. 
Nevertheless, for $p \in (1,\infty)$, this classical energy estimate proves to be remarkably robust, even in such highly singular settings.

\begin{lemma}[A priori estimate for a weak solution]
										\label{a priori lemma}
Let $p \in (1,\infty)$ and suppose that Assumptions \ref{main as 1}-\ref{main as 3} hold.
Assume that $u$ is a weak solution to \eqref{main eqn 1} with $f$ such that for all $t \in (0,T)$, 
\begin{align}
										\label{20240910 60}
\lim_{\delta \downarrow 0}\mathrm{e}^{\int_\delta^t c(r)\mathrm{d}r}  \|u(\delta,\cdot)\|_{\mathrm{L}_p(\fR^d)} 
=0,
\end{align}
\begin{align}
										\label{20240906 30}
\|u(t,\cdot)\|_{\mathrm{L}_p(\fR^d)} < \infty,
\end{align}
and
\begin{align}
										\label{20240906 30-2}
\int_0^{t}\mu_{a,b,c}(t)\|u(s,\cdot)\|_{\mathrm{L}_{p}(\fR^d)} \mathrm{d}t < \infty.
\end{align}
Then for all $t \in (0,T)$,
\begin{align}
										\label{20240906 60}
\int_{\fR^d}|u(t,x)|^p \mathrm{d}x 
\leq  p\int_0^t \mathrm{e}^{p\int_s^t c(r)\mathrm{d}r}\|u(s,\cdot)\|^{p-1}_{\mathrm{L}_p(\fR^d)}  \|f\left(s,\cdot\right)\|_{\mathrm{L}_p(\fR^d)} \mathrm{d}s.
\end{align}
\end{lemma}
\begin{proof}
Without loss of generality, we may assume
\begin{align}
										\label{20240906 30-3}
\int_0^t \mathrm{e}^{p\int_s^t c(r)\mathrm{d}r}\|u(s,\cdot)\|^{p-1}_{\mathrm{L}_p(\fR^d)}  \|f\left(s,\cdot\right)\|_{\mathrm{L}_p(\fR^d)} \mathrm{d}s < \infty.
\end{align}
To enhance readability, we organize the proof into four logical parts.

\vspace{2mm}
{\bf I. Mollification and Classical Formulation}
\vspace{2mm}

Let $\varphi \in \mathrm{C}_c^\infty(\fR^d)$ be a standard non-negative mollifier such that $\int_{\fR^d} \varphi(x) \mathrm{d}x=1$. 
For any $\varepsilon \in (0,\infty)$ and $(t,x) \in (0,T) \times \fR^d$, we define the spatial convolutions:
$$
u^{(\varepsilon)}(t,x) = \int_{\fR^d} u(t,y) \varphi^{(\varepsilon)}(x-y) \mathrm{d}y
$$
and
$$
f^{(\varepsilon)}(t,x) = \int_{\fR^d} f(t,y) \varphi^{(\varepsilon)}(x-y) \mathrm{d}y.
$$
By the standard properties of convolutions, for any $d$-dimensional multi-index $\alpha$, the spatial derivatives satisfy the following bounds:
\begin{align}
										\label{20240906 32}
\|D^\alpha_xu^{(\varepsilon)}(t,\cdot)\|_{\mathrm{L}_p(\fR^d)}
\lesssim_{\varepsilon,\alpha} \|u(t,\cdot)\|_{\mathrm{L}_p(\fR^d)}
\end{align}
and
\begin{align}
										\label{20240906 33}
\|D^\alpha_xf^{(\varepsilon)}(t,\cdot)\|_{\mathrm{L}_p(\fR^d)}
\lesssim_{\varepsilon,\alpha} \|f(t,\cdot)\|_{\mathrm{L}_p(\fR^d)}.
\end{align}
Because of this smoothing effect, for each $\varepsilon \in (0,\infty)$, $u^{(\varepsilon)}$ becomes a classical solution to the regularized equation
\begin{align}
										\label{20240906 10}
u^{(\varepsilon)}(t,x) = \int_0^t a^{ij}(s) u^{(\varepsilon)}_{x^ix^j}(s,x)   \mathrm{d}s +\int_0^t f^{(\varepsilon)}(s,x)\mathrm{d}s \quad \forall (t,x) \in (0,T) \times \fR^d.
\end{align}
Consequently, for any $T' \in (0,T)$ and $x \in \fR^d$, equation \eqref{20240906 10} yields the pointwise bound
\begin{align}
										\label{20240906 20}
\sup_{t \in [0,T']} |u^{(\varepsilon)}(t,x)| 
\leq \int_0^{T'} |a^{ij}(s) u^{(\varepsilon)}_{x^ix^j}(s,x)|   \mathrm{d}s +\int_0^{T'} |f^{(\varepsilon)}(s,x)|\mathrm{d}s  <\infty.
\end{align}

\vspace{2mm}
{\bf II. Transformation and Application of the Chain Rule}
\vspace{2mm}

To handle the lower-order coefficients, we introduce a transformed function. For $\varepsilon \in (0,\infty)$ and $0 <\delta \leq t < T$, we define:
\begin{align}
										\label{20240910 50}
v^{(\delta,\varepsilon)}(t,x)= \mathrm{e}^{-\int_\delta^t c(s)\mathrm{d}s} u^{(\varepsilon)}\left(t,x-\int_\delta^tb(s)\mathrm{d}s \right).
\end{align}
For each fixed $\delta$ and $x$, the map $t \mapsto v^{(\delta,\varepsilon)}(t,x)$ is absolutely continuous on $[\delta,T']$ for all $T' \in (\delta,T)$. Applying the chain rule, we obtain the derivative with respect to $t$ for almost every $t \in (\delta,T)$:
\begin{align*}
v^{(\delta,\varepsilon)}_t(t,x) 
&= -c(t)v^{(\delta,\varepsilon)}(t,x)  \\
&\quad +\mathrm{e}^{-\int_\delta^t c(s)\mathrm{d}s} \left( u^{(\varepsilon)}_t\left(t,x-\int_\delta^tb(s)\mathrm{d}s \right)
- b^i(t) u^{(\varepsilon)}_{x^i}\left(t,x-\int_\delta^tb(s)\mathrm{d}s \right)  \right) \\
&=a^{ij}(t)v_{x^ix^j}^{(\delta,\varepsilon)}(t,x) +\mathrm{e}^{-\int_\delta^t c(s)\mathrm{d}s} f^{(\varepsilon)}\left(t,x-\int_\delta^tb(s)\mathrm{d}s\right).
\end{align*}
Furthermore, the mapping $t \mapsto |v^{(\delta,\varepsilon)}(t,x)|^p$ is also absolutely continuous on $[\delta,T']$. Applying the chain rule once more yields:
\begin{align*}
\frac{d}{dt}|v^{(\delta,\varepsilon)}(t,x)|^p
&=p|v^{(\delta,\varepsilon)}(t,x)|^{p-2}v^{(\delta,\varepsilon)}(t,x)a^{ij}(t)v_{x^ix^j}^{(\delta,\varepsilon)}(t,x) \\
&\quad +p|v^{(\delta,\varepsilon)}(t,x)|^{p-2}v^{(\delta,\varepsilon)}(t,x)\mathrm{e}^{-\int_\delta^t c(s)\mathrm{d}s} f^{(\varepsilon)}\left(t,x-\int_\delta^tb(s)\mathrm{d}s\right).
\end{align*}

\vspace{2mm}
{\bf III. Integration and the Energy Bound}
\vspace{2mm}

By integrating from $\delta$ to $t$ via the fundamental theorem of calculus, we have for all $(t,x) \in [\delta,T) \times \fR^d$:
\begin{align*}
|v^{(\delta,\varepsilon)}(t,x)|^p
&=|v^{(\delta,\varepsilon)}(\delta,x)|^p + p\int_\delta^t \left( a^{ij}(s) |v^{(\delta,\varepsilon)}(s,x)|^{p-2}v^{(\delta,\varepsilon)}(s,x) v^{(\delta,\varepsilon)}_{x^ix^j}(s,x)\right) \mathrm{d}s \\
&\quad +p\int_\delta^t |v^{(\delta,\varepsilon)}(s,x)|^{p-2}v^{(\delta,\varepsilon)}(s,x)\mathrm{e}^{-\int_\delta^s c(r)\mathrm{d}r} f^{(\varepsilon)}\left(s,x-\int_\delta^sb(r)\mathrm{d}r\right) \mathrm{d}s \\
&= |u^{(\varepsilon)}(\delta,x)|^p + p\int_\delta^t \left( a^{ij}(s) |v^{(\delta,\varepsilon)}(t,x)|^{p-2}v^{(\delta,\varepsilon)}(t,x) v^{(\delta,\varepsilon)}_{x^ix^j}(s,x)\right) \mathrm{d}s \\
&\quad +p\int_\delta^t |v^{(\delta,\varepsilon)}(s,x)|^{p-2}v^{(\delta,\varepsilon)}(s,x)\mathrm{e}^{-\int_\delta^s c(r)\mathrm{d}r} f^{(\varepsilon)}\left(s,x-\int_\delta^sb(r)\mathrm{d}r\right) \mathrm{d}s.
\end{align*}
Next, we integrate with respect to $x$ over $\fR^d$. 
By applying Fubini's theorem, integration by parts, and the degenerate ellipticity of $a^{ij}(t)$, we obtain
\begin{align}
										\notag
&\int_{\fR^d}|v^{(\delta,\varepsilon)}(t,x)|^p\mathrm{d}x \\
										\notag
&\leq \int_{\fR^d}|u^{(\varepsilon)}(\delta,x)|^p \mathrm{d}x \\
										\notag
&\quad -p(p-1) \int_\delta^t \int_{\fR^d}\left( \lambda(s) |v^{(\delta,\varepsilon)}(t,x)|^{p-2} |v^{(\delta,\varepsilon)}_{x}(s,x)|^2\right)  \mathrm{d}x\mathrm{d}s \\
										\notag
&\quad +p\int_\delta^t \int_{\fR^d}|v^{(\delta,\varepsilon)}(s,x)|^{p-2}v^{(\delta,\varepsilon)}(s,x)\mathrm{e}^{-\int_\delta^s c(r)\mathrm{d}r} f^{(\varepsilon)}\left(s,x-\int_\delta^sb(r)\mathrm{d}r\right) \mathrm{d}x\mathrm{d}s \\
										\notag
&\leq  \int_{\fR^d}|u^{(\varepsilon)}(\delta,x)|^p \mathrm{d}x \\
										\label{20240910 51}
&\quad +p\int_\delta^t \int_{\fR^d}|v^{(\delta,\varepsilon)}(s,x)|^{p-2}v^{(\delta,\varepsilon)}(s,x)\mathrm{e}^{-\int_\delta^s c(r)\mathrm{d}r} f^{(\varepsilon)}\left(s,x-\int_\delta^sb(r)\mathrm{d}r\right) \mathrm{d}x\mathrm{d}s. 
\end{align}
The use of Fubini's theorem is justified by equations \eqref{20240906 30}, \eqref{20240906 30-2}, \eqref{20240906 30-3}, \eqref{20240906 32}, and \eqref{20240906 33}, guaranteeing the necessary joint integrability. 
Additionally, integration by parts requires careful handling for $p \in (1,2)$ since the map $x \mapsto |x|^p$ is not strictly twice differentiable at the origin; however, this procedure remains valid, and we refer the reader to \cite[Lemma 2.17]{Krylov 1999} for a rigorous justification.

\vspace{2mm}
{\bf IV. Passing to the Limits}
\vspace{2mm}

Substituting the definition of $v^{(\delta,\varepsilon)}$ from \eqref{20240910 50} back into \eqref{20240910 51} and applying H\"older's inequality yields
\begin{align*}
&\mathrm{e}^{-p\int_\delta^t c(r)\mathrm{d}r} \int_{\fR^d}|u^{(\varepsilon)}(t,x)|^p\mathrm{d}x  \\
&=\int_{\fR^d}|v^{(\delta,\varepsilon)}(t,x)|^p\mathrm{d}x \\
&\leq  \int_{\fR^d}|u^{(\varepsilon)}(\delta,x)|^p \mathrm{d}x \\
&\quad +p\int_\delta^t \mathrm{e}^{-p\int_\delta^s c(r)\mathrm{d}r}\|u^{(\varepsilon)}(s,\cdot)\|^{p-1}_{\mathrm{L}_p(\fR^d)}  \|f^{(\varepsilon)}\left(s,\cdot\right)\|_{\mathrm{L}_p(\fR^d)} \mathrm{d}s.
\end{align*}
Multiplying through by the exponential factor gives
\begin{align*}
&\int_{\fR^d}|u^{(\varepsilon)}(t,x)|^p\mathrm{d}x  \\
&\leq  \mathrm{e}^{p\int_\delta^t c(r)\mathrm{d}r}  \int_{\fR^d}|u^{(\varepsilon)}(\delta,x)|^p \mathrm{d}x \\
&\quad +p\int_\delta^t \mathrm{e}^{p\int_s^t c(r)\mathrm{d}r}\|u^{(\varepsilon)}(s,\cdot)\|^{p-1}_{\mathrm{L}_p(\fR^d)}  \|f^{(\varepsilon)}\left(s,\cdot\right)\|_{\mathrm{L}_p(\fR^d)} \mathrm{d}s.
\end{align*}
Taking the limit as $\delta \downarrow 0$ and utilizing \eqref{20240910 60} alongside \eqref{20240906 32}, the initial value term vanishes, leaving:
\begin{align*}
\int_{\fR^d}|u^{(\varepsilon)}(t,x)|^p\mathrm{d}x\leq  p\int_0^t \mathrm{e}^{p\int_s^t c(r)\mathrm{d}r}\|u^{(\varepsilon)}(s,\cdot)\|^{p-1}_{\mathrm{L}p(\fR^d)}\|f^{(\varepsilon)}\left(s,\cdot\right)\|_{\mathrm{L}_p(\fR^d)} \mathrm{d}s.
\end{align*}
Finally, by taking the limit as $\varepsilon \downarrow 0$ and invoking standard properties of Sobolev mollifiers, we recover the desired energy estimate \eqref{20240906 60}.

\end{proof}

\begin{theorem}[Uniqueness of a weak solution]
									\label{unique weak}
Let $p \in (1,\infty)$ and $f$ be a locally integrable function on $(0,T) \times \fR^d$.
Suppose that Assumptions \ref{main as 1} - \ref{main as 3} hold.
Then a weak solution $u$ to \eqref{main eqn 1} is unique in the class that for all $t \in (0,T)$,
\begin{align*}
\lim_{\delta \downarrow 0}\mathrm{e}^{\int_\delta^t c(r)\mathrm{d}r}  \|u(\delta,\cdot)\|_{\mathrm{L}_p(\fR^d)} 
=0,
\end{align*}
\begin{align*}
\|u(t,\cdot)\|_{\mathrm{L}_p(\fR^d)} < \infty,
\end{align*}
and
\begin{align*}
\int_0^{t}\mu_{a,b,c}(s)\|u(s,\cdot)\|_{\mathrm{L}_{p}(\fR^d)} \mathrm{d}t < \infty.
\end{align*}
\end{theorem}
\begin{proof}
Suppose $u_1$ and $u_2$ are two weak solutions within the specified class. 
Due to the linearity of the equation, their difference $u := u_1 - u_2$ is a solution to the homogeneous equation corresponding to \eqref{main eqn 1} (where $f=0$). 
Because $u$ also belongs to this admissible class, applying Lemma \ref{a priori lemma} yields the estimate
$$
\int_{\fR^d}|u(t,x)|^p \mathrm{d}x \leq 0
$$
for all $t \in (0,T)$. 
This implies that $u(t,x) = 0$ almost everywhere in $\fR^d$ for every $t \in (0,T)$, thereby establishing the uniqueness of the solution.
\end{proof}

We are now in a position to prove Theorem \ref{main thm 3}. 
For the reader's convenience, we first restate the theorem before providing its proof.

\begin{theorem}[Well-posedness of a weak solution]
										\label{other p thm}
Let $p \in (1,\infty)$ and $f$ be a locally integrable function on $[0,T) \times \fR^d$.
Suppose that Assumptions \ref{main as 1} - \ref{main as 3} hold.
Additionally, assume the following conditions:
\begin{enumerate}[(i)]
\item For any $0<s<t<T$,
\begin{align*}
\int_s^t \mathrm{e}^{\int_s^r c(\rho)\mathrm{d}\rho}\max_{i,j}|a^{ij}(r)| \mathrm{d}r < \infty.
\end{align*}
\item For any $t \in (0,T)$,
\begin{align*}
\int_0^t \mathrm{e}^{\int_s^t c(r)\mathrm{d}r} \|f(s,\cdot)\|_{\mathrm{L}_p(\fR^d)}\mathrm{d}s < \infty.
\end{align*}
\item For any $T' \in (0,T)$,
\begin{align*}
\int_0^{T'}  \mu_{a,b,c}(t) \int_0^t \mathrm{e}^{\int_s^t c(r)\mathrm{d}r}\|f(s,\cdot)\|_{\mathrm{L}_p(\fR^d)}\mathrm{d}s \mathrm{d}t  <\infty.
\end{align*}
\end{enumerate}
Then there exists a unique weak solution $u$ to \eqref{main eqn 1} in the class that
for all $t \in (0,T)$,
\begin{align*}
\lim_{\delta \downarrow 0}\mathrm{e}^{\int_\delta^t c(r)\mathrm{d}r}  \|u(\delta,\cdot)\|_{\mathrm{L}_p(\fR^d)} 
=0,
\end{align*}
\begin{align*}
\|u(t,\cdot)\|_{\mathrm{L}_p(\fR^d)} < \infty,
\end{align*}
and
\begin{align*}
\int_0^{t}\mu_{a,b,c}(s)\|u(s,\cdot)\|_{\mathrm{L}_{p}(\fR^d)} \mathrm{d}t < \infty.
\end{align*}
In addition, for any $t \in (0,T)$, the solution $u$ obeys the bounds
\begin{align*}
 \|u(t,\cdot)\|_{\mathrm{L}_p(\fR^d)}
\leq  \int_0^{t} \mathrm{e}^{\int_s^t c(r)\mathrm{d}r} \|f(s,\cdot)\|_{\mathrm{L}_p(\fR^d)}\mathrm{d}s
\end{align*}
and
\begin{align*}
&\int_0^{t} \mu_{a,b,c}(s) \|u(s,\cdot)\|_{\mathrm{L}_p(\fR^d)} \mathrm{d}s \\
&\leq \int_0^{t} \mu_{a,b,c}(s) \int_0^s \mathrm{e}^{\int_r^s c(\rho)\mathrm{d}\rho} \|f(r,\cdot)\|_{\mathrm{L}_p(\fR^d)}\mathrm{d}r \mathrm{d}s.
\end{align*}
If we further assume that the coefficient $c(t)$ is locally integrable on $[0,T)$, that is
\begin{align*}
\int_0^{T'} |c(t)| \mathrm{d}t < \infty \quad \forall T' \in (0,T),
\end{align*}
then for every $T' \in (0,T)$, the temporal mapping $t \mapsto \|u(t,\cdot)\|_{\mathrm{L}_p(\fR^d)}$ is absolutely continuous on $[0,T']$, and we obtain the estimate
\begin{align*}
\sup_{t \in [0,T']} \mathrm{e}^{-\int_0^t c(r)\mathrm{d}r} \|u(t,\cdot)\|_{\mathrm{L}_p(\fR^d)}
\leq \int_0^{T'} \mathrm{e}^{-\int_0^s c(r)\mathrm{d}r} \|f(s,\cdot)\|_{\mathrm{L}_p(\fR^d)}\mathrm{d}s.
\end{align*}
\end{theorem}
\begin{proof}

To enhance readability, we divide the proof into several logical parts.

\vspace{2mm}
{\bf I. Existence and A Priori Estimates}
\vspace{2mm}

Theorem \ref{stochastic weak existence} establishes the existence of a weak solution $u$. Furthermore, this solution can be explicitly represented by the stochastic formula:
\begin{align}
									\label{20240914 90}
u(t,x) = \int_0^t  \mathrm{e}^{\int_s^t c(r)\mathrm{d}r}  \bE\left[f\left(s,x + X_{s,t}\right)  \right] \mathrm{d}s
\end{align}
and satisfies the following estimates for all $t \in (0,T)$:
\begin{align}
										\label{20240914 81}
 \|u(t,\cdot)\|_{\mathrm{L}_p(\fR^d)}
\leq \int_0^{t} \mathrm{e}^{\int_s^t c(r)\mathrm{d}r} \|f(s,\cdot)\|_{\mathrm{L}_p(\fR^d)}\mathrm{d}s
\end{align}
and
\begin{align}
										\notag
&\int_0^{t} \mu_{a,b,c}(s) \|u(s,\cdot)\|_{\mathrm{L}_p(\fR^d)} \mathrm{d}s \\
										\label{20240914 82}
&\leq \int_0^{t} \mu_{a,b,c}(s) \int_0^s \mathrm{e}^{\int_r^s c(\rho)\mathrm{d}\rho} \|f(r,\cdot)\|_{\mathrm{L}_p(\fR^d)}\mathrm{d}r \mathrm{d}s.
\end{align}
Theorem \ref{stochastic weak existence} also guarantees that if $c(t)$ is locally integrable on $[0,T)$, the map $t \mapsto \|u(t,\cdot)\|_{\mathrm{L}_p(\fR^d)}$ is absolutely continuous on $[0,T']$ for all $T' \in (0,T)$. 
Under this additional condition, the following supremum bound holds for any $T' \in (0,T)$:
$$
\sup_{t \in [0,T']} \mathrm{e}^{-\int_0^t c(r)\mathrm{d}r} \|u(t,\cdot)\|_{\mathrm{L}_p(\fR^d)}
\leq \int_0^{T'} \mathrm{e}^{-\int_0^s c(r)\mathrm{d}r} \|f(s,\cdot)\|_{\mathrm{L}_p(\fR^d)}\mathrm{d}s.
$$

\vspace{2mm}
{\bf II. The Uniqueness Class Requirements}
\vspace{2mm}

On the other hand, the uniqueness of a weak solution is already secured by Theorem \ref{unique weak}, provided the solution $u$ is locally integrable on $[0,T) \times \fR^d$ and satisfies three specific conditions for all $t \in (0,T)$:
\begin{align}
										\label{20240914 83}
\lim_{\delta \downarrow 0}\mathrm{e}^{\int_\delta^t c(r)\mathrm{d}r}  \|u(\delta,\cdot)\|_{\mathrm{L}_p(\fR^d)} 
=0,
\end{align}
\begin{align}
										\label{20240914 84}
\|u(t,\cdot)\|_{\mathrm{L}_p(\fR^d)} < \infty,
\end{align}
and
\begin{align}
										\label{20240914 85}
\int_0^{t} \mu_{a,b,c}(s)\|u(t,\cdot)\|_{\mathrm{L}_{p}(\fR^d)} \mathrm{d}t,
< \infty.
\end{align}

\vspace{2mm}
{\bf III. Verifying the Uniqueness Conditions}
\vspace{2mm}

To complete the proof, it is sufficient to demonstrate that the bounds \eqref{20240914 81} and \eqref{20240914 82} imply the uniqueness criteria \eqref{20240914 83}, \eqref{20240914 84}, and \eqref{20240914 85}.
Conditions \eqref{20240914 84} and \eqref{20240914 85} follow immediately from \eqref{20240914 81}, \eqref{20240914 82}, and our initial integrability assumptions on $f$.
To verify the initial trace condition \eqref{20240914 83}, we apply the generalized Minkowski inequality to the representation formula \eqref{20240914 90}, yielding
\begin{align*}
\limsup_{\delta \downarrow 0}\mathrm{e}^{\int_\delta^t c(r)\mathrm{d}r}  \|u(\delta,\cdot)\|_{\mathrm{L}_p(\fR^d)}  
&\leq \limsup_{\delta \downarrow 0} \mathrm{e}^{\int_\delta^t c(r)\mathrm{d}r}  \int_0^\delta  \mathrm{e}^{\int_s^\delta c(r)\mathrm{d}r}  \|f(s,\cdot)\|_{\mathrm{L}_p(\fR^d)}  \mathrm{d}s \\
&\leq  \limsup_{\delta \downarrow 0} \int_0^\delta  \mathrm{e}^{\int_s^t c(r)\mathrm{d}r}  \|f(s,\cdot)\|_{\mathrm{L}_p(\fR^d)}  \mathrm{d}s.
\end{align*}
By our standing assumption, the full integral over $(0,t)$ is finite:
\begin{align*}
\int_0^t \mathrm{e}^{\int_s^t c(r)\mathrm{d}r} \|f(s,\cdot)\|_{\mathrm{L}_p(\fR^d)}\mathrm{d}s < \infty.
\end{align*}
Therefore, the integral from $0$ to $\delta$ must shrink to zero as $\delta \downarrow 0$. This forces the limit to vanish, confirming that \eqref{20240914 83} holds. 
This establishes the uniqueness of the solution and concludes the proof of the theorem.

\end{proof}

\mysection{Proof of main theorems and corollaries}
									\label{pf main thm}

\begin{proof}[Proof of Theorem \ref{main thm 4}]

The proof is organized into the following three distinct parts:

\vspace{2mm}
{\bf I. Establishing Initial Bounds}
\vspace{2mm}

First, we demonstrate that the following three integral bounds are finite:

\begin{enumerate}[(i)]
\item For any $0<s<t<T$,
\begin{align*}
\int_s^t \mathrm{e}^{\int_s^r c(\rho)\mathrm{d}\rho}\max_{i,j}|a^{ij}(r)| \mathrm{d}r < \infty.
\end{align*}
\item For any $t \in (0,T)$,
\begin{align*}
\int_0^t \mathrm{e}^{\int_s^t c(r)\mathrm{d}r} \|f(s,\cdot)\|_{\mathrm{L}_p(\fR^d)}\mathrm{d}s < \infty.
\end{align*}
\item For any $T' \in (0,T)$,
\begin{align*}
\int_0^{T'}  \mu_{a,b,c}(t) \int_0^t \mathrm{e}^{\int_s^t c(r)\mathrm{d}r}\|f(s,\cdot)\|_{\mathrm{L}_p(\fR^d)}\mathrm{d}s \mathrm{d}t  <\infty.
\end{align*}
\end{enumerate}
These results are straightforward consequences of our baseline assumptions regarding the coefficients and the function $f$. 
The local integrability of $c(t)$ on the interval $[0,T)$ is the primary factor that guarantees these bounds.

\vspace{2mm}
{\bf II. Existence of the Solution}
\vspace{2mm}

With these bounds confirmed, we can apply Theorems \ref{main thm 3-1} and \ref{main thm 3} to establish that a unique solution $u$ to equation \eqref{main eqn 1} exists and satisfies the necessary a priori estimates within specific function classes. 
It directly follows from these theorems that the a priori estimates are maintained.

Furthermore, due to the estimates provided in these theorems, the local integrability of $c$, and Remark \ref{absolute remark}, this solution $u$ is guaranteed to belong to the intersection of the following three function spaces: $\mathrm{AC}_{0,loc}\left([0,T) ; \mathrm{L}_p(\fR^d)  \right)$, $\mathrm{L}_{1,p,loc}\left( (0,T) \times \fR^d, \mu_{a,b,c}(t)\mathrm{d}t\right)$, and 
$$
\mathrm{L}_{\infty,p,loc}\left( (0,T) \times \fR^d,  \mathrm{e}^{-\int_0^t c(s)\mathrm{d}s}\mathrm{d}t\right).
$$

\vspace{2mm}
{\bf III. Verifying the Uniqueness Class}
\vspace{2mm}

Because we aim to restrict our solution to these specific spaces, Theorems \ref{main thm 3-1} and \ref{main thm 3} are insufficient on their own to prove uniqueness. 
Instead, we rely on Theorems \ref{Fourier unique} and \ref{unique weak}, which are specifically focused on establishing uniqueness classes depending on the parameter $p$.

To conclude the proof, we need to demonstrate that any function located in the intersection of the three previously mentioned spaces also satisfies these specific uniqueness criteria. We examine this in two cases:
\begin{itemize}
\item{\bf Case 1: $p \in (1,\infty)$} 

The verification that $u$ falls into the uniqueness class for this range is trivial, so the detailed steps are omitted.

\item{\bf Case 2: $p = 1$}

For $p=1$, we must prove that $u$ satisfies the following four conditions:
\begin{enumerate}[(i)]
\item For any $t \in (0,T)$,
\begin{align*}
\|u(t,\cdot)\|_{\mathrm{L}_{1}(\fR^d)} < \infty.
\end{align*}
\item For any $T' \in (0,T)$, 
\begin{align*}
\int_0^{T'} \mu_{a,b,c}(t)\|\cF[u(t,\cdot)]\|_{\mathrm{L}_{\infty}(\fR^d)} \mathrm{d}t,
< \infty.
\end{align*}
\item For almost every $\xi \in \fR^d$, the mapping $t \mapsto \cF[u(t,\cdot)](\xi)$ is continuous on $[0,T)$.
\item For almost every $\xi \in \fR^d$,
\begin{align*}
\lim_{\varepsilon \downarrow 0}\exp\left(\int_\varepsilon^t\left(-a^{ij}(s)\xi^i\xi^j+c(s) \right)\mathrm{d}s  \right) \cF[u(\varepsilon,\cdot)](\xi)=0 \quad \forall t \in (0,T).
\end{align*}
\end{enumerate}
We can verify these four conditions as follows:
\begin{itemize}
\item Condition (i) is satisfied immediately.
\item Condition (ii) follows easily from the estimate established in \eqref{Fourier est} and the assumed properties of $u$.
\item Condition (iv) is proven by applying the degenerate ellipticity condition. This allows us to bound the complex exponential expression purely by the integral of the zero-order term $c(s)$, keeping it finite:
\begin{align*}
\left|\exp\left(\int_\varepsilon^t\left(-a^{ij}(s)\xi^i\xi^j+ib^i(s)\xi^i+c(s) \right)\mathrm{d}s \right)\right|
&\leq \left|\exp\left(\int_\varepsilon^t\left(c(s) \right)\mathrm{d}s \right) \right| \\
&\leq \exp\left(\int_0^t|c(s)|\mathrm{d}s  \right) 
\end{align*}
\item Finally, condition (iii) is a direct consequence of the continuity of the mapping $t \mapsto \|u(t,\cdot)\|_{\mathrm{L}_1(\fR^d)}$ on the interval $(0,T)$.
\end{itemize}
\end{itemize}
The theorem is proved. 
\end{proof}

Finally, we give the proof of Theorem \ref{main thm 5}.
\begin{proof}[Proof of Theorem \ref{main thm 5}]

To improve readability and logical flow, we organize the proof into four distinct parts.

\vspace{2mm}
{\bf I. Deriving the Relationship Between Estimates}
\vspace{2mm}

Theorem \ref{main thm 4} establishes the foundational estimates \eqref{main est 1-1} and \eqref{main est 1-2}, along with the existence and uniqueness of a weak solution $u$ in a wider function class. 
We assert that \eqref{main est 1-1} is inherently stronger than \eqref{main est 2-1}.
By introducing a small parameter $\delta > 0$ and applying both a change of variables and H\"older's inequality to \eqref{main est 1-1}, we can construct an upper bound dependent on $\delta$. Specifically, for any $\delta>0$ and $T' \in (0,T)$, the following sequence of inequalities holds:
\begin{align*}
&\int_0^{T'}  \left(\int_{\fR^d}|u(t,x)|^p\mathrm{d}x \right)^{q/p}\mathrm{e}^{-q\int_0^tc(s)\mathrm{d}s}  w(\alpha(t)+\delta t) \left(\lambda(t)+\delta\right)\mathrm{d}t \\
&\leq \sup_{t \in [0,T']}  \|u(t,\cdot)\|_{\mathrm{L}_p(\fR^d)} \mathrm{e}^{-q\int_0^tc(s)\mathrm{d}s}
\int_0^{\alpha(T')+\delta T'}  w(t) \mathrm{d}t \\
&\leq  \|f\|^q_{\mathrm{L}_{1,p}(\left(0,T') \times \fR^d, \mathrm{e}^{-\int_0^tc(s)\mathrm{d}s} \mathrm{d}t\right)}
\int_0^{\alpha(T')+\delta T'}  w(t) \mathrm{d}t\\
&\leq  \int_0^{\alpha(T')+\delta T'}  w(t) \mathrm{d}t 
\cdot \left(\int_0^{T'}  \left(w(\alpha(t)+\delta t)\right)^{-\frac{1}{q-1}} \left(\lambda(t)+\delta\right)\mathrm{d}t\right)^{q-1}  \\
&\quad \times \|f\|^q_{\mathrm{L}_{q,p}(\left(0,T') \times \fR^d, \mathrm{e}^{-\int_0^tc(s)\mathrm{d}s} w(\alpha(t)+\delta t) (\lambda(t)+\delta)^{1-q} \mathrm{d}t\right)}
\\
&\leq [w]_{\mathrm{A}_p(\fR)}\left(\alpha(T') +\delta T'\right)^q \\
&\quad \times \int_0^{T'}  \|f(t,\cdot)\|_{\mathrm{L}_p(\fR^d)}^q\mathrm{e}^{-q\int_0^tc(s)\mathrm{d}s}  w(\alpha(t)+\delta t) (\lambda(t)+\delta)^{1-q}\mathrm{d}t.
\end{align*}
By taking the limit as $\delta \downarrow 0$, we immediately recover estimate \eqref{main est 2-1}.

\vspace{2mm}
{\bf II. Sobolev Derivatives and Mollification}
\vspace{2mm}

If we temporarily assume that \eqref{main est 2} holds, we can immediately derive \eqref{main est 2-2} by combining \eqref{main est 2-1} with standard Sobolev interpolation inequalities, specifically bounding the first derivative by the function and its second derivative.
Furthermore, \eqref{main est 2} guarantees the existence of the spatial derivatives $u_x$ and $u_{xx}$ on the domain where the spectral lower bound satisfies $\lambda(t) > 0$. 
We justify this rigorously using Sobolev mollifiers. By convolving the functions with a smooth approximation of identity $\varphi^{(\varepsilon)}$ and taking a sequence $\varepsilon_n \downarrow 0$, we generate Cauchy sequences in the weighted spaces. 
This guarantees that the approximations converge in $\mathrm{L}_p(\fR^d)$ almost everywhere with respect to the weighted measure $\mu_\lambda(dt) = w(\alpha(t)) \lambda(t)\mathrm{d}t$. 
Consequently, the true Sobolev derivatives exist on that set that the weight is non-degenerate.
The remainder of the proof is thus dedicated entirely to establishing \eqref{main est 2}.

\vspace{2mm}
{\bf III. Proving Estimate \eqref{main est 2}}
\vspace{2mm}

We split the proof of \eqref{main est 2} into two scenarios based on the integrability of the system's coefficients
\begin{itemize}
\item {\bf Step 1:} (The Locally Integrable Case)

If we assume that both the diffusion matrix $a^{ij}$ and the drift vector $b^i$ are locally integrable on the interval $[0,T)$, the required estimate is already proven and can be directly cited from \cite[Theorem 2.5]{KID 2024}.

\item{\bf Step 2:} (The General Case)

When the coefficients lack local integrability near the initial time, an approximation method is required. 
While we can safely truncate the drift vector point-wise at a threshold $M>0$ using $b_M^i(t) = \max\{-M, \min\{b^i(t), M\}\}$, applying this naive component-wise truncation to $a^{ij}(t)$ could destroy the system's degenerate ellipticity.
To circumvent this, we rely on the spectral decomposition $A(t) = Q(t)\Lambda(t)Q(t)^\top$. 
Assuming the eigenvalues satisfy $\lambda_1(t) \leq \cdots \leq \lambda_d(t)$ and the diagonal matrix $\Lambda(t)$ is given by
\begin{align*}
\Lambda(t)=
\begin{bmatrix}
\lambda_1(t) & 0   & 0   & \cdots & 0   \\
0   &\lambda_2(t)  & 0   & \cdots & 0   \\
0   & 0   &\lambda_3(t) & \cdots & 0   \\
\vdots & \vdots & \vdots & \ddots & \vdots \\
0   & 0   & 0   & \cdots & \lambda_d(t)
\end{bmatrix}.
\end{align*}

\begin{itemize}
\item{Measurability:} The mapping to the eigenvalue matrix $\Lambda(t)$ is naturally measurable. 
Furthermore, the Kuratowski–Ryll-Nardzewski selection theorem guarantees that the mapping to the orthogonal eigenvector matrix $Q(t)$ is also measurable.
\item{Spectral Truncation:} We construct a truncated eigenvalue matrix $\Lambda_M(t)$ by limiting each eigenvalue to a maximum of $M>0$:
\begin{align*}
\Lambda_M(t)=
\begin{bmatrix}
\lambda_1(t)\wedge M & 0   & 0   & \cdots & 0   \\
0   &\lambda_2(t) \wedge M  & 0   & \cdots & 0   \\
0   & 0   &\lambda_3(t) \wedge M & \cdots & 0   \\
\vdots & \vdots & \vdots & \ddots & \vdots \\
0   & 0   & 0   & \cdots & \lambda_d(t) \wedge M.
\end{bmatrix}
\end{align*}
We then reconstruct the approximated diffusion matrix as $A_M(t) = Q(t)\Lambda_M(t)Q(t)^\top$.
This method ensures that $A_M(t)$ remains locally integrable while strictly preserving the necessary ellipticity conditions.
\end{itemize}
\end{itemize}

\vspace{2mm}
{\bf IV. Weak Compactness and Passing to the Limit}
\vspace{2mm}

When we apply {\bf Step 1} to the truncated coefficients $a_M^{ij}$ and $b_M^i$, the necessary estimates continue to hold. 
A critical point is that the constants in these estimates rely strictly on our initial assumptions, maintaining complete independence from the truncation parameter $M$. 
For the sequence of solutions $\{u^M\}$, this uniform bound guarantees both relative weak compactness and the Cauchy property within the appropriate weighted spaces. 
Ordinarily, extracting a weakly convergent Cauchy subsequence would be enough to finish the proof using weighted norm approximation. 
However, the original source term $f$ lacks the spatial regularity required to rigorously take the limit. 
We overcome this by executing the full truncation procedure on an auxiliary system driven by a mollified source term $f^{(\varepsilon)}$. 
By applying analogous relative weak compactness and Cauchy criteria to the smoothed solutions $u^{(\varepsilon)}$, we can rigorously justify the limit passage and finalize the proof.
\end{proof}

\begin{proof}[Proof of Theorem \ref{main thm 6}]

For the sake of clarity, we have structured the proof into several separate stages, as is typical.

\vspace{2mm}
{\bf I. The Objective}
\vspace{2mm}

Building upon Theorem \ref{main thm 5}, the remainder of the proof reduces to verifying \eqref{20240914 20} and confirming that $u$ is a strong solution to \eqref{main eqn 1}. 
Once $u$ is established as a strong solution, \eqref{20240914 20} follows immediately from its definition and the a priori estimates provided in Theorem \ref{main thm 5}. 
Therefore, our primary objective is now to demonstrate that $u$ is a strong solution by utilizing a mollified approximation.

\vspace{2mm}
{\bf II. Applying Sobolev Mollifiers}
\vspace{2mm}

We return to the Sobolev mollifiers $u^{(\varepsilon)}$ and $f^{(\varepsilon)}$ that were utilized in the proof of Theorem \ref{main thm 5}. 
As noted in Remark \ref{mollifier remark}, these regularized functions satisfy the following integral representation :
$$
\begin{aligned}
u^{(\varepsilon)}(t,x) = \int_0^t \left(a^{ij}(s) u^{(\varepsilon)}_{x^ix^j}(s,x) + b^i(s)u^{(\varepsilon)}_{x^i}(s,x) + c(s)u^{(\varepsilon)}(s,x) +f^{(\varepsilon)}(s,x)\right) \mathrm{d}s. 
\end{aligned}
$$

\vspace{2mm}
{\bf III. Convergence Analysis}
\vspace{2mm}

Let $\varepsilon_n$ be a sequence of positive numbers such that $\lim_{n \to \infty} \varepsilon_n =0$. 
We must analyze the convergence of each term in the integral equation as $n \to \infty$:
\begin{itemize}
\item{\bf Convergence of $u^{(\varepsilon)}$:} 
By extracting an appropriate subsequence based on \eqref{20240914 31}, the left-hand side term $u^{(\varepsilon)}(t,x)$ converges to $u(t,x)$ for all $t$ and for almost every $x \in \fR^d$. This is guaranteed by the following bound:
$$  \begin{aligned}
  \sup_{t \in [0,T']} \left(\mathrm{e}^{-\int_0^tc(s)\mathrm{d}s}\|(u^{(\varepsilon_n)} - u)(t,\cdot)\|_{\mathrm{L}_p(\fR^d)} \right)
  \leq \|f^{(\varepsilon_n)} - f \|_{\mathrm{L}_{1,p}\left((0,T') \times \fR^d, \mathrm{e}^{-\int_0^tc(s)\mathrm{d}s} \mathrm{d}t\right)}.
  \end{aligned}
  $$
\item{\bf Convergence of the Source Term:}
  The integral of the source term, 
$$
\int_0^t f^{(\varepsilon)}(s,x) \mathrm{d}s,
$$
converges straightforwardly due to our baseline assumption that 
$$
f \in \mathrm{L}_{1,p,loc}\left( (0,T) \times \fR^d\right).
$$
  
\item{\bf Convergence of the Differential Operators:}
We can bound the remaining integral, which involves the coefficients and spatial derivatives, by invoking the generalized Minkowski inequality, H\"older's inequality, and the a priori estimates established in Theorem \ref{main thm 5}. 
This approach leads to the following upper bound:
\begin{align*}
&\left\|\int_0^t \left(a^{ij}(s) u^{(\varepsilon_n)}_{x^ix^j}(s,\cdot) + b^i(s)u^{(\varepsilon_n)}_{x^i}(s,\cdot) + c(s)u^{(\varepsilon_n)}(s,\cdot) \right)\mathrm{d}s\right\|_{\mathrm{L}_p(\fR^d)} \\
&\leq \int_0^t \left(|a^{ij}(s)| \| u^{(\varepsilon_n)}_{x^ix^j}(s,\cdot)\|_{\mathrm{L}_p(\fR^d)} + |b^i(s)|\| u^{(\varepsilon_n)}_{x^i}(s,\cdot)\|_{\mathrm{L}_p(\fR^d)} \right)\mathrm{d}s \\
&\quad + \int_0^t |c(s)|\|u^{(\varepsilon_n)}(s,\cdot)\|_{\mathrm{L}_p(\fR^d)} \mathrm{d}s \\
&\lesssim \int_0^{T'} \mu_{a,b,c}(t)^{\frac{q}{q-1}} \left(w(\alpha(t)) \lambda(t) \right)^{-1/(q-1)} \mathrm{d}t  \\
&\times \int_0^{T'}  \left(\int_{\fR^d}|f^{(\varepsilon_n)}(t,x)|^p\mathrm{d}x \right)^{q/p}\mathrm{e}^{-q\int_0^tc(s)\mathrm{d}s}  w(\alpha(t)) (\lambda(t))^{1-q}\mathrm{d}t.
\end{align*}
\end{itemize}

\vspace{2mm}  
{\bf IV. Concluding the Strong Solution}
\vspace{2mm}

To finalize the proof, we apply the above convergence logic to the Cauchy sequences. By evaluating the differences $u^{(\varepsilon_n)} - u^{(\varepsilon_m)}$ and $f^{(\varepsilon_n)} - f^{(\varepsilon_m)}$ (rather than just $u^{(\varepsilon_n)}$ and $f^{(\varepsilon_n)}$) for any $n,m \in \fN$, we can rigorously pass to the limit. This confirms that $u$ is indeed a strong solution to \eqref{main eqn 1} corresponding to the source term $f$.

\end{proof}

\mysection{Discussions and Open problems}

While Corollary \ref{Fourier unique} and Theorem \ref{unique weak} successfully establish the uniqueness of weak solutions to \eqref{main eqn 1} within classes of $\mathrm{L}_p(\fR^d)$-valued functions for $p \in [1,\infty)$, a natural subsequent question is whether this uniqueness extends to bounded functions.
\begin{open}
Is a weak solution to \eqref{main eqn 1} unique within a class of $\mathrm{L}_\infty(\fR^d)$-valued functions?
\end{open}
It\^o calculus appears essential for establishing well-posedness in the $\mathrm{L}_\infty(\fR^d)$-setting. 
For instance, the existence of such a weak solution is constructed using It\^o's formula, as shown in the proof of Theorem \ref{stochastic weak existence}. 
However, extending this stochastic approach to prove uniqueness is hindered by the low regularity of the coefficients $a^{ij}(t)$ and $b^i(t)$, as highlighted in Remark \ref{Ito uniqueness}. 
Furthermore, we cannot rely on the classical maximum principle for uniqueness; the irregularity of the coefficients prevents the mollified approximations $u^{(\varepsilon)}(t,x)$ from residing in the classical space $C^{1,2}([0,T) \times \fR^d)$.

Our next inquiry focuses on the behavior of the solution near the initial time when the zero-order coefficient is highly singular. 
Specifically, we want to know if regularity estimates analogous to \eqref{main est 2} in Theorem \ref{main thm 5} can be derived even when $c(t)$ is not locally integrable near $t=0$ (i.e., when $\int_0^{T'} |c(t)| \, \mathrm{d}t = \infty$ for some $T' \in (0,T)$).
\begin{open}
Can we establish meaningful regularity estimates for a weak solution to \eqref{main eqn 1} when $c(t)$ fails to be locally integrable near the initial time?
\end{open}
Current estimates, such as \eqref{main est 2}, depend heavily on the exponential integrating factor $\mathrm{e}^{-\int_0^t c(s) \mathrm{d}s}$. 
Consequently, the failure of local integrability invalidates this approach, dictating the need for a fundamentally different class of regularity bounds. 
At present, we have no preliminary hypotheses regarding what structural form these alternative estimates might take.

\begin{open}
Can these theoretical frameworks be extended to broader classes of operators?
\end{open}

A logical next step is exploring whether our combined Itô calculus and weighted Sobolev methodology applies to degenerate evolutionary systems or non-local operators, such as time-dependent fractional Laplacians or pseudo-differential operators with severe temporal singularities. 
Proving maximal $L_p$-regularity under $t=0$ temporal blow-up for these cases would substantially broaden the applicability of non-local diffusion theory.

\bibliographystyle{plain}

\vspace{3mm}
{\bf On behalf of all authors, the corresponding author states that there is no conflict of interest.}
\vspace{3mm}

{\bf Data availability statements:} No datasets were generated or analysed during the current study.
\vspace{3mm}
\end{document}